\documentclass[11pt]{article}
\usepackage{url}
\usepackage{amsfonts}
\usepackage{amssymb}
\usepackage{amsmath}
\usepackage{amsthm}
\usepackage[lofdepth,lotdepth]{subfig}
\usepackage{setspace,amssymb,verbatim}
\usepackage[margin=1in]{geometry}
\usepackage{graphicx}
\usepackage{rotating}
\usepackage{psfrag}
\usepackage{afterpage}
\usepackage{multirow}
\usepackage[lofdepth,lotdepth]{subfig}
\usepackage{tikz}

\theoremstyle{plain}
\newtheorem{theorem}{Theorem}[section]
\newtheorem{proposition}[theorem]{Proposition}
\newtheorem*{theorem*}{Theorem}

\newtheorem*{lemma*}{Lemma}
\newtheorem{lemma}[theorem]{Lemma}

\theoremstyle{remark}

\newtheorem*{case*}{Case}

\theoremstyle{definition}

\usepackage{prettyref}
\newrefformat{sec}{Section \ref{#1}}
\newrefformat{chap}{Chapter \ref{#1}}
\newrefformat{prop}{Proposition \ref{#1}}
\newrefformat{prob}{Problem \ref{#1}}
\newrefformat{rem}{Remark \ref{#1}}
\newrefformat{cor}{Corollary \ref{#1}}
\newrefformat{def}{Definition \ref{#1}}
\newrefformat{ex}{Example \ref{#1}}
\newrefformat{lem}{Lemma \ref{#1}}
\newrefformat{eq}{Equation \eqref{#1}}
\newrefformat{con}{condition (\ref{#1})}
\newrefformat{conj}{Conjecture \ref{#1}}
\newrefformat{fig}{Figure \ref{#1}}
\newrefformat{tab}{Table \ref{#1}}
\newrefformat{app}{Appendix \ref{#1}}

\newcommand{\tw}[1]{{}^#1\!}
\newcommand{\stab}{\mathrm{stab}}

\newcommand{\Z}{\mathbb{Z}}

\newcommand{\aut}{\mathrm{Aut}}
\newcommand{\out}{\mathrm{Out}}
\newcommand{\C}{\mathbb{C}}

\newcommand{\F}{\mathbb{F}}
\newcommand{\E}{\mathcal{E}}

\newcommand{\syl}{\mathrm{Syl}}
\newcommand{\Syl}{\mathrm{Syl}}

\newcommand{\ind}{\mathrm{Ind}}
\newcommand{\irr}{\mathrm{Irr}}
\newcommand{\ibr}{\mathrm{IBr}}

\newcommand{\wh}[1]{\widehat{#1}}

\newcommand{\diag}{\mathrm{diag}}
\newcommand{\dz}{\mathrm{dz}}

\title{$Sp_6(2^a)$ is ``Good" for the McKay, Alperin Weight, and Related Local-Global Conjectures}

\author{Amanda A. Schaeffer Fry\\\small\textit{Department of Mathematics, University of Arizona, Tucson, AZ 85721, USA}}
\date{}
\begin{document}
\maketitle

\begin{abstract}
The so-called ``local-global" conjectures in the representation theory of finite groups relate the representation theory of $G$ to that of certain proper subgroups, such as the normalizers of particular $p$-groups.  Recent results by several authors reduce some of these conjectures to showing that a certain collection of stronger conditions holds for all finite simple groups. Here, we show that $G=Sp_6(2^a)$ is ``good" for these reductions for the McKay conjecture, the Alperin weight conjecture, and their blockwise versions.

\textit{Keywords:} Cross characteristic representations, Local-global conjectures, Finite classical groups, Alperin weight conjecture, McKay conjecture, Alperin-McKay conjecture
\end{abstract}

\section{Introduction: The Conjectures and Reductions}
Much of the representation theory of finite groups is dedicated to showing the validity of various conjectures which relate certain invariants of a finite group with those of certain subgroups.  Often, these have to do with the number of characters of the group of a given type.  One of the first of these ``local-global" conjectures is the McKay conjecture, which says that if $G$ is a finite group with $\ell$ a prime and $P\in\mathrm{Syl}_\ell(G)$, then $|\irr_{\ell'}(G)|=|\irr_{\ell'}(N_G(P))|$, where $\irr_{\ell'}(X)$ represents the set of irreducible characters of $X$ with degree prime to $\ell$.  The Alperin-McKay (AM) conjecture extends the McKay conjecture to include the role of blocks, and says that if $B$ is an $\ell$-block of $G$ with defect group $D$ and $b$ is the block of $N_G(D)$ with $b^G=B$, then the numbers of height zero characters of $B$ and $b$ coincide.



All of these conjectures remain open today, and although there is much evidence for the validity of the McKay conjecture, the question of \emph{why} it should be true remains unclear.  Reduction theorems for the McKay and Alperin-McKay conjectures have been proved in \cite{IsaacsMalleNavarroMcKayreduction} and \cite{spathAMreduction}, respectively, with the hope of providing not only a method by which to prove the conjectures, but also a better understanding of the deeper underlying reason behind them.  These reductions involve a list of conditions that a simple group must satisfy in order to be ``good" for the McKay (resp. Alperin-McKay) conjecture for a prime $\ell$.  The reductions say that if every finite simple group is ``good" for the McKay (resp. Alperin-McKay) conjecture for $\ell$, then \emph{every} finite group satisfies the conjecture for the prime $\ell$.

%

An $\ell$-weight of a finite group $G$ is a pair $(Q,\mu)$, where $Q$ is an $\ell$-radical subgroup (i.e. an $\ell$-subgroup such that $Q=\textbf{O}_\ell(N_G(Q))$) and $\mu$ is a defect-zero character of $N_G(Q)/Q$.  
More generally, a weight for a block $B$ of $G$ is a pair $(Q,\mu)$ as before, where $\mu$ lies in a block $b$ of $N_G(Q)$ for which the induced block $b^G$ is $B$.  The Alperin weight conjecture (AWC) says that if $G$ is a finite group and $\ell$ is a prime dividing $|G|$, then  the number of irreducible Brauer characters of $G$ equals the number of $G$-conjugacy classes of $\ell$-weights of $G$.  The blockwise Alperin weight conjecture (BAWC) says that the number of irreducible Brauer characters belonging to a block $B$ of $G$ equals the number of $G$-conjugacy classes of $\ell$-weights of $B$.

%

In \cite{TiepNavarroAWCreduction}, Navarro and Tiep prove a reduction for the Alperin weight conjecture in the same spirit as that for the McKay conjecture, and in \cite{SpathBAWCreduction}, Sp{\"a}th extends this reduction to the blockwise version of the Alperin weight conjecture.

The reductions for these conjectures give us hope of proving them by appealing to the classification of finite simple groups.  It has been shown that under certain conditions, a simple group of Lie type is ``good" for the various conjectures, but it still needs to be shown in general.  For example, it is known (see \cite{TiepNavarroAWCreduction}, \cite{SpathBAWCreduction}) that a simple group of Lie type defined in characteristic $p$ is ``good" for the Alperin weight and blockwise Alperin weight conjectures for the prime $\ell=p$, but the question is still open when $\ell\neq p$.  Similarly, for $\ell=p\geq5$, Sp{\"a}th has shown in \cite[Proposition 8.4]{spathAMreduction} that a simple group of Lie type defined in characteristic $p$ is ``good" for the Alperin-McKay conjecture (and therefore also the McKay conjecture) for the prime $\ell=p$.  It is worth noting that for $q\geq 4$ a power of $2$, the same argument shows that indeed, $Sp_6(q)$ and $Sp_4(q)$ are good for the Alperin-McKay conjecture for the prime $2$, as the Schur multiplier is non-exceptional in these cases.

In \cite{cabanesSp4good}, Cabanes shows that $Sp_4(2^a)$ is ``good" for the McKay conjecture for all primes $\ell\neq 2$.  According to the discussion preceding \cite[Theorem A]{SpathBAWCreduction}, G. Malle has shown that $Sp_4(2)'\cong A_6$ is ``good" for the blockwise Alperin weight conjecture. We prove the following statement, which therefore implies that $Sp_6(2^a)$ and $Sp_4(2^a)$ are ``good" for each of these conjectures for every prime:

\begin{theorem}\label{thm:Sp6good}
The simple groups $Sp_6(q)$ with $q$ even and $Sp_4(q)$ with $q\geq 4$ even are ``good" for the McKay, Alperin-McKay, Alperin weight, and blockwise Alperin weight conjectures for all primes $\ell\neq 2$.  Moreover, the simple group $Sp_4(2)'$ is ``good" for the Alperin-McKay conjecture for all primes $\ell$ (including $\ell=2$) and $Sp_6(2)$ is ``good" for the Alperin-McKay conjecture for the prime $\ell=2$.
\end{theorem}

For the Alperin weight conjecture and its blockwise counterpart, the reduction requires finding a partition $\ibr_\ell(G)=\bigcup_{Q\in \mathrm{Rad}(G)} \ibr_\ell(G|Q)$, where $\mathrm{Rad}(G)$ is a $G$-transversal of $\ell$-radical subgroups of $G$, and bijections $\ast_Q$ from $\ibr_\ell(G|Q)$ to the defect-zero characters of $N_G(Q)/Q$. For the McKay and Alperin-McKay conjectures, the reduction requires finding bijections $\Omega_Q\colon \irr_0(G|Q)\rightarrow \irr_0(N_G(Q)|Q)$, where $\irr_0(X|D)$ is the set of height-zero irreducible characters of a group $X$ lying in a block with defect group $D$.  Among other conditions, the maps $\ast_Q$ and $\Omega_Q$ must commute with automorphisms of $G$ and be such that the block of $\psi\in\ibr_\ell(G|Q)$ or $\chi\in\irr_0(G|Q)$ is the induced block of $\psi^{\ast_Q}$ or $\Omega_Q(\chi)$.

Though our proof is rather specialized, we hope that it will lead us to find a more
general underlying pattern which will give us an idea of how to extend these results
to higher rank groups.  We make use here of the explicit knowledge of the character table for $Sp_6(2^a)$, found in \cite{Luebeckthesis}, as well as the decomposition numbers of unipotent blocks for $Sp_6(2^a)$ and lower rank groups of Lie type found in \cite{white2000}, \cite{white95}, \cite{james1990},  \cite{okuyamawakiSp4}, \cite{geckdecompSU}, and \cite{okuyamawakiSU3}.  We also exploit the fact that in most cases (namely, when $\ell\neq 3$), the radical subgroups of $Sp_6(2^a)$ are abelian and their normalizers and centralizers can be described relatively easily, and therefore the relevant characters of $N_G(Q)$ can be found explicitly from the structure of these groups.  For higher-rank symplectic groups, we see a larger number of primes $\ell$ yielding more complicated radical subgroups, which may make such a treatment more difficult.

\section{Preliminaries and Notation}\label{sec:preliminaries}
Throughout this paper, $\ell$ denotes a prime, thought of as the characteristic for a representation.  As usual, $\irr(X)$ will denote the set of irreducible ordinary characters of $X$ and $\ibr_\ell(X)$ will denote the set of irreducible $\ell$-Brauer characters of $X$.  Given a subgroup $Y\leq X$ and $\phi\in\irr(Y)$, $\ind_Y^X\phi$ will denote the induced character of $\phi$ to $X$ and $\irr(X|\phi)$ will denote the set of irreducible constituents of $\ind_Y^X\phi$.  Further, if $\chi\in\irr(X)\cup\ibr_\ell(X)$, $\mathrm{Bl}(X|\chi)$ will denote the block of the group $X$ containing $\chi$.  We will also denote by $\irr_0(X|D)$ the set of height-zero characters of $X$ which lie in any block with defect group $D$ and by $\dz(X)$ the set of defect-zero characters of $X$.
Given $\chi\in\irr(X)$, we will denote the central character associated to $\chi$ by $\omega_\chi$.  Let $\F_\ell$ be the field of cardinality $\ell$, and write $\overline{\F_\ell}$ for its algebraic closure. We will denote by $\ast$ a fixed isomorphism from the set of $\ell'$-roots of unity in $\C$ to $\overline{\F}_\ell^\times$, extended to a homomorphism $\ast\colon\widetilde{R}[G]\rightarrow \overline{\F}_\ell$ as in \cite[discussion before (15.19)]{isaacs}, where $\widetilde{R}\subset\C$ is the ring of local integers for the prime $\ell$, and set $\lambda_B=\omega_\chi^\ast$ for $B=\mathrm{Bl}(X|\chi)$, as in \cite[Chapter 15]{isaacs}.  Given a set $\mathfrak{S}$, write $\mathfrak{S}^+:=\sum_{x\in \mathfrak{S}} x$.  If $Y\leq X$ is a subgroup, and $b\in\mathrm{Bl}(Y)$, then the induced block $b^X$ is the unique block $B$ so that $\lambda_b^X(\mathcal{K}^+)=\lambda_B(\mathcal{K}^+)$ for all conjugacy classes $\mathcal{K}$ of $X$, if such a $B$ exists.  (In this situation, recall that $b^X$ is said to be defined.)  Recall that $\lambda_b^X(\mathcal{K}^+)$ is given by $\lambda_b\left((\mathcal{K}\cap Y)^+\right).$

If a group $X$ acts on a set $\mathfrak{S}$ and $\mathfrak{s}\subseteq \mathfrak{S}$, then we denote by $X_\mathfrak{s}$ or $\stab_X(\mathfrak{s})$ the subgroup of $X$ stabilizing $\mathfrak{s}$.  If $X$ acts on a group $Y$, we denote by $Y:X$ or $Y\rtimes X$ the semidirect product of $Y$ with $X$.  We may also say this is the extension of $Y$ by $X$.  In such situations, if $r$ is a positive integer and $p$ is a prime, we will write $Y:r$ if $X=C_r$ is the cyclic group of order $r$ and $Y:p^r$ if $X$ is elementary abelian of order $p^r$.


Let $G$ be a finite group of Lie type, with $G=\underline{G}^F$ for a connected reductive algebraic group $\underline{G}$ of characteristic $p\neq\ell$ and a Frobenius map $F$, and write $G^\ast=(\underline{G}^\ast)^{F^\ast}$, where $(\underline{G}^\ast, F^\ast)$ is dual to $(\underline{G}, F)$.   We can write $\irr(G)$ as a disjoint union $\bigsqcup \mathcal{E}(G, (s))$ of rational Lusztig series corresponding to $G^\ast$- conjugacy classes of semisimple elements $s\in G^\ast$.  Recall that the characters in the series $\mathcal{E}(G, (1))$ are called unipotent characters, and there is a bijection $\mathcal{E}(G,(s))\rightarrow\mathcal{E}(C_{G^\ast}(s), (1))$ such that if $\chi\mapsto \psi$, then $\chi(1)=[G^\ast:C_{G^\ast}(s)]_{p'}\psi(1)$.  (See, for example, \cite[Theorem 13.23 and Remark 13.24]{dignemichel}.)

Let $t$ be a semisimple $\ell'$ - element of $G^\ast$ and write $\mathcal{E}_\ell(G,(t)):=\bigcup\mathcal{E}(G,(ut)),$ where the union is taken over all $\ell$-elements $u$ in $C_{G^\ast}(t)$.  By a fundamental result of Brou{\'e} and Michel \cite{brouemichel}, $\mathcal{E}_\ell(G,(t))$ is a union of $\ell$-blocks.  Hence, we may view $\mathcal{E}_\ell(G,(t))$ as a collection of $\ell$-Brauer characters as well as a set of ordinary characters.


In \cite{bonnaferouquier}, Bonnaf{\'e} and Rouquier show that when $C_{\underline{G}^\ast}(t)$ is contained in an $F^\ast$-stable Levi subgroup, $\underline{L}^\ast$, of $\underline{G}^\ast$, then Deligne-Lusztig induction $R_L^G$ yields a Morita equivalence between $\mathcal{E}_\ell(L,(t))$ and $\mathcal{E}_\ell(G,(t))$, where $L=\underline{L}^{F}$ and $(\underline{L}, F)$ is dual to $(\underline{L}^\ast, F^\ast)$.

Note that when $G=Sp_6(q)$, $q$ even, with $G=\underline{G}^F$ and $(\underline{G}^\ast, F^\ast)$ in duality with $(\underline{G}, F)$, each semisimple conjugacy class $(s)$ of $G^\ast=(\underline{G}^\ast)^{F^\ast}$ satisfies that $|s|$ is odd.  (Here $|s|$ denotes the order of the element $s$.)  Hence by  \cite[Lemma 13.14(iii)]{dignemichel}, the centralizer $C_{\underline{G}^\ast}(s)$ is connected. While applying Deligne-Lusztig theory to $Sp_{2n}(q)$ with $q$ even, it is convenient to view $Sp_{2n}(q)$ as $SO_{2n+1}(q)\cong Sp_{2n}(q)$, so that $G^\ast = Sp_{2n}(q)$.   The following lemma is \cite[Lemma 2.4]{SchaefferFry}.

\begin{lemma}\label{lem:moritaequiv}
Let $G^\ast=Sp_6(q)$, $q$ even, with $G=\underline{G}^F$ and $(\underline{G}^\ast, F^\ast)$ in duality with $(\underline{G}, F)$.  The nontrivial semisimple conjugacy classes $(s)$ of $G^\ast$ each satisfy $C_{\underline{G}^\ast}(s)=\underline{L}^\ast$ for an $F^\ast$-stable Levi subgroup $\underline{L}^\ast$ of $\underline{G}^\ast$ with $C_{G^\ast}(s)=(\underline{L}^\ast)^{F^\ast}=:L^\ast$.  In particular, Bonnaf{\'e}-Rouquier's theorem \cite{bonnaferouquier} implies that there is a Morita equivalence $\mathcal{E}_\ell(L,(1))\rightarrow\mathcal{E}_\ell(G,(t))$ given by Deligne-Lusztig induction (composed with tensoring by a suitable linear character) when $t\neq 1$ is a semisimple $\ell'$-element, where $L=\underline{L}^{F}$ and $(\underline{L}, F)$ is dual to $(\underline{L}^\ast, F^\ast)$.
\end{lemma}

Let $q$ be a power of $2$.  In light of the above discussion and our need to understand radical subgroups, it will often be useful to understand the structure of centralizers of semisimple elements of $Sp_{2n}(q)$.   Given a monic irreducible polynomial $f\in\F_{q}[t]$ of degree $d$ with a root $\alpha$ in some splitting field for $f$, let $f^\checkmark$ denote the monic irreducible polynomial of degree $d$ with roots $\alpha^{-1}, \alpha^{-q},...,\alpha^{-q^{d-1}}$.

\begin{lemma}\label{lem:centralizersemisimpleSp}
Let $G=Sp_{2n}(q)$ with $q$ even, and let $s$ be a semisimple element of $G$.  Decompose the characteristic polynomial $P(t)\in\F_q[t]$ of $s$ acting on the natural module $\F_q^{2n}$ in the form $P(t)=(t-1)^{m_0}\cdot\prod_i f_i(t)^{m_i} \cdot \prod_{j} g_j(t)^{n_j}g_j^\checkmark(t)^{n_j}$, where each $f_i$, $g_j$ is an irreducible polynomial over $\F_q$, $f_i=f_i^\checkmark$ has degree $d_i$ and does not have $1$ as a root, $g_j\neq g_j^\checkmark$ with $\deg g_j=\deg g_j^\checkmark = k_j$, and $2n=m_0+\sum_i d_im_i+2\sum_jk_jn_j$.
Then
\[C_G(s)\cong Sp_{m_0}(q)\oplus \bigoplus_i GU_{m_i}(q^{d_i/2})\oplus \bigoplus_j GL_{n_j}(q^{k_j}).\]
\end{lemma}

For the remainder of the paper, $\ell$ is an odd prime and $q$ is a power of $2$.  We note that $|Sp_6(q)|=q^9(q^2-1)(q^4-1)(q^6-1)$, so if $\ell$ is a prime dividing $|Sp_6(q)|$ and $\ell\neq 3$, then $\ell$ must divide exactly one of $q-1$, $q+1$, $q^2+1$, $q^2+q+1$, or $q^2-q+1$.  If $\ell=3$, then it divides $q-1$ if and only if it divides $q^2+q+1$, and it divides $q+1$ if and only if it divides $q^2-q+1$.  When $\ell$ divides $q^6-1$, we will write $\epsilon\in\{\pm1\}$ for the number such that $\ell|(q^3-\epsilon)$.  If $\ell\neq 3$, we write $d$ for the integer such that $(q^6-1)_\ell=\ell^d$.  (Here $r_\ell$ is the $\ell$-part of the integer $r$.) If $\ell=3$, we write $d$ for the integer such that $3^d=(q-\epsilon)_3$, so that $(q^3-\epsilon)_3=3^{d+1}$.  In any case, we will denote by $m$ the integer $(q-\epsilon)_{\ell'}$.

We will also borrow from CHEVIE \cite{chevie} the notation for characters of $Sp_6(q)$ and the roots of unity $\zeta_i:=\exp\left(\frac{2\pi\sqrt{-1}}{q^i-1}\right)$ and $\xi_i:=\exp\left(\frac{2\pi\sqrt{-1}}{q^i+1}\right)$.  We will sometimes also use $\widetilde{\zeta}_i$ or $\widetilde{\xi}_i$ to denote a corresponding root of unity in $\overline{\F}_q^\times$.

The following sets for indices will be useful.  For $\epsilon\in\{\pm1\},$ let $I^0_{q-\epsilon}$ be the set
\begin{equation}\label{eq:I0q}
I^0_{q-\epsilon}:=\{i\in\Z: 1\leq i\leq q-\epsilon-1\},\end{equation} and let $I_{q-\epsilon}$ be a set of class representatives on $I_{q-\epsilon}^0$ under the equivalence relation $
i\sim j\iff i\equiv \pm j\mod (q-\epsilon)$. Let \begin{equation}\label{eq:I0q2}
I^0_{q^2+1}:=\{i\in \Z: 1\leq i\leq q^2\}\hbox{  and  } I^0_{q^2-1}:=\{i\in\Z: 1\leq i\leq q^2-1, (q-1)\not |\, i, (q+1)\not|\, i\},\end{equation} and let $I_{q^2-\epsilon}$ be a set of representatives for the equivalence relation on $I_{q^2-\epsilon}^0$ given by $i\sim j\iff i\equiv\pm j$ or $\pm qj\mod(q^2-\epsilon)$.  Similarly, let \begin{equation}\label{eq:I0q3}
I_{q^3-\epsilon}^0:=\{i\in\Z: 1\leq i\leq q^3-\epsilon; (q^2+\epsilon q+1)\not|\,i\}\end{equation} and $I_{q^3-\epsilon}$ a set of representatives for the equivalence relation on $I_{q^3-\epsilon}^0$ given by $i\sim j \iff i\equiv\pm j, \pm qj,$ or $\pm q^2j\mod(q^3-\epsilon)$.  Given one of these indexing sets, $I_{\ast}$, we write $I_{\ast}^k$ for the elements $(i_1,...,i_k)$ of $I_{\ast}\times I_{\ast}...\times I_{\ast}$ ($k$ copies) with none of $i_1, i_2,..., i_k$ the same and $I_{\ast}^{k*}$ for the set of equivalence classes of $I_{\ast}^{k}$ under $(i_1,...,i_k)\sim (i_{\rho(1)},..., i_{\rho(k)})$ for all $\rho\in S_k$.

Let $G:=Sp_6(q)$ and let $\E_1$ denote the set of unipotent characters and $\mathcal{E}_i(J)$ denote the Lusztig series $\mathcal{E}(G, (s))$ for $G$, where $s$ is conjugate in $G^\ast$ to the semisimple element $g_i(J)$ in the notation of \cite[Tabelle 19]{Luebeckthesis}. Here $J$ denotes the proper indices (for example, for the family $g_6$, $J=(i)$ for $i\in I_{q-1}$, and for the family $g_{32}$, $J=(i,j,k)$ where $(i,j,k)\in I_{q+1}^{3*}$).  For the convenience of the reader, we have included the list of characters found in $\mathcal{E}_i(J)$, along with the size and number of series of a given type in the appendix.  

D. White \cite{white2000} has calculated the decomposition numbers for the unipotent blocks of $G$, up to a few unknowns in the case $\ell|(q+1)$.  Moreover, using the theory of central characters, which are available in the CHEVIE system \cite{chevie} for $G$, it is straightforward to determine the block distribution of the remaining complex characters. Using this information, \prettyref{lem:moritaequiv}, \prettyref{lem:centralizersemisimpleSp}, and the knowledge of the decomposition matrices of lower-rank groups (see \cite{white95}, \cite{james1990},  \cite{okuyamawakiSp4}, \cite{geckdecompSU}, \cite{okuyamawakiSU3}), we can then determine the description of Brauer characters of $Sp_6(q)$ in terms of the restrictions of ordinary characters to $\ell$-regular elements. We omit these descriptions here, but this information can be found in \cite{SchaefferFryThesis} and in the appendix.  
In particular, the set $\E\left(G, (g_i(J))\right)$ forms a basic set for the blocks of $\E_\ell\left(G, (g_i(J))\right)$ for the semisimple $\ell'$-elements $g_i(J)$, and there is an $\aut(G)$-equivariant bijection from this basic set to the set of Brauer characters in $\E_\ell\left(G, (g_i(J))\right)$.

Given $\ell|(q^2-1)$, we will denote by $B_i(J)$ the $\ell$-blocks in $\mathcal{E}_\ell(G, (s))$ where $s$ is conjugate in $G^\ast$ to the semisimple element $g_i(J)$ in the notation of \cite[Tabelle 19]{Luebeckthesis} and by $\mathfrak{B}_i(J)$ the set $\ibr_\ell(B_i(J))$ of irreducible Brauer characters in the block.  In most cases, $C_{G^\ast}(s)$ has only one unipotent block, and therefore $\mathcal{E}_\ell(G,(s))$ is a single block.  However, when multiple such blocks exist, which occurs for $i=6,7,8,9$ when $\ell|(q^2-1)$, we will denote by $B_i(J)^{(0)}$ the block corresponding in the Bonnaf{\'e}-Rouquier correspondence to the principal block of $C_{G^\ast}(s)$ and by $B_i(J)^{(1)}$ the block corresponding to the unique other block of positive defect.  We will use $\mathfrak{B}_i(J)^{(0)}, \mathfrak{B}_i(J)^{(1)}$ to denote the corresponding irreducible Brauer characters.  Further, $B_0$ and $B_1$ will denote the  principal block and the cyclic unipotent block, respectively, as described in \cite{white2000}, and $\mathfrak{B}_0, \mathfrak{B}_1$ will denote $\ibr_\ell(B_0)$ and $\ibr_\ell(B_1)$, respectively.


\section{Radical Subgroups of $Sp_6(2^a)$ and $Sp_4(2^a)$}\label{sec:radsubs}

In this section we describe the radical subgroups of $Sp_6(q)$ with $q$ even and their normalizers.  In \cite{An94}, J. An describes the radical subgroups for $Sp_{2n}(q)$ with odd $q$, and his results in the first two sections extend to $Sp_6(q)$ when $q$ is even and $\ell|(q^2-1)$, so we will often refer the reader there.  We begin by setting some notation for the subgroups of $Sp_6(q)$ that will be of interest.

Let $G=Sp_{2n}(q)$ with $q$ even and let $\{e_1,e_2,...,e_n,f_1,f_2,...,f_n\}$ be the standard symplectic basis for the natural module $\F_q^{2n}$ for $G$.  For $r\leq n$, we can view $Sp_{2r}(q)$ as a subgroup of $G$ by identification with the pointwise stabilizer $\stab_G(e_{r+1},...,e_n, f_{r+1},..,f_n)$, and by iterating this, we see that for integers $r_1\geq r_2\geq...\geq r_n\geq 0$ so that $s:=\sum_{i=1}^n r_i\leq n$, we may view the direct product $\prod_i Sp_{2r_i}(q)$ as a subgroup of $G$ which stabilizes (point-wise) a $2(n-s)$-dimensional subspace of $\F_q^{2n}$. Moreover, we may further view $GL_{r}^\pm(q)$ as a subgroup of $Sp_{2r}(q)$, so that $\prod_i GL_{r_i}^\pm(q) \leq G$ under this embedding.  We will also require the embeddings $GL_{1}^\pm(q^3)\leq GL_{3}^\pm(q)$ and $GU_1(q^2)\leq GL_2(q)$.  (Here we use the notation $GL_r^+(q):=GL_r(q)$ and $GL_r^-(q):=GU_r(q)$.)

Now specialize $n=3$, so $G=Sp_6(q)$, and write $H:=Sp_4(q)=\stab_G(e_3,f_3)$.  Suppose first that $\ell|(q^2-1)$, and let $\epsilon\in\{\pm1\}$ be such that $\ell|(q-\epsilon)$, with $(q-\epsilon)_\ell=\ell^d$.  (We will also write $\epsilon$ for the corresponding sign $\pm$.)   Let $r_1\geq r_2\geq r_3\geq0$ be as in our discussion above, and define $Q_{r_1,r_2,r_3}:=\textbf{O}_\ell\left(Z\left(\prod_{i=1}^3 GL_{r_i}^\epsilon(q)\right)\right)$, viewed as an $\ell$-subgroup of $G$ under the embedding described above.  Then $C_G(Q_{r_1,r_2,r_3})=Sp_{2(n-s)}(q)\times\prod_{i=1}^3 GL_{r_i}^\epsilon(q)$, and if $c_i$ is the number of times $r_i$ appears, then $N_G(Q_{r_1,r_2,r_3})=Sp_{2(n-s)}(q)\times\prod \left(GL_{r_i}^\epsilon(q):2\right)\wr S_{c_i}$, where the product is now taken over the $i$ so that each distinct $r_i$ appears only once.  (This can be seen from direct calculation, or by arguments similar to those in \cite[Sections 1 and 2]{An94}.)  Here we can view $GL_{r_i}^\epsilon(q)$ as its image under the map $A\mapsto \diag(A, \tw{T}A^{-1})$, possibly viewed in the overgroup $Sp_{2r_i}(q^2)$, with the $C_2$ extension inducing the graph automorphism $\tau_{r_i}^{(\epsilon)} \colon A\mapsto \tw{T}A^{-1}$ on $GL_{r_i}^\epsilon(q)$.  When $r_i=0$ for some $i$, we will suppress the notation, so that we will write, for example, $Q_1$ rather than $Q_{1,0,0}$, and $Q_{1,1}$ rather than $Q_{1,1,0}$.  Hence $Q_1, Q_2,$ and $Q_3$ are cyclic groups of order $\ell^d$, $Q_{1,1}$ and $Q_{2,1}$ are isomorphic to $C_{\ell^d}\times C_{\ell^d}$, and $Q_{1,1,1}$ is isomorphic to $C_{\ell^d}\times C_{\ell^d}\times C_{\ell^d}$.  Moreover, notice that $Q_{1,1}\in\Syl_\ell(H)$, and when $\ell\neq 3$, $Q_{1,1,1}\in\Syl_\ell(G)$.

If $\ell=3$, let $P$ denote the Sylow subgroup, which is $Q_{1,1,1}\rtimes C_3$, or $C_{\ell^d}\wr C_3$, which we can view inside $Sp_2(q)\wr S_3\leq G$.  Write $Z:=Q_3$ and let $R\leq GL_3^\epsilon(q)$ be the embedding of the symplectic-type group which is the central product of $Z$ and an extraspecial group $E$ of order $27$ with exponent $3$, as in \cite[(1A) and (1B)]{An94}.

Now suppose that $\ell|(q^4+q^2+1)$, and let $\epsilon$ be so that $\ell|(q^2+\epsilon q+1)$.  Write $Q^{(3)}:=\textbf{O}_\ell(Z(GL_1^\epsilon(q^3)))$.  When $\ell\neq 3$, $Q^{(3)}$ is a cyclic Sylow $\ell$-subgroup of $G$ of order $(q^2+\epsilon q+1)_\ell$.  When $\ell=3$, we have $3|(q-\epsilon)$ as well, and $(q^2+\epsilon q+1)_3=3$.  In this case, $Q^{(3)}$ is a cyclic group of order $3^{d+1}$ where $(q-\epsilon)_3=3^d$.  When $\ell|(q^2+1)$, write $Q^{(2)}:=\textbf{O}_\ell(Z(GU_1(q^2)))$ so that $Q^{(2)}$ is a cyclic Sylow $\ell$-subgroup of $G$.

Let $s:=s_3, s_2$ be a generator of $Q^{(3)}, Q^{(2)},$ respectively.  Write $N:=N_G(\langle s\rangle)$ and $C:=C_G(\langle s\rangle)$.  From the description in \cite[Tabelles 7,10]{Luebeckthesis} of semisimple classes of $G$, we see that $s$ is conjugate to $s^i$ if and only if $i\in\{\pm q,..., \pm q^{j}\}$, where $j=3,2$ respectively, so that $N/C=\langle \tau,\beta\rangle$ is generated by $\tau\colon s\mapsto s^{-1}$, $\beta\colon s\mapsto s^q$.  Moreover, $C_G(s_3)=C_{q^3-\epsilon}$ and $C_G(s_2)=C_H(s_2)\times Sp_2(q)=C_{q^2+1}\times Sp_2(q)$, so $N_G(Q^{(3)})=C_{q^3-\epsilon}:6$ and $N_G(Q^{(2)})=N_H(Q^{(2)})\times Sp_2(q)=C_{q^2+1}:2^2\times Sp_2(q)$.

\begin{proposition}\label{prop:radsubs}
\begin{enumerate}
\item Let $G=Sp_6(q)$ with $q$ even and let $Q$ be a nontrivial $\ell$-radical subgroup of $G$ for a prime $\ell\neq 2$ dividing $|G|$.  Then:
\begin{itemize}
\item If $3\neq \ell|(q^2-1)$, then $Q$ is $G$-conjugate to one of $Q_1, Q_2, Q_3, Q_{1,1}, Q_{2,1}$ or $Q_{1,1,1}$.
\item If $\ell=3|(q^2-1)$, then $Q$ is $G$-conjugate to one of $Q_1, Q_2, Q_3, Q_{1,1}, Q_{2,1}$, $Q_{1,1,1}, Q^{(3)}, P$ or $R$.
\item If $3\neq \ell|(q^4+q^2+1)$, then $Q$ is $G$-conjugate to $Q^{(3)}$.
\item If $\ell|(q^2+1)$, then $Q$ is $G$-conjugate to $Q^{(2)}$.
\end{itemize}

\item Let $H=Sp_4(q)$ with $q$ even (viewed as $\stab_G(e_3,f_3)$) and let $Q$ be a nontrivial $\ell$-radical subgroup of $H$ for a prime $\ell\neq 2$ dividing $|H|$.  Then:
\begin{itemize}
\item If $\ell|(q^2-1)$, then $Q$ is $H$-conjugate to one of $Q_1, Q_2,$ or $Q_{1,1}$.
\item If $\ell|(q^2+1)$, then $Q$ is $H$-conjugate to $Q^{(2)}$.
\end{itemize}
\end{enumerate}

Moreover, no two of the subgroups listed are $G$-conjugate.
\end{proposition}
\begin{proof}
First, it is clear from the description in \cite{Luebeckthesis} of the semisimple classes of $G$ that the listed subgroups each lie in a different conjugacy class of subgroups.

On the other hand, when $\ell|(q^2-1)$, the relevant arguments from \cite[Sections 1 and 2]{An94} apply in the case $G=Sp_6(q)$ and $H=Sp_4(q)$ with $q$ even, with some minor adjustments, so the remainder of the statement can be extracted from \cite[(2D)]{An94}.  When $\ell\not|\,(q^2-1)$, a Sylow $\ell$-subgroup is cyclic, say generated by the semisimple element $s$.  Then any power $s^i$ of $s$ has the same centralizer, which can be seen from \prettyref{lem:centralizersemisimpleSp}.    Moreover, this centralizer contains a cyclic direct factor $C$ containing $\langle s\rangle$.  Hence $\langle s\rangle$ is characteristic in the centralizer of any proper, nontrivial subgroup $\langle s^i\rangle$ of the Sylow subgroup $\langle s\rangle$, so $\langle s^i\rangle$ cannot be $\ell$-radical.
\end{proof}
We note that more details regarding \prettyref{prop:radsubs} can be found in \cite{SchaefferFryThesis}.

\section{Characters of $N_G(Q)$ and the Maps}\label{sec:dzN(Q)}\label{sec:maps}
Let $G=Sp_6(q)$.  Let $Q$ be an $\ell$-radical subgroup, and write $N:=N_G(Q)$ and $C:=C_G(Q)$. In this section, we describe the characters of $N$ that will be of interest, and in particular the defect-zero characters of $N/Q$.  We then go on to describe disjoint sets $\underline{\irr_0}(G|Q)$ and $\underline{\irr_0}(N|Q)$ and bijections $\Omega_Q\colon\underline{\irr_0}(G|Q)\rightarrow \underline{\irr_0}(N|Q)$.  In \prettyref{sec:good} below, we show that in fact $\irr_0(G|Q)=\underline{\irr_0}(G|Q)$ and $\irr_0(N|Q)=\underline{\irr_0}(N|Q)$, and that these are the required maps for the reduction of the Alperin-McKay conjecture in \cite{spathAMreduction}.  In this section, we also define maps $\ast_Q\colon \ibr_\ell(G|Q)\rightarrow\dz(N_G(Q)/Q)$ in the cases $\ell|(q^2-1)$, which we show in \prettyref{sec:good} are the required maps for the reduction of the (B)AWC in \cite{SpathBAWCreduction}.  (We define the sets $\ibr_\ell(G|Q)$, $\underline{\irr_0}(G|Q)$, and $\underline{\irr_0}(N|Q)$ to be the sets of characters involved in the maps described here.) Note that for $\ell\not|(q^2-1)$, we define no maps $\ast_Q$, as in this case the Sylow $\ell$-subgroups of $G$ are cyclic, and therefore by \cite[Proposition 6.2]{SpathBAWCreduction}, $G$ is BAWC-good.  Also, in \prettyref{sec:mapsSp4}, we give similar maps for $Sp_4(q)$.

The reader will see that the sets $\ibr_\ell(G|Q)$ are very closely related to the sets $\irr_0(G|Q)$. In many cases, we have chosen $\psi\in \ibr_\ell(G|Q)$ if $\psi$ lies in a block with defect group $Q$, and in these cases the largest-degree character $\chi\in \irr(G)$ in the decomposition of $\psi$ with respect to the restrictions of ordinary characters in a basic set of the block $\mathrm{Bl}(G|\psi)$ to $\ell'$-elements satisfy $\chi\in \irr_0(G|Q)$.  Of course, when $Q$ is not a defect group for any block, which we will see occurs for some $Q$ in the case $\ell=3$, the situation is a little different.  Moreover, we remark that if we identify the semisimple class $(s)$ in $G^\ast$ with the corresponding semisimple class $(s')$ in $G\cong G^\ast$, then the sets $\underline{\irr_0}(G|Q)$ primarily contain characters of $G$ from series $\mathcal{E}(G, (s))$ with $Q\in \mathrm{Syl}_\ell(C_{G}(s'))$ (up to $G$-conjugacy), which may suggest at least an initial strategy for defining maps for more general groups. 

In most cases, the characters for $N$ here will be given by the description of an irreducible constituent of $C$.  That is, the maps we describe are not defined uniquely, but rather will be from a given set of characters of $G$ to the set of characters of $N$ with a given constituent on $C$.   In these situations, the bijection can then be chosen by specifying a map between the two subsets, with the only requirement being that the choice of image for a given family of characters is consistent throughout the choices of indices $J$.   We also remark that the fact that for all of the below maps, the number of characters of $N$ with the same constituents on $C$ matches the number of $\chi\in \irr(G)$ that we have mapped to them, will follow from the discussion of characters of $N_G(Q)$ and the Bonnaf{\'e}-Rouquier correspondence, together with the knowledge of $C_{G^\ast}(s)$ and its unipotent blocks for semisimple $s\in G^{\ast}$.  
The indexing sets for the $\E_i(J)$ are evident from \cite[Tabelle 19]{Luebeckthesis} - note that they match the indexing sets for the images under $\Omega_Q$ given below.  For the convenience of the reader, we have also included the list of indexing sets in the appendix.


Recall that for $\ell|(q^2-1)$, we have radical subgroups $Q_1, Q_2, Q_3, Q_{1,1}, Q_{2,1},$ and $Q_{1,1,1}$, with the additional subgroups $Q^{(3)}, P,$ and $R$ when $\ell=3$.  So, when referring to $Q_1, Q_2, Q_3, Q_{1,1}, Q_{2,1},$ and $Q_{1,1,1}$ we will assume $\ell|(q^2-1)$, without necessarily assuming that $\ell\neq 3$, unless otherwise stated.  When referring to $P, R$, we assume $\ell=3$, when referring to $Q^{(3)}$, we assume $\ell|(q^4+q^2+1)$, with the possibility that $\ell=3$, and when referring to $Q^{(2)}$, we assume $\ell|(q^2+1)$.  Throughout this section, let $\epsilon\in\{\pm1\}$ such that $\ell|(q^3-\epsilon)$, if such an $\epsilon$ exists.

The characters of $N$ that we are interested in are those which are defect-zero characters of $N/Q$ or height-zero characters of $N$ with defect group $Q$.  In either case, these characters will be $\chi\in\irr(N)$ with $\chi(1)_{\ell}=|N/Q|_{\ell}$.  So if $Q=Q_1,Q_2,$ or $Q_3$, we have $\chi(1)_\ell=\ell^{2d}$, except in the case $Q_3$ when $\ell=3$, in which case $\chi(1)_3=3^{2d+1}$.  If $Q=Q_{1,1}$ or $Q_{2,1}$, then $\chi(1)_\ell=\ell^d$, and if $Q$ is a Sylow subgroup, $\chi(1)_\ell=1$.  If $Q=Q_{1,1,1}, Q^{(3)}$ or $R$ when $\ell=3|(q^2-1)$, then $\chi(1)_{3}=3$.  In most cases, it will suffice for our purposes to describe the constituents of $\chi$ when restricted to $C$, and to keep in mind the action of $N/C$ on $C$ and its characters.

In many of the groups we are concerned with, we have an extension of a subgroup by $C_2$.  Suppose that $X=Y:2$, with the order-two automorphism on $Y$ denoted by $\tau$.  By Clifford theory, a character $\chi\in\irr(X)$ satisfies $\chi|_Y=\theta+\theta^\tau$ if an irreducible constituent $\theta$ of $\chi|_Y$ is not invariant under the automorphism $\tau$, and in this case, $\chi=\theta^X=(\theta^\tau)^X$.  Since $X/Y$ is cyclic, if a constituent $\theta$ is invariant under $\tau$, then $\chi|_Y=\theta$.  In this case, Gallagher's theorem tells us that there are two such characters $\chi$, namely $\chi$ and $\chi\lambda$ where $\lambda$ is the nonprincipal character of $X/Y\cong C_2$.  In particular, $\chi\in\irr(X)$ has degree $\chi(1)=2\theta(1)$ or $\theta(1)$ for some $\theta\in \irr(Y)$.  In general, when a character $\theta$ of $Y\lhd X$ extends to $X$, we will sometimes write $\theta^{(\nu)}$ for the character $\theta\nu$ of $X$ with $\nu\in X/Y$ by Gallagher's theorem.

We note that from the discussions below for $N_G(Q)$, it will also be easy to see the characters of interest for $N_H(Q)$ with $H=Sp_4(q)$ by similar arguments.

\subsection{Characters of Some Relevant Subgroups}\label{sec:somecharsofsubgps}

From \prettyref{sec:radsubs}, we see that when $\ell|(q^4-1)$, the characters of the groups $GL_{r}^\epsilon(q):2$, for $r=1,2,3$, $Sp_4(q)$, and $SL_2(q)=Sp_2(q)$ will play a large role for many of the radical subgroups, so we discuss the characters of these groups here.  Recall that the $C_2$ extension of $GL_r^\epsilon(q)$ considered here acts on $GL_r^\epsilon(q)$ via $\tau_r^{(\epsilon)}\colon A\mapsto \tw{T}A^{-1}$.  We simply write $\tau:=\tau_r^{(\epsilon)}$ for this graph automorphism when $r, \epsilon$ are understood.

First let $\ell|(q-\epsilon)$ for $\epsilon\in\{\pm1\}$.  Let $\varphi_i\in \irr(C_{q-\epsilon})=\irr(GL_1^\epsilon(q))$ denote the linear character which maps $\tilde{\zeta}\mapsto \zeta^{i}$, where $\zeta, \tilde{\zeta}$ are $\zeta_1, \tilde{\zeta}_1$ or $\xi_1, \tilde{\xi}_1$, in the cases $\epsilon=1$ and $-1$, respectively. Then $\varphi_i^\tau=\varphi_{-i},$ so $\varphi_i$ is invariant under $\tau$ exactly when $(q-\epsilon)|i$, i.e., when $\varphi_i=1$.  Hence an irreducible character of $GL_1^\epsilon(q):2$ which is nontrivial on $GL_1^\epsilon(q)$ can be identified by a constituent of its restriction to $GL_1^\epsilon(q)$, and therefore can be labeled by $\widetilde{\varphi_i}:=\ind_{GL_1^\epsilon(q)}^{GL_1^\epsilon(q):2}\varphi_i$ for $i\in I_{q-\epsilon}$, where $I_{q-\epsilon}$ is as defined after \eqref{eq:I0q}.  (Note that this yields $|I_{q-\epsilon}|=\frac{q-\epsilon-1}{2}$ characters of $GL_1^\epsilon(q):2$ of this form.)  Moreover, there are two characters $1^{(1)}, 1^{(-1)}$ of $GL_1^\epsilon(q):2$ which are trivial on $GL_1^\epsilon(q)$, corresponding to the two characters $\{\pm1\}$ of $C_2$, by Gallagher's theorem.

Writing $GL^\epsilon_2(q)\cong C_{q-\epsilon}\times SL_2(q)$, we note that $\tau$ induces an inner automorphism of $SL_2(q)$, so fixes all characters of $SL_2(q)$, and the action of $\tau$ on $C_{q-\epsilon}$ is the same as in the preceding paragraph.  So, we will write $\varphi=(\varphi_i\times \psi)$ for the character of $GL^\epsilon_2(q)\cong C_{q-\epsilon}\times SL_2(q)$, with $\varphi_i$ as above, and $\psi\in \irr(SL_2(q))$.  The characters of $GL^\epsilon_2(q):2$ whose restriction to $C_{q-\epsilon}$ is nontrivial will then be denoted $\widetilde{(\varphi_i\times \psi)}:=\ind_{GL^\epsilon_2(q)}^{GL^\epsilon_2(q):2}(\varphi_i\times\psi)$ for $i\in I_{q-\epsilon}$ and those whose restriction to $C_{q-\epsilon}$ is trivial will be denoted $(1\times\psi)^{(1)}, (1\times\psi)^{(-1)}$.   Now, the only series of characters of $SL_2(q)=Sp_2(q)$ with degree divisible by $\ell$ is $\chi_4(j)$ when $\epsilon =1$ and $\chi_3(j)$ when $\epsilon=-1$, with degrees $q-\epsilon$ and indexing $j\in I_{q+\epsilon}$ (see, for example, the character table information in CHEVIE \cite{chevie}).  To ease notation and emphasize the degree of this character, we will write $\chi_{q-\epsilon}(j)$ to denote the character $\chi_4(j)$ or $\chi_3(j)$ of $SL_2(q)$ in the case $\epsilon=1$ or $-1$, respectively.  (Also, when $\ell|(q^2+1)$, note that no character of $SL_2(q)$ has degree divisible by $\ell$.)

Now consider $GL_3^\epsilon(q):2$.   The characters $\chi_8(i)$ of $GL_3^\epsilon(q)$ (in the notation of CHEVIE \cite{chevie}), indexed by $1\leq i\leq q^3-\epsilon$ with $(q^2+\epsilon q+1)\not|\,i$ and $\chi_8(i)=\chi_8(qi)=\chi_8(q^2i)$, each have degree $(q-\epsilon)^2(q+\epsilon)$ and are the only characters of $GL_3^\epsilon(q)$ of degree divisible by $\ell^{2d}$ when $\ell|(q^2-1)$.  Inspection of the character table in CHEVIE reveals that $\chi_8(i)^\tau=\chi_8(-i)$ and no character in this series is invariant under $\tau$.  So, the characters we will be concerned with for this group are of the form $\chi_8(i)+\chi_8(-i)$ on $GL^\epsilon_3(q)$, and we will write $\widetilde{\chi_8(i)}:=\ind_{GL^\epsilon_3(q)}^{GL_3^\epsilon(q):2}\chi_8(i)$ for the corresponding character of $\irr\left((GL_3^\epsilon(q):2)|\chi_8(i)\right)$, indexed by $i\in I_{q^3-\epsilon}$, where $I_{q^3-\epsilon}$ is as defined after \eqref{eq:I0q3}.

Finally, when $\ell|(q-\epsilon)$, the irreducible characters $\theta$ of $Sp_4(q)$ with $\theta(1)_\ell=\ell^{2d}$ are those in the families (in the notation of CHEVIE) $\chi_5, \chi_{18}(i), \chi_{19}(i,j)$ when $\epsilon=1$ and $\chi_2, \chi_{15}(i,j), \chi_{18}(i)$ when $\epsilon=-1$.  We note that $\chi_2$ and $\chi_5$ are the Weil characters $\rho_2^2$ and $\alpha_2$, respectively, in the notation of \cite{TiepGuralnick04}.  Also, note that the indexing for the families $\chi_{15}(i,j)$ is $(i,j)\in I_{q-1}^{2*}$, 
for $\chi_{18}(i)$ is $i\in I_{q^2+1}$, and for $\chi_{19}(i,j)$ is $i,j\in I_{q+1}^{2*}$.  

When $Q\in \mathrm{Syl}_\ell(G)$, all characters of $N_G(Q)$ have defect group $Q$ (see for example \cite[Corollary (15.39)]{isaacs}), since $Q$ is an $\ell-$radical subgroup and $Q\in \mathrm{Syl}_\ell(N_G(Q))$.  Hence in this case, $\irr_0(N_G(Q)|Q)=\irr_{\ell'}(N_G(Q))$.

\subsection{$Q=Q_1$}

Let $Q:=Q_1$ with $\ell|(q-\epsilon)$, so that $N=\left(GL_1^\epsilon(q):2\right)\times Sp_4(q)$ and $C=GL_1^\epsilon(q)\times Sp_4(q)$ from \prettyref{sec:radsubs}.  Let $\chi\in\irr(N)$ with $\chi(1)_\ell=\ell^{2d}$.  If $\chi$ is nontrivial on $GL_1^\epsilon(q)$, Clifford theory and the discussion in \prettyref{sec:somecharsofsubgps} yield that $\chi$ is uniquely determined by a constituent $(\varphi_i\times \psi)$ of $\chi|_C$, where $i\in I_{q-\epsilon}$ and $\psi\in \irr(Sp_4(q))$ is one of the characters described in the fifth paragraph of \prettyref{sec:somecharsofsubgps}.  Hence $\chi=(\widetilde{\varphi_i}\times \psi)$, where $\widetilde{\varphi_i}\in\irr(GL_1^\epsilon(q):2)$ is as defined in \prettyref{sec:somecharsofsubgps}. If $\chi$ is trivial on $GL_1^\epsilon(q)$ then $\chi|_C=(1\times\psi)$ (with $\psi$ again as above) is irreducible, 
and there are two choices $(1^{(1)}\times\psi)$ and $(1^{(-1)}\times\psi)$ for $\chi$ for each such choice of $\psi$.

Now, to be a character of $N/Q$, we require that $Q$ be in the kernel.  As $Q=\textbf{O}_\ell(GL_1^\epsilon(q))$, we see that this means $\dz(N/Q)$ is comprised of the two characters $(1^{(1)}\times\psi)$ and $(1^{(-1)}\times\psi)$ of $\irr(N|\left(1\times \psi\right))$ for each $\psi\in\irr(Sp_4(q))$ as above, together with the characters $(\widetilde{\varphi_i}\times \psi)$ of $\irr(N|(\varphi_i\times \psi))$, where $\psi$ is as above and $i\in I_{q-\epsilon}$ satisfies the additional condition $\ell^d|i$.

The maps $\Omega_Q\colon\underline{\irr_0}(G|Q)\rightarrow \underline{\irr_0}(N|Q)$ and $\ast_Q\colon\ibr_\ell(G|Q)\rightarrow\dz(N/Q)$ for this radical subgroup are below.  (Recall that we define $\underline{\irr_0}(G|Q)$ to be the characters of $G$ involved in our map $\Omega_Q$ and $\ibr_\ell(G|Q)$ to be the set of Brauer characters involved in our map $\ast_Q$.)
\scriptsize
\begin{align*}
 \Omega_{Q_1}\colon & \scriptsize\left\{\begin{array}{cc}
                      \{\chi_5,\chi_{11}\}\rightarrow \irr(N|(1\times\alpha_2)) & \epsilon=1 \\
                      \{\chi_4,\chi_{9}\}\rightarrow \irr(N|(1\times\rho_2^2)) & \epsilon=-1
                    \end{array}\right.\\
&\scriptsize \left\{\begin{array}{cc}
\chi_{17}(i)\mapsto \irr(N|(\varphi_i\times\alpha_2)) & \epsilon=1\\
\chi_{20}(i)\mapsto \irr(N|(\varphi_i\times \rho_2^2)) & \epsilon=-1
\end{array}\right.\quad i\in I_{q-\epsilon}\\
& \scriptsize\left\{\begin{array}{cc}
\E_{23}(i,j)\rightarrow \irr(N|(1\times\chi_{19}(i,j))) & \epsilon=1\\
\E_{17}(i,j)\rightarrow \irr(N|(1\times \chi_{15}(i,j))) & \epsilon=-1 \end{array}\right.\quad (i,j)\in I_{q+\epsilon}^{2*}\\
&\scriptsize\E_{24}(i)\rightarrow \irr(N|(1\times\chi_{18}(i))) \quad\hbox{(for $\epsilon=1$ or $-1$)}, \quad i\in I_{q^2+1}\\
&\scriptsize\left\{\begin{array}{cc}
\E_{28}(i,j,k)\rightarrow \irr(N|(\varphi_i\times\chi_{19}(j,k))) & \epsilon=1\\
\E_{26}(j,k,i)\rightarrow \irr(N|(\varphi_i\times\chi_{15}(j,k)))& \epsilon=-1\end{array}\right.\quad i\in I_{q-\epsilon}, (j,k)\in I_{q+\epsilon}^{2*}\\
&\scriptsize\left\{\begin{array}{ccc}
\E_{30}(i,j)&\multirow{2}{*}{$\rightarrow \irr(N|(\varphi_i\times\chi_{18}(j)))$} & \epsilon=1\\
\E_{33}(j,i) && \epsilon=-1\end{array}\right.\quad i\in I_{q-\epsilon}, j\in I_{q^2+1}.
\end{align*}

\scriptsize
\begin{align*}
\ast_{Q_1}\colon &\scriptsize \left\{\begin{array}{ccc}
                      \multirow{2}{*}{$\mathfrak{B}_1\rightarrow$} &\irr\left(N|(1\times\alpha_2)\right) & \epsilon=1 \\
                      &\irr\left(N|(1\times\rho_2^2)\right) & \epsilon=-1
                    \end{array}\right.\\
&\scriptsize \left\{\begin{array}{cc}
\mathfrak{B}_6(i)^{(1)}\rightarrow \irr\left(N|(\varphi_i\times\alpha_2)\right) & \epsilon=1\\
\mathfrak{B}_7(i)^{(1)}\rightarrow \irr\left(N|(\varphi_i\times \rho_2^2)\right) & \epsilon=-1
\end{array}\right.\quad i\in I_{q-\epsilon}, \ell^d|i\\
&\scriptsize \left\{\begin{array}{cc}
\mathfrak{B}_{23}(i,j)\rightarrow \irr\left(N|(1\times\chi_{19}(i,j))\right) & \epsilon=1\\
\mathfrak{B}_{17}(i,j)\rightarrow \irr\left(N|(1\times \chi_{15}(i,j))\right) & \epsilon=-1 \end{array}\right.\quad (i, j)\in I_{q+\epsilon}^{2*}\\
&\scriptsize \mathfrak{B}_{24}(i)\rightarrow \irr\left(N|(1\times\chi_{18}(i))\right) \quad\hbox{(for $\epsilon=1$ or $-1$)}\quad i\in I_{q^2+1}\\
&\scriptsize\left\{\begin{array}{cc}
\mathfrak{B}_{28}(i,j,k)\rightarrow \irr\left(N|(\varphi_i\times\chi_{19}(j,k))\right) & \epsilon=1\\
\mathfrak{B}_{26}(j,k,i))\rightarrow \irr\left(N|(\varphi_i\times\chi_{15}(j,k))\right)& \epsilon=-1\end{array}\right.\quad i\in I_{q-\epsilon}, \ell^d|i; \quad(j, k)\in I_{q+\epsilon}^{2*}\\
&\scriptsize\left\{\begin{array}{ccc}
\mathfrak{B}_{30}(i,j)&\multirow{2}{*}{$\rightarrow \irr\left(N|(\varphi_i\times\chi_{18}(j))\right)$} & \epsilon=1\\
\mathfrak{B}_{33}(j,i) && \epsilon=-1\end{array}\right.\quad i\in I_{q-\epsilon}, \ell^d|i; \quad j\in I_{q^2+1}.
\end{align*}
\normalsize

\subsection{$Q=Q_2$}

Let $Q=Q_2$, so $N=\left(GL_2^\epsilon(q):2\right)\times Sp_2(q)$ and $C=GL_2^\epsilon(q)\times Sp_2(q)$.  Recall that we view $GL_2^\epsilon(q)$ as $C_{q-\epsilon}\times SL_2(q)$.  When restricted to $C$, the characters of $N$ satisfying $\chi(1)_\ell=\ell^{2d}$ have constituents of the form $\left((\varphi_i\times \chi_{q-\epsilon}(j))\times \chi_{q-\epsilon}(k)\right)$, with $\varphi_i\in\irr(C_{q-\epsilon})$ and $\chi_{q-\epsilon}\in\irr(SL_2(q))=\irr(Sp_2(q))$ as in \prettyref{sec:somecharsofsubgps}.  If $\chi$ is nontrivial on $C_{q-\epsilon}$, then again it is uniquely determined by a choice of $i\in I_{q-\epsilon}$ and $(j,k)\in I_{q+\epsilon}\times I_{q+\epsilon}$, so that $\chi=\left(\widetilde{(\varphi_i\times \chi_{q-\epsilon}(j))}\times\chi_{q-\epsilon}(k)\right)$.  
If $\chi$ is trivial on $C_{q-\epsilon}$, then $\chi|_C$ is irreducible and there are again two choices of $\chi$ with $\chi|_C=\left((1\times \chi_{q-\epsilon}(j))\times\chi_{q-\epsilon}(k)\right)$ for each $(j,k)\in I_{q+\epsilon}\times I_{q+\epsilon}$, namely $\left((1\times \chi_{q-\epsilon}(j))^{(1)}\times\chi_{q-\epsilon}(k)\right)$ and $\left((1\times \chi_{q-\epsilon}(j))^{(-1)}\times\chi_{q-\epsilon}(k)\right)$.

Now, to be a character of $N/Q$, we require that $Q$ be in the kernel.  As $Q=\textbf{O}_\ell(C_{q-\epsilon})$, we see that this means $\dz(N/Q)$ is comprised of the two characters $\left((1\times \chi_{q-\epsilon}(j))^{(1)}\times\chi_{q-\epsilon}(k)\right)$ and $\left((1\times \chi_{q-\epsilon}(j))^{(-1)}\times\chi_{q-\epsilon}(k)\right)$ of $\irr\left(N|\left((1\times \chi_{q-\epsilon}(j)),\chi_{q-\epsilon}(k)\right)\right)$ for each $(j,k)\in I_{q+\epsilon}\times I_{q+\epsilon}$, together with the character $\left(\widetilde{(\varphi_i\times \chi_{q-\epsilon}(j))}\times\chi_{q-\epsilon}(k)\right)$, for each $i\in I_{q-\epsilon}$ satisfying $\ell^d|i$ and $(j,k)\in I_{q+\epsilon}\times I_{q+\epsilon}$.

The maps for $Q=Q_2$ are as follows:

\scriptsize
\begin{align*}
\Omega_{Q_2}\colon &\scriptsize \left\{\begin{array}{ccc}
    \E_{9}(i)\setminus\{\chi_{29}(i)\}&\multirow{2}{*}{$\rightarrow \irr\left(N|(1\times\chi_{q-\epsilon}(2i)\times\chi_{q-\epsilon}(i))\right)$} & \epsilon=1\\
    \E_{8}(i)\setminus\{\chi_{26}(i)\} && \epsilon=-1
        \end{array}\right.\quad i\in I_{q+\epsilon}\\
&\scriptsize \left\{\begin{array}{ccc}
    \E_{22}(i,j)&\multirow{2}{*}{$\rightarrow \irr\left(N|(1\times\chi_{q-\epsilon}(2i)\times \chi_{q-\epsilon}(j))\right)$} & \epsilon=1\\
    \E_{16}(i,j) && \epsilon=-1
        \end{array}\right.\quad (i,j)\in I_{q+\epsilon}^2\\
& \scriptsize\left\{\begin{array}{ccc}
    \E_{29}(i,j)&\multirow{2}{*}{$\rightarrow \irr\left(N|(\varphi_{2i_1}\times\chi_{q-\epsilon}(2i_2)\times\chi_{q-\epsilon}(j))\right)$} & \epsilon=1\\
    \E_{27}(i,j) && \epsilon=-1
        \end{array}\right.\quad i=i_1(q+\epsilon)+i_2(q-\epsilon)\in I_{q^2-1}, j\in I_{q+\epsilon}\\
\end{align*}
\normalsize

\scriptsize
\begin{align*}
\ast_{Q_2}\colon & \scriptsize\left\{\begin{array}{ccc}
    \mathfrak{B}_{9}(i)^{(0)}&\multirow{2}{*}{$\rightarrow \irr\left(N|(1\times\chi_{q-\epsilon}(2i)\times \chi_{q-\epsilon}(i))\right)$} & \epsilon=1\\
    \mathfrak{B}_{8}(i)^{(0)} && \epsilon=-1
        \end{array}\right.\quad i\in I_{q+\epsilon}\\
&\scriptsize \left\{\begin{array}{ccc}
    \mathfrak{B}_{22}(i,j)&\multirow{2}{*}{$\rightarrow \irr\left(N|(1\times\chi_{q-\epsilon}(2i)\times \chi_{q-\epsilon}(j))\right)$} & \epsilon=1\\
    \mathfrak{B}_{16}(i,j) && \epsilon=-1
        \end{array}\right.\quad (i, j)\in I_{q+\epsilon}^2\\
& \scriptsize\left\{\begin{array}{ccc}
    \mathfrak{B}_{29}(i,j)&\multirow{2}{*}{$\rightarrow \irr\left(N|(\varphi_{2i_1}\times\chi_{q-\epsilon}(2i_2)\times \chi_{q-\epsilon}(j))\right)$} & \epsilon=1\\
    \mathfrak{B}_{27}(i,j) && \epsilon=-1
        \end{array}\right.\quad i=i_1(q+\epsilon)+i_2(q-\epsilon)\in I_{q^2-1}, \ell^d|i; \quad j\in I_{q+\epsilon}\\
\end{align*}
\normalsize
\subsection{$Q=Q_3$}

Let $Q=Q_3$ with $\ell|(q-\epsilon)$.  Then $N=GL^{\epsilon}_3(q):2=C:2$.
Using the character table in CHEVIE \cite{chevie}, we see that when $\ell=3$, there are no characters for $GL^\epsilon_3(q)$ (and hence there are no such characters for $N$) with degree $3^{2d+1}$.  Hence $\dz(N/Q)$ and $\irr_0(N|Q)$ are empty in this case.

For the remainder of our discussion of $Q=Q_3$, we assume $\ell\neq 3$ and $\chi\in\irr(N)$ with $\chi(1)_\ell=\ell^{2d}$.  Then $\chi=\widetilde{\chi_8(i)}$ with $i\in I_{q^3-\epsilon}$.  To be a character of $N/Q$, $\chi$ must be trivial on $Q$, which under our identification is the subgroup $\textbf{O}_\ell(Z(C))$, which consists of representatives of the conjugacy classes $C_1(k)$ for $m|k$ of $GL^\epsilon_3(q)$ in the notation of CHEVIE.  Now, on the class $C_1(k)$ of $GL_3(q)$, the character $\chi_8(i)$ takes the value $(q-1)^2(q+1)\zeta_1^{ik}$, and on the class $C_1(k)$ of $GU_3(q)$, $\chi_8(i)$ takes the value $(q+1)^2(q-1)\xi_1^{ik}$.  
Hence we see that $Q$ is in the kernel of $\chi_8(i)$ exactly when $\ell^d|i$.  So $\dz(N/Q)$ is comprised of the $\frac{q(q+\epsilon)m}{6}$ characters $\widetilde{\chi_8(i)}$ satisfying $i\in I_{q^3-\epsilon}$ with $\ell^d|i$ (i.e., the member of $\irr(N|\chi_8(i))$).

The maps for $Q=Q_3$ are as follows:

\noindent\begin{minipage}[b]{0.45\linewidth}
\scriptsize
\[
\begin{array}{c}\Omega_{Q_3} \\
(\ell\neq 3)\end{array}\colon\scriptsize \left\{\begin{array}{ccc}
    \E_{31}(i)&\multirow{2}{*}{\hspace{-.2cm}$\rightarrow \irr\left(N|\chi_8(i)\right)$} & \hspace{-.1cm}\epsilon=1\\
    \E_{34}(i) && \hspace{-.1cm}\epsilon=-1
        \end{array}\right.\hbox{   } i\in I_{q^3-\epsilon}
\]
\end{minipage}
\hspace{0.6cm}
\begin{minipage}[b]{0.45\linewidth}
\scriptsize
\[
\begin{array}{c}\ast_{Q_3} \\
(\ell\neq 3)\end{array}\colon \scriptsize\left\{\begin{array}{ccc}
    \mathfrak{B}_{31}(i)&\multirow{2}{*}{\hspace{-.2cm}$\rightarrow \irr\left(N|\chi_8(i)\right)$} & \hspace{-.1cm}\epsilon=1\\
    \mathfrak{B}_{34}(i) && \hspace{-.1cm}\epsilon=-1
        \end{array}\right.\hbox{   } i\in I_{q^3-\epsilon}, \ell^d|i
\]
\end{minipage}
\normalsize

%

\subsection{$Q=Q_{1,1}$}
Let $Q=Q_{1,1}$ with $\ell|(q-\epsilon)$.  Then $N=(GL_1^\epsilon(q):2)\wr S_2\times Sp_2(q)$ and $C=GL_1^\epsilon(q)\times GL_1^\epsilon(q)\times Sp_2(q)$ from \prettyref{sec:radsubs}.
Let $L:=\left(GL_1^\epsilon(q):2\right)\times \left(GL_1^\epsilon(q):2\right)\times Sp_2(q)$ and let $\varrho$ denote the action of $S_2$ on $L$, which fixes $Sp_2(q)$ and switches the two copies of $GL_1^\epsilon(q):2=C_{q-\epsilon}:2$. Note that the character $(\varphi\times \varphi'\times \theta)\in\irr(L)=\irr(C_{q-\epsilon}:2)\times \irr(C_{q-\epsilon}:2)\times \irr(Sp_2(q))$ is invariant under $\varrho$ if and only if $\varphi=\varphi'$.  Hence the characters $\chi$ of $N$ with $\chi(1)_\ell=\ell^d$ can be described as follows.

There is a unique such character of $N$ whose restriction to $C$ contains the constituent $(\varphi_i\times\varphi_j\times \chi_{q-\epsilon}(k))$ for $i\neq j\in I_{q-\epsilon}, k\in I_{q+\epsilon}$ (where $\varphi_i\in\irr(GL_1^\epsilon(q))$ is as before), namely $\ind_C^N(\varphi_i\times\varphi_j\times \chi_{q-\epsilon}(k))$.  There are two characters of $N$ with $\chi(1)_\ell=\ell^d$ whose restriction to $C$ contains the constituent $(\varphi_i\times\varphi_i\times \chi_{q-\epsilon}(k))$ or $(\varphi_i\times1\times \chi_{q-\epsilon}(k))$ for $i\in I_{q-\epsilon}, k\in I_{q+\epsilon}$.  Namely, $\irr(N|(\varphi_i\times\varphi_i\times \chi_{q-\epsilon}(k)))$ is comprised of the two extensions to $N$ of the character $(\widetilde{\varphi_i}\times\widetilde{\varphi_i}\times \chi_{q-\epsilon}(k))\in\irr(L)$, which we will denote by $(\widetilde{\varphi_i}\times\widetilde{\varphi_i}\times \chi_{q-\epsilon}(k))^{(1)}$ and $(\widetilde{\varphi_i}\times\widetilde{\varphi_i}\times \chi_{q-\epsilon}(k))^{(-1)}$.  (Recall the notation $\widetilde{\varphi_i}\in\irr(GL_1^\epsilon(q):2)$ from \prettyref{sec:somecharsofsubgps}.)  The set $\irr\left(N|(\varphi_i\times 1\times \chi_{q-\epsilon}(k))\right)$ consists of $\ind_L^N(\widetilde{\varphi_i}\times1^{(1)}\times \chi_{q-\epsilon}(k))$ and $\ind_L^N(\widetilde{\varphi_i}\times1^{(-1)}\times \chi_{q-\epsilon}(k))$.  Finally, there are five such characters of $N$ which have constituent $(1\times1\times\chi_{q-\epsilon}(k))$ on $C$ for each $k\in I_{q+\epsilon}$, so are trivial on $GL_1^\epsilon(q)\times GL_1^\epsilon(q)$.  (These correspond to the five characters of $C_2\wr S_2$, which we will later write as $(1^{(1)}\times 1^{(-1)}), (1^{(1)}\times 1^{(1)})^{(\lambda)}$, and $(1^{(-1)}\times 1^{(-1)})^{(\lambda)}$, where $\lambda\in\{\pm1\}=\irr(C_2)$.)

Since $Q=\textbf{O}_\ell\left((GL_1^\epsilon(q))^2\right)$, to be trivial on $Q$, the characters as listed above must satisfy that in addition, $\ell^d|i$ for all of the $\varphi_i, i\in I_{q-\epsilon}$ occurring in the restriction to $C$ as in the preceding paragraph.  This means that $\dz(N/Q)$ is comprised of the five members of $\irr\left(N|(1\times1\times\chi_{q-\epsilon}(k))\right)$ for each $k\in I_{q+\epsilon}$, the two members of $\irr\left(N|(\varphi_i\times\varphi_i\times \chi_{q-\epsilon}(k))\right)$ for each $i\in I_{q-\epsilon}, k\in I_{q+\epsilon}$ satisfying the additional property $\ell^d|i$, the two members of $\irr\left(N|(\varphi_i\times1\times \chi_{q-\epsilon}(k))\right)$ for each $i\in I_{q-\epsilon}, k\in I_{q+\epsilon}$ satisfying the additional property $\ell^d|i$, and the unique member of $\irr\left(N|(\varphi_i\times\varphi_j\times \chi_{q-\epsilon}(k))\right)$ for each $i\neq j\in I_{q-\epsilon}, k\in I_{q+\epsilon}$ satisfying the additional property $\ell^d|i,j$.

The maps for this radical subgroup are as follows:

\scriptsize
\begin{align*}
\Omega_{Q_{1,1}}\colon & \scriptsize\left\{\begin{array}{ccc}
    \E_{7}(i)\setminus\{\chi_{23}(i)\}&\multirow{2}{*}{$\rightarrow \irr\left(N|(1\times1\times\chi_{q-\epsilon}(i))\right)$} & \epsilon=1\\
    \E_{6}(i)\setminus\{\chi_{14}(i)\} && \epsilon=-1
        \end{array}\right.\quad i\in I_{q+\epsilon}\\
& \scriptsize\left\{\begin{array}{ccc}
    \E_{20}(i,j)&\multirow{2}{*}{$\rightarrow \irr\left(N|(\varphi_i\times1\times\chi_{q-\epsilon}(j))\right)$} & \epsilon=1\\
    \E_{20}(j,i) && \epsilon=-1
        \end{array}\right.\quad i\in I_{q-\epsilon}, j\in I_{q+\epsilon}\\
& \scriptsize\left\{\begin{array}{ccc}
    \E_{18}(i,j)&\multirow{2}{*}{$\rightarrow \irr\left(N|(\varphi_i\times\varphi_i\times\chi_{q-\epsilon}(j))\right)$} & \epsilon=1\\
    \E_{21}(i,j) && \epsilon=-1
        \end{array}\right.\quad i\in I_{q-\epsilon}, j\in I_{q+\epsilon}\\
&\scriptsize \left\{\begin{array}{ccc}
    \E_{26}(i,j,k)&\multirow{2}{*}{$\rightarrow \irr\left(N|(\varphi_i\times\varphi_j\times\chi_{q-\epsilon}(k))\right)$} & \epsilon=1\\
    \E_{28}(k,i,j) && \epsilon=-1
        \end{array}\right.\quad (i, j)\in I_{q-\epsilon}^{2*}, k\in I_{q+\epsilon}\\
\end{align*}
\normalsize
\scriptsize
\begin{align*}
\ast_{Q_{1,1}}\colon & \scriptsize\left\{\begin{array}{ccc}
    \mathfrak{B}_{7}(i)^{(0)}&\multirow{2}{*}{$\rightarrow \irr\left(N|(1\times1\times\chi_{q-\epsilon}(i))\right)$} & \epsilon=1\\
    \mathfrak{B}_{6}(i)^{(0)} && \epsilon=-1
        \end{array}\right.\quad i\in I_{q+\epsilon}\\
&\scriptsize \left\{\begin{array}{ccc}
    \mathfrak{B}_{20}(i,j)&\multirow{2}{*}{$\rightarrow \irr\left(N|(\varphi_i\times1\times\chi_{q-\epsilon}(j))\right)$} & \epsilon=1\\
    \mathfrak{B}_{20}(j,i) && \epsilon=-1
        \end{array}\right.\quad i\in I_{q-\epsilon}, \ell^d|i; \quad j\in I_{q+\epsilon}\\
&\scriptsize \left\{\begin{array}{ccc}
    \mathfrak{B}_{18}(i,j)&\multirow{2}{*}{$\rightarrow \irr\left(N|(\varphi_i\times\varphi_i\times\chi_{q-\epsilon}(j))\right)$} & \epsilon=1\\
    \mathfrak{B}_{21}(i,j) && \epsilon=-1
        \end{array}\right.\quad i\in I_{q-\epsilon}, \ell^d|i; \quad j\in I_{q+\epsilon}\\
&\scriptsize \left\{\begin{array}{ccc}
    \mathfrak{B}_{26}(i,j,k)&\multirow{2}{*}{$\rightarrow \irr\left(N|(\varphi_i\times\varphi_j\times\chi_{q-\epsilon}(k))\right)$} & \epsilon=1\\
    \mathfrak{B}_{28}(k,i,j) && \epsilon=-1
        \end{array}\right. \quad (i, j)\in I_{q-\epsilon}^{2*}, \ell^d|i,j; \quad k\in I_{q+\epsilon}\\
\end{align*}
\normalsize

\subsection{$Q=Q_{2,1}$}

When $Q=Q_{2,1}$, we have $N=\left(GL_2^\epsilon(q):2\right)\times \left(GL_1^\epsilon(q):2\right)$ and $C=GL_2^\epsilon(q)\times GL_1^\epsilon(q)$, which we identify with $GL_1^\epsilon(q)\times SL_2(q)\times GL_1^\epsilon(q)$.
Here characters of $N/Q$ of defect zero or height-zero characters of $N$ with defect group $Q$ satisfy $\chi(1)_{\ell}=\ell^d$.  Hence $\chi$ must be of the form $\left(\varphi\times \varphi'\right)$ where $\varphi\in\irr(GL_2^\epsilon(q):2)$ is as in the case $Q=Q_2$ and $\varphi'$ is any member of $\irr(C_{q-\epsilon}:2)$ as described in \prettyref{sec:somecharsofsubgps}.  That is, these characters include the unique member, $\left(\widetilde{(\varphi_i\times \chi_{q-\epsilon}(j))}\times \widetilde{\varphi_k}\right)$, of $\irr(N)$ lying above $(\varphi_i\times\chi_{q-\epsilon}(j)\times\varphi_k)\in\irr(C)$ for each $i,k\in I_{q-\epsilon}, j\in I_{q+\epsilon}$ (with the possibility $i=k$), the two members of $\irr(N)$ lying above $(\varphi_i\times\chi_{q-\epsilon}(j)\times 1)\in\irr(C)$ for each $i\in I_{q-\epsilon}, j\in I_{q+\epsilon}$ (namely, $\left(\widetilde{(\varphi_i\times \chi_{q-\epsilon}(j))}\times 1^{(\nu)}\right)$ for each $\nu\in\{\pm1\}$), the two members of $\irr(N)$ lying above $(1\times\chi_{q-\epsilon}(j)\times \varphi_i)\in\irr(C)$ for each $i\in I_{q-\epsilon}, j\in I_{q+\epsilon}$ (namely, $\left((1\times \chi_{q-\epsilon}(j))^{(\nu)}\times \widetilde{\varphi_i}\right)$ for each $\nu\in\{\pm1\}$), and the four characters of $\irr(N|(1\times \chi_{q-\epsilon}(j)\times 1))$ for each $j\in I_{q+\epsilon}$, namely $((1\times \chi_{q-\epsilon}(j))^{(\nu)}\times 1^{(\mu)})$ for each $(\nu,\mu)\in\{\pm1\}\times \{\pm1\}$.  (Recall this notation for characters of $GL_2^\epsilon(q):2$ discussed in \prettyref{sec:somecharsofsubgps}.)

To be trivial on $Q$, we again only further require $\ell^d|i$ for any $\varphi_i$ occurring in the restriction to $C$.  That is, $\dz(N/Q)$ consists of $\irr(N|(\varphi_i\times \chi_{q-\epsilon}(j)\times \varphi_k))$, $\irr(N|(\varphi_i\times \chi_{q-\epsilon}(j)\times 1))$, $\irr(N|(1\times \chi_{q-\epsilon}(j)\times \varphi_k))$, and $\irr(N|(1\times \chi_{q-\epsilon}(j)\times 1))$ where $i,k\in I_{q-\epsilon}, j\in I_{q+\epsilon}$ satisfy the additional property that $\ell^d|i,k$.

The maps in this case are below:

\scriptsize
\begin{align*}
\Omega_{Q_{2,1}}\colon &\scriptsize \left\{\begin{array}{ccc}
    \E_{13}(i)&\multirow{2}{*}{$\rightarrow \irr\left(N|(1\times\chi_{q-\epsilon}(2i)\times1)\right)$} & \epsilon=1\\
    \E_{11}(i)&& \epsilon=-1
        \end{array}\right.\quad i\in I_{q+\epsilon}\\
&\scriptsize \left\{\begin{array}{ccc}
    \E_{21}(i,j)&\multirow{2}{*}{$\rightarrow \irr\left(N|(1\times\chi_{q-\epsilon}(2i)\times\varphi_j)\right)$} & \epsilon=1\\
    \E_{18}(i,j)&& \epsilon=-1
        \end{array}\right.\quad i\in I_{q+\epsilon}, j\in I_{q-\epsilon}\\
& \scriptsize  \E_{19}(i)\rightarrow \irr\left(N|(\varphi_{2i_1}\times\chi_{q-\epsilon}(2i_2)\times1)\right)\quad\hbox{ (for $\epsilon=1$ or $-1$)},\quad i=i_1(q+\epsilon)+i_2(q-\epsilon)\in I_{q^2-1}\\
&\scriptsize \left\{\begin{array}{ccc}
    \E_{27}(i,j)&\multirow{2}{*}{$\rightarrow \irr\left(N|(\varphi_{2i_1}\times\chi_{q-\epsilon}(2i_2)\times\varphi_j)\right)$} & \epsilon=1\\
    \E_{29}(i,j)&& \epsilon=-1
        \end{array}\right.\quad i=i_1(q+\epsilon)+i_2(q-\epsilon)\in I_{q^2-1}, j\in I_{q-\epsilon}\\
\end{align*}
\scriptsize
\begin{align*}
\ast_{Q_{2,1}}\colon &\scriptsize \left\{\begin{array}{ccc}
    \mathfrak{B}_{13}(i)&\multirow{2}{*}{$\rightarrow \irr\left(N|(1\times\chi_{q-\epsilon}(2i)\times1)\right)$} & \epsilon=1\\
    \mathfrak{B}_{11}(i)&& \epsilon=-1
        \end{array}\right.\quad i\in I_{q+\epsilon}\\
&\scriptsize \left\{\begin{array}{ccc}
    \mathfrak{B}_{21}(i,j)&\multirow{2}{*}{$\rightarrow \irr\left(N|(1\times\chi_{q-\epsilon}(2i)\times\varphi_j)\right)$} & \epsilon=1\\
    \mathfrak{B}_{18}(i,j)&& \epsilon=-1
        \end{array}\right.\quad i\in I_{q+\epsilon};\quad j\in I_{q-\epsilon}, \ell^d|j\\
&\scriptsize   \mathfrak{B}_{19}(i)\rightarrow \irr\left(N|(\varphi_{2i_1}\times\chi_{q-\epsilon}(2i_2)\times1)\right)\quad\hbox{ (for $\epsilon=1$ or $-1$)},\quad i=i_1(q+\epsilon)+i_2(q-\epsilon)\in I_{q^2-1}, \ell^d|i\\
&\scriptsize \left\{\begin{array}{ccc}
    \mathfrak{B}_{27}(i,j)&\multirow{2}{*}{$\rightarrow \irr\left(N|(\varphi_{2i_1}\times\chi_{q-\epsilon}(2i_2)\times\varphi_j)\right)$} & \epsilon=1\\
    \mathfrak{B}_{29}(i,j)&& \epsilon=-1
        \end{array}\right.\quad i=i_1(q+\epsilon)+i_2(q-\epsilon)\in I_{q^2-1}; \quad j\in I_{q-\epsilon};\quad \ell^d|i,j\\
\end{align*}
\normalsize

\subsection{$Q=Q_{1,1,1}$}\label{sec:Q111}

Let $Q=Q_{1,1,1}$, with $\ell|(q-\epsilon)$, so $N=(GL_1^\epsilon(q):2)\wr S_3$ and $C=\left(GL_1^\epsilon(q)\right)^3$.  Also, write $L:=\left(GL_1^\epsilon(q):2\right)^3$ to denote the normal subgroup of $N$ with quotient $S_3$.
Let $\theta=(\varphi\times \varphi'\times \varphi'')\in \irr(L)=\irr(C_{q-\epsilon}:2)^3$ be a constituent of $\chi\in\irr(N)$ when restricted to $L$.  Then $\theta$ is invariant under the $S_3$ action 
if and only if $\varphi=\varphi'=\varphi''$. In this case, $\theta$ extends to a character of $N$ (this can be seen, for example, from \cite[(11.31) and (6.20)]{isaacs})
and we get three such characters, corresponding to the three characters of the quotient $N/L=S_3$, by Gallagher's theorem, with degrees $\theta(1), \theta(1),$ and $2\theta(1)$.  In terms of the restriction of characters to $C\cong C_{q-\epsilon}^3$, this yields three characters in $\irr(N|(\varphi_i\times\varphi_i\times\varphi_i))$ for each $i\in I_{q-\epsilon}$, where $\varphi_i\in\irr(C_{q-\epsilon})$ is as above, and six of the characters in $\irr(N|(1\times 1\times 1))$, due to the extensions $(1^{(-1)}\times 1^{(-1)}\times 1^{(-1)})$ and $(1^{(1)}\times 1^{(1)}\times 1^{(1)})\in\irr(L|(1\times1\times1))$.  (It may be worthwhile to note here that $\irr(N|(\varphi_i\times\varphi_i\times\varphi_i))=\irr(N|(\varphi_{-i}\times\varphi_{-i}\times\varphi_{-i}))=\irr(N|(\varphi_i\times\varphi_{-i}\times\varphi_i))$.) 

Moreover, $\theta=(\varphi\times \varphi'\times \varphi'')\in \irr(L)$ has a stabilizer $T:=N_\theta$ in $N$ with $|T/L|=2$ precisely when exactly two of $\varphi$, $\varphi'$, and $\varphi''$ are the same.  In this case, we get two extensions to $T$, and the character $\chi$ of $N$ is determined by a constituent on $T$.  Let $\varpi$ denote a $3$-cycle in $S_3$.  Then the two characters in $\irr(N|\theta)$ have restriction to $L$ as $\theta+\theta^{\varpi}+\theta^{\varpi^2}$ and have degree $3\theta(1)$.  In terms of constituents on $C$, this yields two characters in $\irr(N|(\varphi_i\times\varphi_i\times\varphi_j))$ for each $(i,j)\in I_{q-\epsilon}^2$, four characters in $\irr(N|(\varphi_i\times\varphi_i\times 1))$ for each $i\in I_{q-\epsilon}$ (due to the choices $(\widetilde{\varphi_i}\times \widetilde{\varphi_i}\times 1^{(1)}), (\widetilde{\varphi_i}\times \widetilde{\varphi_i}\times 1^{(-1)}) \in\irr(L|(\varphi_i\times\varphi_i\times 1))$), four of the characters in $\irr(N|(\varphi_i\times1\times 1))$  for each $i\in I_{q-\epsilon}$ (due to the choices $(\widetilde{\varphi_i}\times 1^{(1)}\times 1^{(1)}), (\widetilde{\varphi_i}\times 1^{(-1)}\times 1^{(-1)}) \in\irr(L|(\varphi_i\times1\times 1))$), and four more characters in $\irr(N|(1\times 1\times 1))$ (due to the extensions $(1^{(-1)}\times 1^{(-1)}\times 1^{(1)})$ and $(1^{(1)}\times 1^{(1)}\times 1^{(-1)})\in\irr(L|(1\times 1\times1))$).  Note that there are therefore a total of ten characters in $\irr(N|(1\times 1\times 1))$.  (Also, note here that $\irr(N|(1^{(-1)}\times 1^{(-1)}\times 1^{(1)}))=\irr(N|(1^{(-1)}\times 1^{(1)}\times 1^{(-1)})=\irr(N|(1^{(1)}\times 1^{(-1)}\times 1^{(-1)})$ and $\irr(N|(\varphi_i\times\varphi_i\times 1))=\irr(N|(\varphi_i\times\varphi_{-i}\times 1))=\irr(N|(1\times\varphi_i\times\varphi_i))$, and so on.)

 Finally, if $\varphi\neq\varphi'\neq\varphi''$, then the irreducible character $\theta=(\varphi\times\varphi'\times\varphi'')$ of $L$ has stabilizer $N_\theta=L$.  Hence such a character is uniquely determined by a constituent $\theta$ on $L$, restricts to $L$ as $\sum_{\rho\in S_3}\theta^\rho$, and has degree $6\theta(1)$.  In terms of constituents on $C$, this yields a unique member of $\irr(N|(\varphi_i\times\varphi_j\times\varphi_k))$ for each $(i,j,k)\in I_{q-\epsilon}^{3*}$, two members of $\irr(N|(\varphi_i\times\varphi_j\times 1))$ for each $(i,j)\in I_{q-\epsilon}^{2*}$, and the final character in $\irr(N|(\varphi_i\times1\times 1))$ for each $i\in I_{q-\epsilon}$ (due to the character $(\widetilde{\varphi_i}\times 1^{(1)}\times 1^{(-1)}) \in\irr(L|(\varphi_i\times1\times 1))$).  Note that this makes a total of five characters in each $\irr(N|(\varphi_i\times1\times 1))$.


 Now, if $\ell\neq 3$, then every character of $N$ satisfies $\chi(1)_\ell=1$, as desired.
 If $\ell=3$, all except those characters whose restriction $\chi|_L$ to $L$ has a constituent $(\varphi\times\varphi\times\varphi)$ satisfy the condition $\chi(1)_3=3$.  Now, to be trivial on $Q$, we again just need to further require that $\ell^d|i$ for any $\varphi_i$ appearing in the restriction to a copy of $C_{q-\epsilon}$.

 The maps $\Omega_{Q_{1,1,1}}$ are below.  Note that we distinguish between the cases $\ell=3$ and $\ell\neq 3$.

\noindent\begin{minipage}[b]{0.45\linewidth}
\tiny
\begin{align*}
\begin{array}{c}\Omega_{Q_{1,1,1}} \\
(\ell\neq 3)\end{array}\colon & \scriptsize\left\{\begin{array}{ccc}
    \E_{1}\setminus\{\chi_5,\chi_{11}\}&\multirow{2}{*}{\hspace{-.1cm}$\rightarrow \irr\left(N|(1\times1\times1)\right)$} & \hspace{-.1cm}\epsilon=1\\
    \E_{1}\setminus\{\chi_4,\chi_{9}\}&& \hspace{-.1cm}\epsilon=-1
        \end{array}\right.\\
& \scriptsize\left\{\begin{array}{ccc}
    \E_{6}(i)\setminus\{\chi_{17}(i)\}&\multirow{2}{*}{\hspace{-.1cm}$\rightarrow \irr\left(N|(\varphi_i\times1\times1)\right)$} & \hspace{-.1cm}\epsilon=1\\
    \E_{7}(i)\setminus\{\chi_{20}(i)\}&& \hspace{-.1cm}\epsilon=-1
        \end{array}\right.\\&\quad i\in I_{q-\epsilon}\\
&\scriptsize \left\{\begin{array}{ccc}
    \E_{8}(i)&\multirow{2}{*}{\hspace{-.1cm}$\rightarrow \irr\left(N|(\varphi_i\times\varphi_i\times\varphi_i)\right)$} & \hspace{-.1cm}\epsilon=1\\
    \E_{9}(i)&&\hspace{-.1cm} \epsilon=-1
        \end{array}\right.\quad i\in I_{q-\epsilon}\\
& \scriptsize\left\{\begin{array}{ccc}
    \E_{11}(i)&\multirow{2}{*}{\hspace{-.1cm}$\rightarrow \irr\left(N|(\varphi_i\times\varphi_i\times1)\right)$} & \hspace{-.1cm}\epsilon=1\\
    \E_{13}(i)&& \hspace{-.1cm}\epsilon=-1
        \end{array}\right.\quad i\in I_{q-\epsilon}\\
& \scriptsize\left\{\begin{array}{ccc}
    \E_{17}(i,j)&\multirow{2}{*}{\hspace{-.1cm}$\rightarrow \irr\left(N|(\varphi_i\times\varphi_j\times1)\right)$} &\hspace{-.1cm} \epsilon=1\\
    \E_{23}(i,j)&&\hspace{-.1cm} \epsilon=-1
        \end{array}\right.\quad (i, j)\in I_{q-\epsilon}^{2*}\\
&\scriptsize \left\{\begin{array}{ccc}
    \E_{16}(i,j)&\multirow{2}{*}{\hspace{-.1cm}$\rightarrow \irr\left(N|(\varphi_i\times\varphi_i\times\varphi_j)\right)$} & \hspace{-.1cm}\epsilon=1\\
    \E_{22}(i,j)&& \hspace{-.1cm}\epsilon=-1
        \end{array}\right.\\&\quad (i, j)\in I_{q-\epsilon}^2\\
&\scriptsize \left\{\begin{array}{ccc}
    \E_{25}(i,j,k)&\multirow{2}{*}{\hspace{-.1cm}$\rightarrow \irr\left(N|(\varphi_i\times\varphi_j\times\varphi_k)\right)$} &\hspace{-.1cm} \epsilon=1\\
    \E_{32}(i,j,k)&& \hspace{-.1cm}\epsilon=-1
        \end{array}\right.\\&\quad (i,j,k)\in I_{q-\epsilon}^{3*}\\
\end{align*}
\end{minipage}
\hspace{-1.2cm}
\begin{minipage}[b]{0.45\linewidth}

\tiny
\begin{align*}
\begin{array}{c}\Omega_{Q_{1,1,1}} \\
(\ell= 3)\end{array}\colon & \scriptsize\left\{\begin{array}{ccc}
    \E_{6}(i)\setminus\{\chi_{17}(i)\}&\multirow{2}{*}{\hspace{-.1cm}$\rightarrow \irr\left(N|(\varphi_i\times1\times1)\right)$} & \hspace{-.1cm}\epsilon=1\\
    \E_{7}(i)\setminus\{\chi_{20}(i)\}&& \hspace{-.1cm}\epsilon=-1
        \end{array}\right.\\&\quad i\in I_{q-\epsilon}, m\not|\,i\\
& \scriptsize\left\{\begin{array}{ccc}
    \E_{11}(i)&\multirow{2}{*}{\hspace{-.1cm}$\rightarrow \irr\left(N|(\varphi_i\times\varphi_i\times1)\right)$} & \hspace{-.1cm}\epsilon=1\\
    \E_{13}(i)&& \hspace{-.1cm}\epsilon=-1
        \end{array}\right.\\&\quad i\in I_{q-\epsilon}, m\not|\,i\\
&\scriptsize \left\{\begin{array}{ccc}
    \E_{17}(i,j)&\multirow{2}{*}{\hspace{-.1cm}$\rightarrow \irr\left(N|(\varphi_i\times\varphi_j\times1)\right)$} & \hspace{-.1cm}\epsilon=1\\
    \E_{23}(i,j)&& \hspace{-.1cm}\epsilon=-1
        \end{array}\right.\\
&\scriptsize \quad (i, j)\in I_{q-\epsilon}^{2*}; \hbox{\scriptsize $m$ does not divide one of $i, j$}\\
&\scriptsize \left\{\begin{array}{ccc}
    \E_{16}(i,j)&\multirow{2}{*}{\hspace{-.1cm}$\rightarrow \irr\left(N|(\varphi_i\times\varphi_i\times\varphi_j)\right)$} & \hspace{-.1cm}\epsilon=1\\
    \E_{22}(i,j)&& \hspace{-.1cm}\epsilon=-1
        \end{array}\right.\\&\quad (i, j)\in I_{q-\epsilon}^2; \scriptsize i\not\equiv\pm j\mod m\\
&\scriptsize \left\{\begin{array}{ccc}
    \E_{25}(i,j,k)&\multirow{2}{*}{\hspace{-.1cm}$\rightarrow \irr\left(N|(\varphi_i\times\varphi_j\times\varphi_k)\right)$} & \hspace{-.1cm}\epsilon=1\\
    \E_{32}(i,j,k)&& \hspace{-.1cm}\epsilon=-1
        \end{array}\right.\\
&\scriptsize\quad (i,j,k)\in I_{q-\epsilon}^{3*}; i\not\equiv\pm j\not\equiv\pm k\not\equiv \pm i\mod m\\
\end{align*}
\end{minipage}
\normalsize

In the case that $\ell=3$, we have excluded the cases when $m$ divides all indices, since then the $\E_i(J)$ given above actually lie in the principal block, so have defect group $P$.  Similarly, if $m$ divides $i$, then $\varphi_i$ maps a $q-\epsilon$ root of unity to an $\ell^d$ root, so if $m$ divides all indices, then the image $\omega_\theta^\ast$ in $\overline{\F}_\ell^\times$ of the central character for the character $\theta$ of $N$ is the same as that for the principal character $1_N$.  Hence they lie in the same block and $\theta$ has defect group $P$, which is also a Sylow $\ell$-subgroup of $N$.


Moreover, in the cases of $\E_{16}(i,j), \E_{25}(i,j,k)$ (resp. $\E_{22}(i,j), \E_{32}(i,j,k)$) (for $3|(q-1)$, resp. $3|(q+1)$), we must also exclude any case where the indices are all equivalent (but nonzero) modulo $m$, as then these series lie in the block $B_{8}(k)$ (resp. $B_9(k)$) for some $k\in I_{q-\epsilon}$ with $3^d|k$, which also has defect group $P$.

Suppose $\ell\neq 3$ so that $Q_{1,1,1}\in\mathrm{Syl}_\ell(G)$.  The map $\ast_{Q_{1,1,1}}$ in this case is:
\scriptsize
\begin{align*}
\begin{array}{c}\ast_{Q_{1,1,1}} \\
(\ell\neq 3)\end{array} \colon&\scriptsize  \mathfrak{B}_{0}\rightarrow \irr\left(N|(1\times1\times1)\right) \\
 & \scriptsize\left\{\begin{array}{ccc}
    \mathfrak{B}_{6}(i)^{(0)}&\multirow{2}{*}{$\rightarrow \irr\left(N|(\varphi_i\times1\times1)\right)$} & \epsilon=1\\
    \mathfrak{B}_{7}(i)^{(0)}&& \epsilon=-1
        \end{array}\right.\quad i\in I_{q-\epsilon}, \ell^d|i\\
&\scriptsize \left\{\begin{array}{ccc}
    \mathfrak{B}_{8}(i)&\multirow{2}{*}{$\rightarrow \irr\left(N|(\varphi_i\times\varphi_i\times\varphi_i)\right)$} & \epsilon=1\\
    \mathfrak{B}_{9}(i)&& \epsilon=-1
        \end{array}\right.\quad i\in I_{q-\epsilon}, \ell^d|i\\
& \scriptsize\left\{\begin{array}{ccc}
    \mathfrak{B}_{11}(i)&\multirow{2}{*}{$\rightarrow \irr\left(N|(\varphi_i\times\varphi_i\times1)\right)$} & \epsilon=1\\
    \mathfrak{B}_{13}(i)&& \epsilon=-1
        \end{array}\right.\quad i\in I_{q-\epsilon}, \ell^d|i\\
&\scriptsize \left\{\begin{array}{ccc}
    \mathfrak{B}_{17}(i,j)&\multirow{2}{*}{$\rightarrow \irr\left(N|(\varphi_i\times\varphi_j\times1)\right)$} & \epsilon=1\\
    \mathfrak{B}_{23}(i,j)&& \epsilon=-1
        \end{array}\right.\quad (i, j)\in I_{q-\epsilon}^{2*}, \ell^d|i,j\\
&\scriptsize \left\{\begin{array}{ccc}
    \mathfrak{B}_{16}(i,j)&\multirow{2}{*}{$\rightarrow \irr\left(N|(\varphi_i\times\varphi_i\times\varphi_j)\right)$} & \epsilon=1\\
    \mathfrak{B}_{22}(i,j)&& \epsilon=-1
        \end{array}\right.\quad (i, j)\in I_{q-\epsilon}^2, \ell^d|i,j\\
&\scriptsize \left\{\begin{array}{ccc}
    \mathfrak{B}_{25}(i,j,k)&\multirow{2}{*}{$\rightarrow \irr\left(N|(\varphi_i\times\varphi_j\times\varphi_k)\right)$} & \epsilon=1\\
    \mathfrak{B}_{32}(i,j,k)&& \epsilon=-1
        \end{array}\right.\quad (i,j,k)\in I_{q-\epsilon}^{3*};\quad \ell^d|i,j,k\\
\end{align*}
\normalsize

Now suppose that $\ell=3$.  Then we will distribute $\mathfrak{B}_0$ between the three sets $\ibr_3(G|Q_{1,1,1})$, $\ibr_3(G|P)$, and $\ibr_3(G|R)$.  Of the $10$ Brauer characters in the block $B_0$, we will require that four of these belong to $\ibr_3(G|Q_{1,1,1})$, 
another four belong to $\ibr_3(G|P)$, 
 and the final two belong to $\ibr_3(G|R)$.  
In fact, the choice of this partition is arbitrary, as long as the number of characters assigned to each subgroup is correct, so we will simply write $\mathfrak{B}_0=\mathfrak{B}_0(Q_{1,1,1})\cup \mathfrak{B}_0(P)\cup \mathfrak{B}_0(R)$ for an appropriate partition.   Below is the corresponding map for $Q_{1,1,1}$.  Those for $P$ and $R$ will follow in the next two subsections.


\scriptsize
\begin{align*}
\begin{array}{c}\ast_{Q_{1,1,1}} \\
(\ell= 3)\end{array} \colon&\scriptsize  \mathfrak{B}_{0}(Q_{1,1,1})\rightarrow \irr\left(N|(1^{(\nu)}\times1^{(\nu)}\times1^{(\mu)})\right),\quad \nu\neq\mu\in\{\pm1\} \\
 &\scriptsize \left\{\begin{array}{ccc}
    \mathfrak{B}_{6}(i)^{(0)}&\multirow{2}{*}{$\rightarrow \irr\left(N|(\varphi_i\times1\times1)\right)$} & \epsilon=1\\
    \mathfrak{B}_{7}(i)^{(0)}&& \epsilon=-1
        \end{array}\right.\quad i\in I_{q-\epsilon}, 3^d|i\\
&\scriptsize \left\{\begin{array}{ccc}
    \mathfrak{B}_{11}(i)&\multirow{2}{*}{$\rightarrow \irr\left(N|(\varphi_i\times\varphi_i\times1)\right)$} & \epsilon=1\\
    \mathfrak{B}_{13}(i)&& \epsilon=-1
        \end{array}\right.\quad i\in I_{q-\epsilon}, 3^d|i\\
&\scriptsize \left\{\begin{array}{ccc}
    \mathfrak{B}_{17}(i,j)&\multirow{2}{*}{$\rightarrow \irr\left(N|(\varphi_i\times\varphi_j\times1)\right)$} & \epsilon=1\\
    \mathfrak{B}_{23}(i,j)&& \epsilon=-1
        \end{array}\right.\quad (i, j)\in I_{q-\epsilon}^{2*}, 3^d|i,j\\
&\scriptsize \left\{\begin{array}{ccc}
    \mathfrak{B}_{16}(i,j)&\multirow{2}{*}{$\rightarrow \irr\left(N|(\varphi_i\times\varphi_i\times\varphi_j)\right)$} & \epsilon=1\\
    \mathfrak{B}_{22}(i,j)&& \epsilon=-1
        \end{array}\right.\quad (i, j)\in I_{q-\epsilon}^2, 3^d|i,j\\
& \scriptsize\left\{\begin{array}{ccc}
    \mathfrak{B}_{25}(i,j,k)&\multirow{2}{*}{$\rightarrow \irr\left(N|(\varphi_i\times\varphi_j\times\varphi_k)\right)$} & \epsilon=1\\
    \mathfrak{B}_{32}(i,j,k)&& \epsilon=-1
        \end{array}\right.\quad (i,j,k)\in I_{q-\epsilon}^{3*},\quad 3^d|i,j,k\\
\end{align*}
\normalsize

Note that for the image of $\mathfrak{B}_0(Q_{1,1,1})$, we have used the notation for the constituent when restricted to $L$ rather than $C$.

\subsection{$Q=P$}

Now suppose $Q=P$ with $\ell=3|(q-\epsilon)$.  Direct calculation and arguments analogous to those in \cite{An94} show that \cite[Formula (2.5)]{An94} holds for $Q$, and we have $N/Q\cong \left(C_m:2\right)\times C_2$.  Certainly, any character of $N/Q$ has defect zero since $P$ is a Sylow $3$-subgroup of $G$.  We can view $C_m:2$ as a quotient of $C_{q-\epsilon}:2$
and as such, the characters of $C_m:2$ are of the form $1^{(1)}, 1^{(-1)},$ and $\widetilde{\varphi_i}=\ind_{C_{q-\epsilon}}^{C_{q-\epsilon}:2}\varphi_i$ as before, where $3^d|i$.  For the $C_2$ factor, let $\langle\lambda\rangle=\irr(C_2)$.

Write $P_1:=Q_{1,1,1}$, and note that $P=P_1\rtimes C_3$.  Then $C\leq C_G(P_1)=(C_{q-\epsilon})^3$, and since $C$ must commute with the $C_3$-action, we see that in fact $C\cong C_{q-\epsilon}$ is the subgroup consisting of $(x,x,x)\in (C_{q-\epsilon})^3$ for $x\in C_{q-\epsilon}$.  Now, $N$ is as in \cite[Formula (2.5)]{An94} and can be written as a semidirect product $K\rtimes S_3$, where $K\leq \left(C_{q-\epsilon}:2\right)^3$ is comprised of elements $(X,Y,Z)$ where $X,Y,Z\in C_{q-\epsilon}:2$ belong to the same coset modulo $C_{3^d}$.  Let $\varphi\in\irr_0(N|P)=\irr_{3'}(N)$.  Since $P_1$ is normal in $N$, we know by Clifford theory that $\varphi|_{P_1}$ is the sum of $N$-conjugates of some $(\mu^i\times\mu^j\times\mu^k)$, where $\mu$ is the character of $C_{3^d}$ that sends a fixed generator to a fixed primitive $3^d$ root of unity in $\C$ and $0\leq i,j,k\leq 3^d-1$.  If the $i,j,k$ are not all the same, then the $S_3$ action will cause the number of distinct conjugates in this decomposition to be a multiple of $3$, and hence $\varphi$ will have degree divisible by $3$, contradicting the fact that $\varphi$ has height zero.  Hence an irreducible constituent of the restriction of $\varphi$ to $P_1$ is of the form $\theta_i:=(\mu^i\times\mu^i\times\mu^i)$ for some $0\leq i\leq 3^d-1$, with $\mu$ as above.

Now, we can write $K=(P_1\times C_m):2$ (here $C_m\leq C$), and let $J:=P_1\times C_m$ be the index-2 subgroup.  
The extensions of $\theta_i$ to $J$ are of the form $\theta_i\times\delta^j$ for $0\leq j\leq m-1$, where $\delta\in\irr(C_m)$  represents the character sending a fixed generator of $C_m$ to a fixed primitive $m$th root of unity in $\C$.  We will write $\theta(i,j):=\theta_i\times\delta^j$ for this character of $J$.  Note that the action of $K/J\cong C_2$ sends $\theta(i,j)$ to $\theta(-i,-j)$. Moreover, each $\theta(i, j)$ is invariant under the $S_3$ action, so extends to $J\rtimes S_3$.  (This can be seen, for example, from \cite[(11.31) and (6.20)]{isaacs}.)


From here, we see that if $\varphi|_J$ contains a nontrivial constituent, then $\varphi|_K$ is uniquely determined by a constituent $\theta(i,j)$ on $J$ with $(i,j)\neq(0,0)$, and for each such choice of $(i,j)$ there are $3$ characters $\varphi\in\irr(N|\theta(i,j))=\irr(N|\theta(-i,-j))$, by Gallagher's theorem.  Moreover, there are $6$ characters $\varphi\in\irr_0(N|P)$ with $J\leq\ker\varphi$, also by Gallagher's theorem.  

To describe the images of the map $\Omega_P$, we use the constituents when restricted to $J$ rather than $C$.  For the map $\ast_P$, we use the notation of characters of $N/P$ as in the first paragraph of this section.  Recall from \prettyref{sec:Q111} the partition $\mathfrak{B}_0=\mathfrak{B}_0(Q_{1,1,1})\cup \mathfrak{B}_0(P)\cup \mathfrak{B}_0(R)$ with four of the Brauer characters of $\mathfrak{B}_0$ lying in $\mathfrak{B}_0(P)$.  Similarly, of the three Brauer characters of the block $B_8(i)$ with $3^d|i$ (resp. $B_9(i)$) when $\ell|(q-1)$ (resp. $\ell|(q+1)$), we require that two of these are members of $\ibr_3(G|P)$ and the other is a member of $\ibr_3(G|R)$.  Again, the partition is arbitrary, and we will write $\mathfrak{B}_8(i)=\mathfrak{B}_8(i,P)\cup \mathfrak{B}_8(i,R)$ (resp. $\mathfrak{B}_9(i)=\mathfrak{B}_9(i,P)\cup \mathfrak{B}_9(i,R)$) for an appropriate partition.  The maps for $P$ are as follows: 

\noindent\begin{minipage}[b]{0.45\linewidth}
\tiny
\begin{align*}
\begin{array}{c}\Omega_{P} \\
(\ell= 3)\end{array}\colon &\scriptsize \left\{\begin{array}{ccc}
    \{\chi_1,\chi_3,\chi_4,\chi_9,\chi_{10}, \chi_{12}\}&\multirow{2}{*}{\hspace{-.1cm}$\rightarrow \irr\left(N|1_J\right)$} & \hspace{-.1cm}\epsilon=1\\
    \{\chi_1,\chi_2,\chi_5,\chi_8,\chi_{11}, \chi_{12}\}&& \hspace{-.1cm}\epsilon=-1
        \end{array}\right.\\
& \scriptsize\left\{\begin{array}{ccc}
    \E_{8}(i)&\multirow{2}{*}{\hspace{-.1cm}$\rightarrow \irr\left(N|\theta(i\hspace{-.15cm}\mod 3^d, 3(i\hspace{-.15cm}\mod m))\right)$} & \hspace{-.1cm}\epsilon=1\\
    \E_{9}(i)&&\hspace{-.1cm} \epsilon=-1
        \end{array}\right.\\&\quad i\in I_{q-\epsilon}\\
\end{align*}
\end{minipage}
\begin{minipage}[b]{0.45\linewidth}
\tiny
\begin{align*}
\begin{array}{c}\ast_{P} \\
(\ell= 3)\end{array} \colon& \scriptsize \mathfrak{B}_{0}(P)\rightarrow \{(1^{(\nu)}\times\lambda^{\mu})\colon  \nu, \mu\in\{\pm1\} \}\\
&\scriptsize\left\{\begin{array}{ccc}
    \mathfrak{B}_{8}(i, P)&\multirow{2}{*}{\hspace{-.1cm}$\rightarrow \{(\widetilde{\varphi}_{3i}\times\lambda^{\mu})\colon \mu\in\{\pm1\}\}$} & \hspace{-.1cm}\epsilon=1\\
    \mathfrak{B}_{9}(i, P)&& \hspace{-.1cm}\epsilon=-1
        \end{array}\right.\\&\quad i\in I_{q-\epsilon}, 3^d|i\\
\end{align*}
\end{minipage}
\normalsize

\subsection{$Q=R$}
Let $\ell=3|(q-\epsilon)$ with $(q-\epsilon)_{3'}=m, (q-\epsilon)_3=3^d$, and let $Q=R$ be the group $Z\cdot E_{27}\leq GL_3^\epsilon(q)$ viewed as a subgroup of $G$ 
as in \prettyref{sec:radsubs}.  Then arguing as in \cite[(1C)]{An94}, we see $N:=N_G(R)$ has an index two subgroup $N^\circ$ satisfying $N/N^\circ\cong N_G(Z)/C_G(Z)=C_2$.  Further, $R\lhd N^\circ$, and we have $N/R=(N^\circ/R).2$, with the order-$2$ automorphism given by the action of 
the map $\tau\colon A\mapsto (A^T)^{-1}$ on $GL_3^\epsilon(q)$.  Also, $Sp_2(3)\cong N^\circ/(RZ(N^\circ))\cong (N^\circ/R)/(RZ(N^\circ)/R)$, so $N^\circ/R$ contains a quotient group isomorphic to $Sp_2(3)$.  Moreover, each linear character of $RZ(N^\circ)/R\cong Z(N^\circ)/(R\cap Z(N^\circ))= Z(N^\circ)/\textbf{O}_3(Z(N^\circ))\cong C_m$ is extendable to a character of $N^\circ/R$ (again arguing as in \cite[(1C)]{An94}).  Hence by Gallagher's theorem, the characters of $N^\circ/R$ are exactly the characters $\tilde{\theta}\varsigma$ with $\tilde{\theta}$ the extension of $\theta\in\irr(RZ(N^\circ)/R)=\irr(C_m)$ to $N^\circ$ and $\varsigma\in\irr(N^\circ/(RZ(N^\circ))=\irr(Sp_2(3))$.

Since $|N/R|=2m|Sp_2(3)|$, we have that a defect-zero character of $N/R$ will have $\chi(1)_{3}=3$.  Since $N^\circ/R$ has index 2 in $N/R$, the constituents of the restriction of $\chi$ to $N^\circ/R$ must satisfy this degree condition as well, so we require that $\varsigma$ have degree divisible by $3$.  Since $Sp_2(3)$ has exactly one such character (namely, the Steinberg character, of degree $3$), we will henceforth use $\varsigma$ to denote this Steinberg character.  Note that $\varsigma$ is invariant under the action of $\tau$, and that as before, the principal character is the only character of $C_m$ invariant under $\tau$.

This yields $\frac{m-1}{2}+2$ defect-zero characters of $N/R$, which we will denote by $1^{(1)}\varsigma, 1^{(-1)}\varsigma$, and 
$\widetilde{\varphi_i}\varsigma$ for $i\in I_{q-\epsilon}$ with $3^d|i$, where by an abuse of notation, $\widetilde{\varphi_i}\varsigma$ represents the defect-zero character whose restriction to $RZ(N^\circ)/R\cong C_m$ contains $\varphi_i$ as a constituent and $1^{(1)}\varsigma, 1^{(-1)}\varsigma$ are the two extensions to $N/R$ of defect-zero characters of $N^\circ/R$ trivial on $RZ(N^\circ)/R$.

We will see that the $3$-radical subgroup $R$ does not appear as a defect group for any block of $G$, so we define no map $\Omega_R$.  (See part (2) of the proof of \prettyref{prop:ht0}.)  Below is the map $\ast_R$.  Recall our partitions $\mathfrak{B}_0=\mathfrak{B}_0(Q_{1,1,1})\cup \mathfrak{B}_0(P)\cup \mathfrak{B}_0(R)$, $\mathfrak{B}_8(i)=\mathfrak{B}_8(i,P)\cup \mathfrak{B}_8(i,R)$, and $\mathfrak{B}_9(i)=\mathfrak{B}_9(i,P)\cup \mathfrak{B}_9(i,R)$.
\scriptsize
\begin{align*}
\begin{array}{c}\ast_{R} \\
(\ell= 3)\end{array} \colon& \scriptsize \mathfrak{B}_{0}(R)\rightarrow \{1^{(\mu)}\varsigma\colon \mu\in\{\pm1\} \}\\
&\scriptsize\left\{\begin{array}{ccc}
    \mathfrak{B}_{8}(i, R)&\multirow{2}{*}{$\rightarrow \widetilde{\varphi}_{3i}\varsigma$} & \epsilon=1\\
    \mathfrak{B}_{9}(i, R)&& \epsilon=-1
        \end{array}\right.\quad i\in I_{q-\epsilon}, 3^d|i\\
\end{align*}
\normalsize

\subsection{$Q=Q^{(3)}$}
Let $Q=Q^{(3)}$ with $\ell|(q^2+\epsilon q+1)$, so $N=C:6$, and $C=C_{q^3-\epsilon}$.  
Let $\tau, \beta$ be as in the discussion preceding \prettyref{prop:radsubs}, so that $N/C=\langle \tau,\beta\rangle$.  Let $\phi_i\in\irr(C)$ be the character which maps $\widetilde{\zeta}\mapsto \zeta^i$, where $\widetilde{\zeta}$ is a fixed generator of $C$ and $\zeta=\exp\left(\frac{2\pi\sqrt{-1}}{q^3-\epsilon}\right)$.  Let $\chi\in\irr(N)$ and let $\phi_i$ be a constituent of $\chi|_C$.  Note that $\phi_i$ is invariant under the action of $\beta$ if and only if $(q^2+\epsilon q+1)|i$.


Now, since $\phi_i\neq \phi_{-i}$ for any $i\neq0$, it follows that if $(q^2+\epsilon q +1)\not|\,i$, then $\phi_i$ has stabilizer $C$ in $N$, and $\chi$ is uniquely determined by a constituent $\phi_i$ on $C$ for $i\in I_{q^3-\epsilon}$.  
If $i\neq0$ and $(q^2+\epsilon q+1)|i$, then $\phi_i$ has $[N:\stab_N(\phi_i)]=2$, and $[\stab_N(\phi_i):C]=3$.   In this case, there are three choices of $\chi$ that restrict to $C$ as the sum $\phi_i+\phi_{-i}$.  That is, we obtain three characters $\chi\in\irr(N|\phi_i)$ for each $i\in I_{q-\epsilon}$.  Finally, there are six characters in $\irr(N|1_C)$.



When $\ell\neq 3$, $Q^{(3)}\in\mathrm{Syl}_\ell(G)$ is the unique (up to conjugacy) radical subgroup and every character of $N$ has degree prime to $\ell$, and hence $\irr_0(N|Q)=\irr_{\ell'}(N)=\irr(N)$.  Recall that in this case, we define no map $\ast_Q$, as by \cite[Proposition 6.2]{SpathBAWCreduction}, $G$ is good for the BAWC. The map $\Omega_{Q^{(3)}}$ in this case is:

\scriptsize
\begin{align*}
\begin{array}{c}\Omega_{Q^{(3)}} \\
(\ell\neq 3)\end{array} \colon&\scriptsize
\left\{\begin{array}{ccc}
    \{\chi_1,\chi_3,\chi_4,\chi_9,\chi_{10}, \chi_{12}\}&\multirow{2}{*}{$\rightarrow \irr\left(N|1_C\right)$} & \epsilon=1\\
    \{\chi_1,\chi_2,\chi_5,\chi_8,\chi_{11}, \chi_{12}\}&& \epsilon=-1
        \end{array}\right.\\
&\scriptsize \left\{\begin{array}{ccc}
    \E_{8}(i)&\multirow{2}{*}{$\rightarrow \irr\left(N|\phi_{(q^2+\epsilon q+1)i}\right)$} & \epsilon=1\\
    \E_{9}(i)&& \epsilon=-1
        \end{array}\right.\quad i\in I_{q-\epsilon}\\
& \scriptsize\left\{\begin{array}{ccc}
    \E_{31}(i)&\multirow{2}{*}{$\rightarrow \irr\left(N|\phi_i\right)$} & \epsilon=1\\
    \E_{34}(i) && \epsilon=-1
        \end{array}\right.\quad i\in I_{q^3-\epsilon}
\end{align*}
\normalsize

When $\ell=3$, recall that a defect-zero character of $N/Q$ or a height-zero character of $N$ with defect group $Q$ will have $\chi(1)_3=3$.  Let $\chi$ be such a character, with constituent $\phi_i$ on $C$.  Then $|\stab_N(\phi_i)|$ cannot be divisible by $3$, so $\phi_i$ must not be stabilized by $\beta$ and we see $(q^2+\epsilon q+1)\not|\,i$.  Hence the characters with $\chi(1)_3=3$ are exactly the members of $\irr(N|\phi_i)$ for $i\in I_{q^3-\epsilon}$.  To be trivial on $Q$, we just further require that $3^{d+1}|i$. 

The maps for this group are below.  We note that the characters of $G$ involved in the maps when $\ell=3$ correspond to those involved in the maps for $Q_3$ above when $\ell\neq 3$.

\noindent\begin{minipage}[b]{0.45\linewidth}
\scriptsize
\[
\begin{array}{c}\Omega_{Q^{(3)}} \\
(\ell= 3)\end{array} \colon\scriptsize \left\{\begin{array}{ccc}
    \E_{31}(i)&\multirow{2}{*}{\hspace{-.2cm}$\rightarrow \irr\left(N|\phi_i\right)$} & \hspace{-.1cm}\epsilon=1\\
    \E_{34}(i) && \hspace{-.1cm}\epsilon=-1
        \end{array}\right.\hbox{   } i\in I_{q^3-\epsilon}
\]
\end{minipage}
\hspace{.6cm}
\begin{minipage}[b]{0.45\linewidth}
\scriptsize
\[
\begin{array}{c}\ast_{Q^{(3)}} \\
(\ell= 3)\end{array}\colon\scriptsize \left\{\begin{array}{ccc}
    \mathfrak{B}_{31}(i)&\multirow{2}{*}{\hspace{-.2cm}$\rightarrow \irr\left(N|\phi_i\right)$} & \hspace{-.1cm}\epsilon=1\\
    \mathfrak{B}_{34}(i) && \hspace{-.1cm}\epsilon=-1
        \end{array}\right.\hbox{   } i\in I_{q^3-\epsilon}, 3^{d+1}|i
\]
\end{minipage}
\normalsize

\subsection{$Q=Q^{(2)}$}
Now let $\ell|(q^2+1)$ and let $Q=Q^{(2)}$ be a Sylow $\ell$-subgroup of $G$.  Then again $\irr_0(N|Q)=\irr_{\ell'}(N)=\irr(N)$, as $N=\left(C_{q^2+1}:2^2\right)\times Sp_2(q)$.  Recall $C=C_{q^2+1}\times Sp_2(q)$.  

Fix a generator $\widetilde{\xi_2}$ of $C_{q^2+1}$ and let $\vartheta_i$ denote the character of $C_{q^2+1}$ so that $\vartheta_i(\widetilde{\xi_2})=\xi_2^i$. 
Then since $\xi_2^i\neq \xi_2^{-i}$ or $\xi_2^{iq}$ for $i\neq 0$, we see that $\stab_{N}(\vartheta_i)=C.$  Hence if $\chi\in\irr(N)$ is nontrivial on $C_{q^2+1}$, then $\chi$ is uniquely determined by a constituent $\vartheta_i\times \theta$ on $C$, for $i\in I_{q^2+1}$ and
$\theta\in\irr(Sp_2(q))$.  
For each choice of $\theta\in\irr(Sp_2(q))$, we also have $4$ characters of $N$ whose restriction to $C_{q^2+1}$ is trivial.

Now, let $B_0$ and $B_1$ be the principal block and the unique other unipotent $\ell$-block of $G$ of positive defect, as in \cite{white2000}, and let $\mathcal{U}(B)$ denote the unipotent characters in the block $B$.  The map $\Omega_{Q^{(2)}}$ is as follows:

\scriptsize
\begin{align*}
\Omega_{Q^{(2)}} \colon&\scriptsize \mathcal{U}(B_0) \rightarrow \irr\left(N|(1\times1)\right)\\
   &\scriptsize \mathcal{U}(B_1)\rightarrow \irr\left(N|(1\times\chi_2)\right)\\
   &\scriptsize\E_6(i)\setminus\{\chi_{15}(i), \chi_{16}(i)\} \rightarrow \irr\left(N|(1\times\chi_3(i))\right) \quad\quad i\in I_{q-1}\\
   &\scriptsize\E_7(i)\setminus\{\chi_{21}(i), \chi_{22}(i)\} \rightarrow \irr\left(N|(1\times\chi_4(i))\right)\quad\quad i\in I_{q+1}\\
   &\scriptsize\chi_{55}(i) \mapsto \irr\left(N|(\vartheta_i\times1)\right)\quad\quad i\in I_{q^2+1}\\
   &\scriptsize\chi_{56}(i)\mapsto \irr\left(N|(\vartheta_i\times\chi_2)\right)\quad\quad i\in I_{q^2+1}\\
   &\scriptsize\chi_{62}(j,i) \mapsto \irr\left(N|(\vartheta_i\times\chi_3(j))\right) \quad\quad i\in I_{q^2+1}; \quad j\in I_{q-1}\\
   &\scriptsize\chi_{65}(i,j) \mapsto \irr\left(N|(\vartheta_i\times\chi_4(j))\right)\quad\quad i\in I_{q^2+1}; \quad j\in I_{q+1}\\
\end{align*}
\normalsize

We remark that here we have used the notation of CHEVIE \cite{chevie} for the characters of $Sp_2(q)$.  Recall that we again define no map $\ast_{Q}$ in this case, as $Q\in\syl_\ell(G)$ is cyclic. 
\subsection{The Maps for $Sp_4(2^a)$}\label{sec:mapsSp4}
Here we will use the notation of \cite{white95} for blocks of $Sp_4(q)$, with $q$ even.  We will again write $\mathfrak{B}_i(J)$ for the Brauer characters in $B_i(J)$ when defining our maps $\ast_Q$.  Let $\ell|(q^2-1)$.

\noindent\begin{minipage}[b]{0.45\linewidth}
\scriptsize
\begin{align*}
\ast_{Q_1}\colon &\scriptsize \left\{\begin{array}{ccc}
                      \mathfrak{B}_9(i)&\multirow{2}{*}{$\rightarrow \irr\left(N|(\chi_{q-\epsilon}(i)\times 1)\right)$} & \epsilon=1 \\
                      \mathfrak{B}_7(i)& & \epsilon=-1\\
                    \end{array}\right.\\&\quad\quad\quad i\in I_{q+\epsilon}\\
&\scriptsize \left\{\begin{array}{ccc}
                      \mathfrak{B}_{17}(j,i)&\multirow{2}{*}{$\rightarrow \irr\left(N|(\chi_{q-\epsilon}(i)\times \varphi_j)\right)$} & \epsilon=1 \\
                      \mathfrak{B}_{17}(i,j)& & \epsilon=-1\\
                    \end{array}\right.\\&\quad\quad\quad i\in I_{q+\epsilon};\quad j\in I_{q-\epsilon}, \ell^d|j\\
\end{align*}
\end{minipage}
\hspace{.5cm}
\begin{minipage}[b]{0.45\linewidth}
\scriptsize
\begin{align*}
\Omega_{Q_1}\colon &\scriptsize \left\{\begin{array}{ccc}
                      \{\chi_9(i), \chi_{10}(i)\}&\multirow{2}{*}{$\rightarrow \irr\left(N|(\chi_{q-\epsilon}(i)\times1)\right)$} & \epsilon=1 \\
                      \{\chi_7(i), \chi_{8}(i)\}& & \epsilon=-1\\
                    \end{array}\right.\\&\quad\quad\quad i\in I_{q+\epsilon}\\
&\scriptsize \left\{\begin{array}{ccc}
                      \chi_{17}(j,i)&\multirow{2}{*}{$\mapsto \irr\left(N|(\chi_{q-\epsilon}(i)\times \varphi_j)\right)$} & \epsilon=1 \\
                      \chi_{17}(i,j)& & \epsilon=-1\\
                    \end{array}\right.\\&\quad\quad\quad i\in I_{q+\epsilon}; \quad j\in I_{q-\epsilon}\\
\end{align*}
\end{minipage}

\noindent\begin{minipage}[b]{0.45\linewidth}
\scriptsize
\begin{align*}
\ast_{Q_2}\colon &\scriptsize \left\{\begin{array}{ccc}
                      \mathfrak{B}_{13}(i)&\multirow{2}{*}{$\rightarrow \irr\left(N|(\chi_{q-\epsilon}(2i)\times 1)\right)$} & \epsilon=1 \\
                      \mathfrak{B}_{11}(i)&& \epsilon=-1
                    \end{array}\right.\\&\quad\quad\quad i\in I_{q+\epsilon}\\
&\scriptsize\left\{\begin{array}{ccc}
                      \mathfrak{B}_{16}(i)&\multirow{2}{*}{$\rightarrow \irr\left(N|(\chi_{q-\epsilon}(2i_1)\times \varphi_{2i_2})\right)$} & \epsilon=1 \\
                      \mathfrak{B}_{16}(i)& & \epsilon=-1\\
                    \end{array}\right.\\&\quad\quad\quad i=i_1(q-\epsilon)+i_2(q+\epsilon)\in I_{q^2-1}, \ell^d|i\\
\end{align*}
\end{minipage}
\begin{minipage}[b]{0.45\linewidth}
\scriptsize
\begin{align*}
\Omega_{Q_2}\colon & \scriptsize\left\{\begin{array}{ccc}
                      \{\chi_{13}(i), \chi_{14}(i)\}&\multirow{2}{*}{$\rightarrow \irr\left(N|(\chi_{q-\epsilon}(2i)\times 1)\right)$} & \epsilon=1 \\
                      \{\chi_{11}(i), \chi_{12}(i)\}&& \epsilon=-1
                    \end{array}\right.\\&\quad\quad\quad i\in I_{q+\epsilon}\\
&\scriptsize\left\{\begin{array}{ccc}
                      \chi_{16}(i)&\multirow{2}{*}{$\mapsto \irr\left(N|(\chi_{q-\epsilon}(2i_1)\times \varphi_{2i_2})\right)$} & \epsilon=1 \\
                      \chi_{16}(i)& & \epsilon=-1\\
                    \end{array}\right.\\&\quad\quad\quad i=i_1(q-\epsilon)+i_2(q+\epsilon)\in I_{q^2-1}\\
\end{align*}
\end{minipage}

\noindent\begin{minipage}[b]{0.45\linewidth}
\scriptsize
\begin{align*}
\ast_{Q_{1,1}}\colon &\scriptsize \mathfrak{B}_0\rightarrow \irr\left(N|(1\times1)\right)\\
&\scriptsize \left\{\begin{array}{ccc}
                      \mathfrak{B}_{7}(i)&\multirow{2}{*}{$\rightarrow \irr\left(N|(\varphi_i\times 1)\right)$} & \epsilon=1 \\
                      \mathfrak{B}_9(i)& & \epsilon=-1\\
                    \end{array}\right.\\&\quad\quad\quad i\in I_{q-\epsilon}, \ell^d|i\\
& \scriptsize\left\{\begin{array}{ccc}
                      \mathfrak{B}_{11}(i)&\multirow{2}{*}{$\rightarrow \irr\left(N|(\varphi_i\times \varphi_i)\right)$} & \epsilon=1 \\
                      \mathfrak{B}_{13}(i)& & \epsilon=-1\\
                    \end{array}\right.\\&\quad\quad\quad i\in I_{q-\epsilon}, \ell^d|i\\
&\scriptsize \left\{\begin{array}{ccc}
                      \mathfrak{B}_{15}(i,j)&\multirow{2}{*}{$\rightarrow \irr\left(N|(\varphi_i\times\varphi_j)\right)$} & \epsilon=1 \\
                      \mathfrak{B}_{19}(i, j)&  & \epsilon=-1\\
                    \end{array}\right.\\&\quad\quad\quad (i, j)\in I_{q-\epsilon}^{2*}, \ell^d|i,j\\
\end{align*}
\end{minipage}
\begin{minipage}[b]{0.45\linewidth}
\scriptsize
\begin{align*}
\Omega_{Q_{1,1}}\colon &\scriptsize \{\chi_1,\chi_3,\chi_4,\chi_5,\chi_6\}\rightarrow \irr\left(N|(1\times1)\right)\\
& \scriptsize\left\{\begin{array}{ccc}
                      \{\chi_7(i), \chi_{8}(i)\}&\multirow{2}{*}{$\rightarrow \irr\left(N|(\varphi_i\times 1)\right)$} & \epsilon=1 \\
                     \{\chi_9(i), \chi_{10}(i)\}& & \epsilon=-1\\
                    \end{array}\right.\\&\quad\quad\quad i\in I_{q-\epsilon}\\
&\scriptsize \left\{\begin{array}{ccc}
                      \{\chi_{11}(i), \chi_{12}(i)\}&\multirow{2}{*}{$\rightarrow \irr\left(N|(\varphi_i\times \varphi_i)\right)$} & \epsilon=1 \\
                      \{\chi_{13}(i), \chi_{14}(i)\}& & \epsilon=-1\\
                    \end{array}\right.\\&\quad\quad\quad i\in I_{q-\epsilon}\\
&\scriptsize \left\{\begin{array}{ccc}
                      \chi_{15}(i,j)&\multirow{2}{*}{$\mapsto \irr\left(N|(\varphi_i\times\varphi_j)\right)$} & \epsilon=1 \\
                      \chi_{19}(i, j)&  & \epsilon=-1\\
                    \end{array}\right.\\&\quad\quad\quad (i, j)\in I_{q-\epsilon}^{2*}\\
\end{align*}
\end{minipage}
\normalsize

\noindent Now let $\ell|(q^2+1)$, so $Q=Q^{(2)}\in\syl_\ell(H)$.

\scriptsize
\begin{align*}
\Omega_{Q^{(2)}} \colon&\scriptsize \{\chi_1,\chi_2,\chi_5,\chi_6\} \rightarrow \irr(N|1_C)\\
   &\scriptsize\chi_{18}(i) \mapsto \irr(N|\vartheta_i)\quad\quad i\in I_{q^2+1}\\
\end{align*}
\normalsize

\section{$Sp_6(2^a)$ and $Sp_4(2^a)$ are Good for the Conjectures}\label{sec:good}

In this section, we prove \prettyref{thm:Sp6good}.  We begin with a discussion regarding the automorphisms of $G=Sp_6(2^a)$ and $H=Sp_4(2^a)$.  


Let $Q$ be an $\ell$-radical subgroup of $G=Sp_6(2^a)$, where $\ell\neq 2$ is a divisor of $|G|$.  Let $\sigma_2$ be the field automorphism of $G$ induced by the Frobenius map $F_2\colon x\mapsto x^2$ for $x\in \overline{\F}_2$.  That is, $(a_{ij})^{\sigma_2} = (a_{ij}^2)$ for $(a_{ij})$ some matrix in $G$.  Then $\aut(G)=\langle G,\sigma_2\rangle$.  If $Q$ is generated by diagonal matrices and matrices with entries in $\F_2$, then $Q^{\sigma_2}=Q$, and we will write $\sigma:=\sigma_2$.  Otherwise, $Q$ is conjugate in $\underline{G}:=Sp_6(\overline{\F}_q)$ to a group $D$ of this form.  Moreover, the $G$-conjugacy class of $Q$ is determined by $D$.  If $Q=\langle x\rangle$ is cyclic, then $x$ is conjugate in $\underline{G}$ to a generator, $y$ for $D$.  But $y$ is also conjugate in $\underline{G}$ to $y^{\sigma_2}$, so $x$ is conjugate to $x^{\sigma_2}$ in $\underline{G}$.  But since two semisimple elements of $G$ are conjugate whenever they are conjugate in $\underline{G}$ (see, for example, the description in \cite{Luebeckthesis} of conjugacy classes of $G$), we see that $Q$ is conjugate in $G$ to $Q^{\sigma_2}$.  If $Q$ is abelian but not cyclic, we can view $Q$ as a subgroup of the product of lower-rank symplectic groups (e.g. $Q_{2,1}\leq Sp_4(q)\times Sp_2(q)$), and a similar argument on the direct factors shows that $Q$ is $G$-conjugate to $Q^{\sigma_2}$.   Finally, if $Q$ is nonabelian, then $\ell=3$ and $Q$ must be either $R$ or $P$, in which case $Q^{\sigma_2}$ must be $G$-conjugate to $Q$ since $Q$ is the unique (up to $G$-conjugacy) $\ell$-radical subgroup of its isomorphism type.

Hence in any case, there is some $\sigma\in\aut(G)$ (obtained by multiplying $\sigma_2$ by an inner automorphism) so that $Q^\sigma=Q$ and $\aut(G)=\langle G,\sigma\rangle$.  Moreover, in the case that $Q$ is one of the abelian radical subgroups, viewing $Q_{r_1, r_2, r_3}$ as a subgroup of $Sp_{2(n-s)}(q)\times\prod_{i=1}^3 Sp_{2r_i}(q)$ based on the action on the basis $\{e_1,e_2,e_3,f_1,f_2,f_3\}$ as in the discussion before \prettyref{prop:radsubs}, we see that the inner automorphism can be chosen to preserve these direct factors $Sp_{2(n-s)}(q), Sp_{2r_i}(q)$.  For the remainder of this section, given the $\ell$-radical subgroup $Q$, $\sigma$ will denote a choice of such an automorphism.

Now let $H=Sp_4(2^a)$.  Then $\out(H)$ is still cyclic, generated by a graph automorphism $\gamma_2$.  Now, the action of $\gamma_2$ switches the fundamental roots of the root system of type $B_2$, and the action on the elements of $H$ can be seen from \cite[Proposition 12.3.3]{Carter1}.  We see that $\gamma_2$ satisfies $\gamma_2^2=\sigma_2$.  We may then replace $\gamma_2$ with some $\gamma$ which fixes a Sylow $\ell$-subgroup and satisfies $\gamma^2=\sigma$.

\begin{proposition}\label{prop:equivariantprop}
Let $G=Sp_6(q)$, with $q=2^a$, $\ell\neq 2$ a prime dividing $|G|$, and $Q\leq G$ a nontrivial $\ell$-radical subgroup.  Then the maps $\Omega_Q$ and $\ast_Q$ (for $\ell|(q^2-1)$) described in \prettyref{sec:maps} are $\aut(G)_Q$-equivariant.
\end{proposition}
\begin{proof}

Let $\chi\in\underline{\irr_0}(G|Q)$ (resp. $\psi\in\ibr_\ell(G|Q)$) as defined in \prettyref{sec:maps}. Since $\mathrm{Out}(G)=\aut(G)/G\cong C_a$ is cyclic, it suffices to show that $(\Omega_Q(\chi))^\sigma=\Omega_Q(\chi^\sigma)$ (resp. $(\psi^{\ast_Q})^{\sigma}=(\psi^\sigma)^{\ast_Q}$) for a generator, $\sigma$, of $\mathrm{Out}(G)$.  In particular, let $\sigma$ be an automorphism of $G$ of the form described above and note that we can write $\sigma=y\sigma_2$ for some $y\in G$.

(1) Note that $\sigma$ fixes the unipotent classes of $G$.   Now, a semisimple class of $G$ is determined by its eigenvalues (possibly in an extension field of $\F_q$) on the action of the natural module $\langle e_1, e_2, e_3, f_1, f_2, f_3\rangle_{\F_q}$ of $G$.  Hence, as the action of $\sigma$ on semisimple classes of $G$ is to square the eigenvalues, we see that $\sigma$ sends the class $C_i(j_1, j_2, j_3)$ of $G$ (in the notation of CHEVIE $\cite{chevie}$, with the possibility of some of the indices $j_k$ not appearing, e.g. in the case of the unipotent classes which have no indices, or $C_{13}(j_1)$ which has only one index) to the class $C_i(2j_1, 2j_2, 2j_3)$ (which we mention is also equal to the class $C_i(-2j_1, -2j_2, -2j_3))$.

Let $\theta\in\irr(G)$.  Then $\theta^\sigma(g)=\theta(g^{\sigma^{-1}})$ for $g\in G$. From the observations in the above paragraph and careful inspection of the character values for irreducible characters of $G$ in CHEVIE and \cite{Luebeckthesis}, we see that the character $\chi_i(j_1, j_2, j_3)$ (again in the notation of CHEVIE) is mapped under $\sigma$ to $\chi_i(2j_1, 2j_2, 2j_3)$.  (Note here that the action of $\sigma$ on unipotent and semisimple classes is enough in this situation to determine the action on the character.)  That is, $\sigma$ preserves the family of a character, and in the notation of \prettyref{sec:maps}, $\E_i(J)^\sigma=\E_i(2J)$, where $2J=(2j_1,..,2j_k)$ for an indexing set $J=(j_1,...,j_k)$.  We note that this can also be seen using similar considerations for the action of the corresponding automorphism of $G^\ast\cong G$ on semisimple elements of $G^\ast$, together with the knowledge of the action of automorphisms on unipotent characters (see, for example \cite[Proposition 2.5]{malleheightzero}) and properties of Lusztig's parameterization of characters (see \cite[Theorem 7.1(vi)]{dignemichel90}).


Now, as discussed in \prettyref{sec:preliminaries}, the set $\E_i(J)$ forms a basic set for the block $B_i(J)$, and we may write $\psi\in \mathfrak{B}_i(J)$ as a $\Z$-linear combination $\psi=\sum a_\theta \wh{\theta}$ of the $\wh{\theta}$ for $\theta\in\E_i(J)$, so noting that $\psi^\sigma=\left(\sum a_\theta \wh{\theta}\right)^\sigma=\sum a_{\theta} \wh{\theta^\sigma}$, we see that $\mathfrak{B}_i(J)^\sigma=\mathfrak{B}_i(2J)$, with the character families preserved.  Also, note that both Brauer and ordinary characters of unipotent blocks of $G$ are fixed under $\sigma$.

(2) Now, since $\sigma$ preserves $Q$, we see that $\sigma$ induces an automorphism on $N=N_G(Q)$ and $C=C_G(Q)$.  Let $Q$ be one of the abelian radical subgroup, $Q=Q_{r_1, r_2, r_3}$ so $C=Sp_{2(n-s)}(q)\times\prod_{i=1}^3 GL_{r_i}^\epsilon(q)\leq Sp_{2(n-s)}(q)\times\prod_{i=1}^3 Sp_{2r_i}(q)$, and $\sigma=y\sigma_2$ can be chosen with $y$ preserving the factors $Sp_{2(n-s)}(q), Sp_{2r_i}(q)$, as discussed above, and hence $\sigma$ preserves the direct factors $Sp_{2(n-s)}(q), GL_{r_i}^\epsilon(q)$ of $C$.  Moreover, $N$ contains a normal subgroup $C\leq M=Sp_{2(n-s)}(q)\times\prod_{i=1}^3 GL_{r_i}^\epsilon(q):2\leq Sp_{2(n-s)}(q)\times\prod_{i=1}^3 Sp_{2r_i}(q)$, whose direct factors are again preserved by $\sigma$.

The fusion of semisimple classes of direct factors of $C$ of the form $GL_1^\epsilon(q), SL_2(q), Sp_4(q)$, and $GL_3^\epsilon(q)$ into $G$ is clear from the eigenvalues, so we see that the action of $\sigma$ on these classes is the same as the field automorphism induced by the map $\sigma_2\colon x\mapsto x^2$, as in the case for $G$.  This yields that the action of $\sigma$ on the irreducible ordinary characters of the factors $Sp_4(q), SL_2(q)$, and $GL_3^\pm(q)$ of $C$ that we require in the descriptions of $\dz(N/Q)$ and $\underline{\irr_0}(N|Q)$ is analogous to the action on $\irr(G)$.  That is, these characters are indexed in a similar fashion $\{\chi_i(j_1, j_2, j_3)\}$ in CHEVIE, and we have $\chi_i(j_1, j_2, j_3)^\sigma=\chi_i(2j_1, 2j_2, 2j_3)$.  Similarly, the character $\varphi_i$ of a factor $GL_1(q)$ in $C$ is mapped to $\varphi_{2i}$, the character $\phi_i$ of a factor $GL_1^\epsilon(q^3)\cong C_{q^3-\epsilon}$ (in the case $Q=Q^{(3)}$) is mapped to $\phi_{2i}$, and the character $\vartheta_i$ of a factor $C_{q^2+1}$ (in the case $Q=Q^{(2)}$) is mapped to $\vartheta_{2i}$.  Hence in the situations where $Q$ is abelian and $\Omega_Q(\chi)$ (resp. $\psi^{\ast_Q}$) is one of the characters of $N$ uniquely determined by a constituent on $C$, we see that $(\Omega_Q(\chi))^\sigma=\Omega_Q(\chi^\sigma)$ (resp. $(\psi^{\ast_Q})^{\sigma}=(\psi^\sigma)^{\ast_Q}$), as desired.

(3) Moreover, when $Q=Q_{r_1, r_2, r_3}$, because $\sigma$ preserves the direct factors of $M$, we see that it commutes with the $S_{c_i}$- action (recall here that $N=Sp_{2(n-s)}(q)\times\prod \left(GL_{r_i}^\epsilon(q):2\right)\wr S_{c_i}$, where the product is now taken over the $i$ so that each distinct $r_i$ appears only once) and the $C_2$-actions on the individual $GL_{r_i}^\epsilon(q)$'s.  When $Q=Q^{(2)}$ or $Q^{(3)}$, since $\sigma$ leaves $C$ invariant, we see from the action of $\sigma$ on semisimple elements of $G$ that $\sigma$ squares elements of the cyclic direct factor of $C$ and commutes with the generators of $N/C$.  Hence in the case that $Q$ is abelian and a constituent on $C$ does not uniquely determine the character of $N$, $\sigma$ fixes the choice of extension and we again have $(\Omega_Q(\chi))^\sigma=\Omega_Q(\chi^\sigma)$ (resp. $(\psi^{\ast_Q})^{\sigma}=(\psi^\sigma)^{\ast_Q}$).


(4) 
The observations in $(1)-(3)$ imply that $(\Omega_Q(\chi))^\sigma=\Omega_Q(\chi^\sigma)$ (resp. $(\psi^{\ast_Q})^{\sigma}=(\psi^\sigma)^{\ast_Q}$) for our choice of generator $\sigma$, except possibly when $Q= P$ or $R$ and $\ell=3|(q^2-1)$.

Now, when $Q=P$, choosing $\sigma$ as in the case $Q=Q_{1,1,1}$, the discussion on height-zero characters of $N$, combined with (2),(3) and the fact that $\sigma$ commutes with the action of the $S_3$-subgroup of $N$ 
yields that the character $
\theta(i,j)$ of $J$ (in the notation from \prettyref{sec:dzN(Q)}) is mapped to $\theta(2i,2j)$ under $\sigma$ and that the choice of extensions are fixed.  Hence again in this case, the maps are equivariant.

Finally, let $Q=R$.  Since $\Omega_R$ is trivial, we need only consider the map $\ast_R$, and therefore the members of $\dz(N_G(R)/R)$.  By \prettyref{sec:dzN(Q)}, this set is comprised of the characters $\widetilde{\varphi}_i\varsigma$ with $3^d|i$, where $\varsigma\in \irr(Sp_2(3))$ is the Steinberg character and $\widetilde{\varphi}_i=\widetilde{\varphi}_{-i}$ is the character whose restriction to $C_{q-\epsilon}$ contains $\varphi_i$ as a constituent.  Now, $\varsigma$, $1^{(1)}$, and $1^{(-1)}$ are fixed by $\sigma$, and $\varphi_i^\sigma=\varphi_{2i}$ as before, so we see that $\ast_R$ is again equivariant.
\end{proof}

\begin{proposition}\label{prop:equivariantprop2}
Let $H=Sp_4(q)$, with $q=2^a$, $\ell\neq 2$ a prime dividing $|H|$, and $Q\leq H$ a nontrivial $\ell$-radical subgroup.  Then the maps $\Omega_Q$ and $\ast_Q$ (for $\ell|(q^2-1)$) described in \prettyref{sec:mapsSp4} are $\aut(H)$-equivariant.
\end{proposition}
\begin{proof}
Again, it suffices to show that $\ast_Q$ and $\Omega_Q$ commute with the generator $\gamma$ of $\out(H)$.  We will use the notation of classes and characters from CHEVIE, \cite{chevie}.  From comparing notation of CHEVIE, \cite{enomoto72}, and \cite{Carter1}, we deduce that the action of $\gamma$ on the unipotent classes of $H$ is to switch $C_2$ and $C_3$ and fix the other unipotent classes.  Moreover, $C_7(i)^\gamma=C_{11}(i)$ and $C_{11}(i)^\gamma=C_{7}(2i)$.  Similarly, $C_9(i)^\gamma=C_{13}(i)$ and $C_{13}(i)^\gamma=C_{9}(2i)$.  Hence $\gamma$ switches $Q_1$ and $Q_2$.  Also, $C_{15}(i,j)^\gamma=C_{15}(i+j,i-j)$, $C_{19}(i,j)^\gamma=C_{19}(i+j,i-j)$, and $Q_{1,1}$ is stabilized by $\gamma$.  Moreover, $C_{17}(i,j)=C_{16}(i(q+1)+j(q-1))$,  $C_{16}(i)^\gamma=C_{17}(i\mod(q-1), j\mod(q+1))$, and $C_{18}(i)^\gamma=C_{18}((q+1)i)$.

From this, using the character table for $H$ in CHEVIE \cite{chevie}, we can see the action of $\gamma$ on the relevant characters (and blocks) of $H$.  Namely, $B_7(i)^\gamma=B_{11}(i)$, $B_{11}(i)^\gamma=B_{7}(2i)$, $B_9(i)^\gamma=B_{13}(i)$, $B_{13}(i)^\gamma=B_{9}(2i)$, $B_{15}(i,j)^\gamma=B_{15}(i+j,i-j)$,  $B_{19}(i,j)^\gamma=B_{19}(i+j,i-j)$, and $\chi_{18}(i)^\gamma=\chi_{18}((q+1)i)$.  Also, $B_0$ is fixed, except that $\chi_3$ and $\chi_4$ are switched.

Let ${\varphi}_i$ for $i\in I_{q-\epsilon}$ be as usual.  Considering the action of $\gamma$ on elements of $N_H(Q_1)$ and $N_H(Q_2)$, we see that the characters $({\varphi_i}\times \chi_{q-\epsilon}(j))$ of $C_H(Q_1)$ are mapped under $\gamma$ to the corresponding character $({\varphi_i}\times \chi_{q-\epsilon}(j))$ in $C_H(Q_2)$.  Applying $\gamma$ again yields $({\varphi_{2i}}\times \chi_{q-\epsilon}(2j))$ in $C_H(Q_1)$.  Moreover, for $\nu\in\{\pm1\}$, $(1^{(\nu)}\times \chi_{q-\epsilon}(j))\in\irr(N_H(Q_1))$ is mapped to the corresponding character $(1\times \chi_{q-\epsilon}(j))^{(\nu)}\in\irr(N_H(Q_2))$, which is then mapped to $(1^{(\nu)}\times \chi_{q-\epsilon}(2j))\in\irr(N_H(Q_1))$.

Inspecting the values of the characters of $N_H(Q_{1,1})/C_H(Q_{1,1})\cong C_2\wr S_2$, we see that they are fixed under $\gamma$, aside from $(1^{(1)}\times 1^{(1)})^{(-1)}$ and $(1^{(-1)}\times 1^{(-1)})^{(1)}$, which are switched.
So, choosing $\{\chi_3,\chi_4\}\rightarrow \{(1^{(1)}\times1^{(1)})^{(-1)}, (1^{(-1)}\times 1^{(-1)})^{(1)}\}$, we see that this is consistent with our maps.

Also, the characters $\theta$ of $N_H(Q_{1,1})$ which are nontrivial on $C_H(Q_{1,1})$ satisfy that if $(\varphi_i\times \varphi_j)$ is a constituent of $\theta|_{C_H(Q_{1,1})}$, then $(\varphi_{i+j}\times \varphi_{i-j})$ is a constituent of $\theta^\gamma|_{C_H(Q_{1,1})}$, where $i,j\in I_{q-\epsilon}\cup\{0\}$, and $\varphi_0:=1_{C_{q-\epsilon}}$.  Moreover, in the case $i=0$, the action on $C_2\wr S_2$ yields that the choice of extension is fixed under $\gamma$ (i.e. $(\widetilde{\varphi_i}\times 1^{(\nu)})^\gamma=(\widetilde{\varphi_i}\times\widetilde{ \varphi_i})^{(\nu)}$ where $\nu\in\{\pm1\}$).

Finally, $\vartheta_i\in\irr\left(C_G(Q^{(2)})\right)$ is mapped under $\gamma$ to $\vartheta_{(q+1)i}$, and when $i=0$ the choice of extension to $N_G(Q^{(2)})$ is fixed by $\gamma$.

Altogether, these discussions yield that $(\psi^{\ast_Q})^\gamma=(\psi^{\gamma})^{\ast_Q}$ for each $\psi\in\ibr_\ell(H|Q)$, as desired, and similar for $\Omega_Q$.
\end{proof}

\begin{proposition}\label{prop:inducedblockprop}
Let $G=Sp_6(2^a)$ or $Sp_4(2^a)$, $\ell$ an odd prime dividing $|G|$, and $Q$ a nontrivial $\ell$-radical subgroup of $G$.  Let the sets $\underline{\irr_0}(G|Q), \underline{\irr_0}(N_G(Q)|Q),$ and $\ibr_\ell(G|Q)$ and the maps $\Omega_Q$, $\ast_Q$ be as described in \prettyref{sec:maps}.  Then
\begin{itemize}
\item If $\chi\in\underline{\irr_0}(G|Q)$ with $B=\mathrm{Bl}(G|\chi)$ and $b= \mathrm{Bl}(N_G(Q)|\Omega_Q(\chi))$, then $b^G=B$.
\item If $\psi\in\ibr_\ell(G|Q)$ with $B=\mathrm{Bl}(G|\psi)$ and $b= \mathrm{Bl}(N_G(Q)|\psi^{\ast_Q})$, then $b^G=B$.
\end{itemize}
\end{proposition}
\begin{proof}
Let $N:=N_G(Q)$ and $C:=C_G(Q)$.  As $b\in\mathrm{Bl}(N_G(Q))$, $b^G$ is defined and $b^G=B$ if and only if $\lambda_B(\mathcal{K}^+)=\lambda_b\left((\mathcal{K}\cap C)^+\right)$ for all conjugacy classes $\mathcal{K}$ of $G$ (see, for example, \cite[Lemma 15.44]{isaacs}).  Let $\chi\in\irr(G|B)$.  The central character $\omega_\chi$ for $G$ are available in CHEVIE (see section 3.26 of the CHEVIE manual), and the values of $\varphi\in\irr(N|b)$ on $C$ can be computed by their descriptions and using the character tables for $Sp_4(q), SL_2(q),$ and $GL_3^\pm(q)$ available in CHEVIE.  Hence it remains only to determine the fusion of classes of $C$ into $G$ in order to compute $\omega_\varphi\left((\mathcal{K}\cap C)^+\right)=\frac{1}{\varphi(1)}\sum_{\mathcal{C}\subseteq\mathcal{K}} \varphi(g)|\mathcal{C}|$, where $g\in \mathcal{C}$ and the sum is taken over classes $\mathcal{C}$ of $C$ which lie in $\mathcal{K}$, and compare the image of this under $\ast$ with $\omega_\chi(\mathcal{K}^+)^\ast$.  (We note that $\omega_\chi(1^+)=1=\omega_\varphi((1\cap C)^+)$ for all $\chi\in\irr(G), \varphi\in\irr(N)$, so it suffices to consider nontrivial classes $\mathcal{K}$.)

We present here the complete discussion for $\ast_R$ when $G=Sp_6(q)$, $\ell=3|(q-\epsilon)$.  The other situations are similar, though quite tedious.

When $Q=R$, we have $C=C_{q-\epsilon}= Z(GL_3^\epsilon(q))$, embedded in $G$ in the usual way.  The set $\ibr_3(G|R)$ consists of two Brauer characters in a unipotent block and one Brauer character in each set $\mathfrak{B}_8(i)$ if $\epsilon=1$ or $\mathfrak{B}_9(i)$ if $\epsilon=-1$ with $i\in I_{q-\epsilon}$ divisible by $3^d$.  Choosing $\chi=1_G$ for $B=B_0$, $\chi=\chi_{27}(i)$ for $B=B_8(i)$, and $\chi=\chi_{30}(i)$ for $B=B_9(i)$, we have $\omega_{1_G}(\mathcal{K}^+)^\ast=0=\omega_{\chi_{27}(i)}(\mathcal{K}^+)^\ast$ when $\epsilon=1$ for every nontrivial conjugacy class $\mathcal{K}\neq C_{25}(j)$ of $G$ (in the notation of CHEVIE) for any $j\in I_{q-1}$ and $\omega_{1_G}(\mathcal{K}^+)^\ast=0=\omega_{\chi_{30}(i)}(\mathcal{K}^+)^\ast$ when $\epsilon=-1$ for every nontrivial conjugacy class $\mathcal{K}\neq C_{28}(j)$ for any $j\in I_{q+1}$.  (These can be seen from the central character table for $G$ in CHEVIE, together with the fact that $(q-\epsilon)^\ast=3^\ast=0$.)

Now, let $\zeta$ generate $C\cong C_{q-\epsilon}$, so $\zeta^i$ is identified in $G$ with the semisimple element with eigenvalues $\tilde{\zeta}$ and $\tilde{\zeta}^{-1}$, each of multiplicity $3$, where $\tilde{\zeta}$ is a fixed primitive $(q-\epsilon)$ root of unity in $\overline{\F_q}^\times$.  Then $\{\zeta^i, \zeta^{-i}\}=C_{25}(i)\cap C$ if $\epsilon=1$ and $=C_{28}(i)\cap C$ if $\epsilon=-1$.

Let $\overline{\zeta}=\exp\left(\frac{2\pi\sqrt{-1}}{q-\epsilon}\right)$ in $\C^\times$ and let $\chi:=\chi_{27}(i)$ or $\chi_{30}(i)$ and $\mathcal{K}=C_{25}(j)$ or $C_{28}(j)$, in the cases $\epsilon=1, -1$, respectively.   Then $\omega_{\chi}(\mathcal{K}^+)^\ast=(\overline{\zeta}^{3ij}+\overline{\zeta}^{-3ij})^\ast$ from CHEVIE, since $(q-\epsilon)^\ast=3^\ast=0$.  But the value of $\varphi:=\widetilde{\varphi}_{3i}\varsigma$ on $\zeta^j$ is $\overline{\zeta}^{3ij}$, so $\omega_\varphi\left((\mathcal{K}\cap C)^+\right)=\frac{(\overline{\zeta}^{3ij}+\overline{\zeta}^{-3ij})2}{2}=\overline{\zeta}^{3ij}+\overline{\zeta}^{-3ij}$.  Hence we have $b^G=B$ in this case.

Moreover, $\omega_{1_G}(\mathcal{K}^+)^\ast=2$, so $b^G=B$ in this case as well, since if $\varphi=1^{(\nu)}\varsigma$ for $\nu\in\{\pm1\}$, then $\omega_{\varphi}\left((\mathcal{K}\cap C)^+\right)=2$, completing the proof for $Q=R$.
\end{proof}


\begin{proposition}\label{prop:ht0}
Let $G=Sp_6(2^a)$ or $Sp_4(2^a)$ and $Q$ a nontrivial $\ell$-radical subgroup with $\ell$ an odd prime dividing $|G|$.  The sets $\underline{\irr_0}(G|Q)$ and $\underline{\irr_0}(N_G(Q)|Q)$ defined in \prettyref{sec:maps} are exactly the sets $\irr_0(G|Q)$ and $\irr_0(N_G(Q)|Q)$ of height-zero characters of $G$ and $N_G(Q)$, respectively, with defect group $Q$.
\end{proposition}
\begin{proof}
(1) Let $N:=N_G(Q)$, $\varphi\in\underline{\irr_0}(N|Q)$, and $\chi\in\underline{\irr_0}(G|Q)$ such that $\Omega_Q(\chi)=\varphi$.  Let $b=\mathrm{Bl}(N|\varphi)$, so that $b^G$ is the block $B$ containing $\chi$, by \prettyref{prop:inducedblockprop}.  Let $D_b$ and $D_B$ denote defect groups for $b$ and $B$, respectively, so we may assume $D_b\leq D_B$.  Then as $Q$ is $\ell$-radical, we know that $Q\leq D_b\leq D_B$  (see, for example, \cite[Corollary 15.39]{isaacs}). Now, since $|G|_{\ell}/|D_B|$ must be the highest power of $\ell$ dividing the degree of every member of $\irr(B)$, inspection of the character degrees in $B$ yields that $|D_B|=|Q|$, so in fact $Q=D_b=D_B$.  Hence by inspection of the degrees of characters in our constructed sets, we see that $\underline{\irr_0}(G|Q)\subseteq \irr_0(G|Q)$ and $\underline{\irr_0}(N|Q)\subseteq \irr_0(N|Q)$.

(2) Moreover, we have constructed the set $\underline{\irr_0}(G|Q)$ to contain all characters $\chi'\in\irr(B)$ whose degrees satisfy $\chi'(1)_{\ell}=[G:Q]_{\ell}$.  (That is to say, given any block in $\mathrm{Bl}(G)$, if we included in $\underline{\irr_0}(G|Q)$ one irreducible ordinary character of the block whose degree satisfies this condition, then we included \emph{all} such members of the block.)  Further, every block $B'\in \mathrm{Bl}(G)$ of positive defect intersects the set $\underline{\irr_0}(G|Q')$ for some $\ell$-radical subgroup $Q'$, so we see that in fact $\underline{\irr_0}(G|Q)=\irr_0(G|Q)$.  Note that when $\ell=3$, this means $R$ does not occur as a defect group for any block of $G=Sp_6(2^a)$.

(3) Now, except in the case $G=Sp_6(q)$ with $\ell=3$ and $Q=Q_{1,1,1}$ or $P$, from the discussion in \prettyref{sec:dzN(Q)} we see that in fact every character $\theta$ of $N$ with $\theta(1)_{\ell}=|N|_{\ell}/|Q|$ has been included in $\underline{\irr_0}(N|Q)$, so $\underline{\irr_0}(N|Q)=\irr_0(N|Q)$.  Hence we are left with the case $\ell=3$ and $Q=Q_{1,1,1}$ or $P$.  However, by the discussion after the description of the map $\Omega_{Q_{1,1,1}}$ in the case $\ell=3$, we see that $\underline{\irr_0}(N|Q_{1,1,1})=\irr_0(N|Q_{1,1,1})$ in this case as well.  Finally, for $Q=P$, we have already described $\irr_0(N|P)$ in \prettyref{sec:dzN(Q)}.
\end{proof}

We note that \prettyref{prop:ht0} is consistent with Brauer's height-zero conjecture, which says that an $\ell$-block $B$ of a finite group has an abelian defect group if and only if every irreducible ordinary character in $B$ has height zero.  (Note that it has been shown recently in \cite{kessarmalle} that the ``if" part of this conjecture holds.)  It is also consistent with a consequence of  \cite[Theorem 7.14]{kessarmalle}, which implies that the defect group for a block which is not quasi-isolated (i.e. satisfies the conditions for Bonnaf{\'e}-Rouquier's theorem) is isomorphic to the defect group of its Bonnaf{\'e}-Rouquier correspondent.

\begin{theorem}\label{thm:goodBAWC}
Let $G=Sp_6(2^a)$ with $a\geq 1$ or $Sp_4(2^a)$ with $a\geq 2$.  Then $G$ is ``good" for the Alperin weight and blockwise Alperin weight conjectures for all primes $\ell\neq 2$.
\end{theorem}
\begin{proof}
1) Let $\ell\neq2$ be a prime dividing $|G|$.  Since $\out(G)$ is cyclic, we know $G$ is BAWC-good for any prime $\ell$ such that a Sylow $\ell$-subgroup of $G$ is cyclic, by \cite[Proposition 6.2]{SpathBAWCreduction}.  Hence, $G$ is BAWC-good for $\ell$ as long as $\ell\not|\,(q^2-1)$.  Moreover, considerations in GAP show that the statement is true for $\ell=3$ when $G=Sp_6(2)$. (The main tools here were the PrimeBlocks command, the Brauer character table for the double cover $2.Sp_6(2)$ in the Character Table Library \cite{GAPctlib}, as well as the faithful permutation representation of $2.Sp_6(2)$ on 240 points given in the online ATLAS \cite{onlineATLAS}.) Henceforth, we shall assume $\ell|(q^2-1)$ and $a\geq 2$.

2) As $a\geq 2$, the Schur multiplier of $G$ is trivial, so $G$ is its own Schur cover, so in the notation of \cite[Section 3]{TiepNavarroAWCreduction}, we may assume $S$ is just $G$ itself.  Furthermore, \cite[Lemma 6.1]{SpathBAWCreduction} implies that it suffices to show that $G$ is AWC-good for $\ell$ in the sense of \cite{TiepNavarroAWCreduction} and that the maps used satisfy condition 4.1(ii)(3) of \cite{SpathBAWCreduction}.  For the trivial group $Q=\{1\}$, the map sending defect-zero characters to themselves satisfies the conditions trivially.  Hence, it suffices to show that our sets $\ibr_\ell(G|Q)$ and maps $\ast_Q$ satisfy the conditions of \cite[Section 3]{TiepNavarroAWCreduction} and that for $\psi\in\ibr_\ell(G|Q)$, $\psi$ is a member of the induced block $b^G$, where $b=\mathrm{Bl}(N_G(Q)|\psi^{\ast_Q})$.  By \prettyref{prop:inducedblockprop}, the latter condition is satisfied.

3)  Since $Z(G)=1$ and our sets $\ibr_\ell(G|Q)$ depend only on the conjugacy class of $Q$, we know that our sets satisfy \cite[Condition 3.1.a]{TiepNavarroAWCreduction}.  Our sets $\ibr_\ell(G|Q)$ are certainly disjoint, since distinct Lusztig series or blocks are disjoint, and the union of all of these with the set $\{\wh{\chi}|\chi\in\dz(G)\}$ is all of $\ibr_\ell(G)$, by \prettyref{lem:moritaequiv} and the results of \cite{white2000} and \cite{white95}, so our sets also satisfy \cite[Condition 3.1.b]{TiepNavarroAWCreduction}.  Moreover, by Propositions \ref{prop:equivariantprop} and \ref{prop:equivariantprop2}, our maps and sets also satisfy the final partition condition and bijection condition, \cite[Conditions 3.1.c, 3.2.a]{TiepNavarroAWCreduction}.

4)  Let $Q$ be an $\ell$-radical subgroup, and fix $\theta\in\ibr_\ell(G|Q)$.  Identify $G$ with $\mathrm{Inn}(G)$, so that we can write $G\lhd \aut(G)$.  Write $X:=\aut(G)_\theta$ and let $B:=X_Q$ be the subgroup of $\aut(G)$ stabilizing both $Q$ and $\theta$.  Then certainly, $G\lhd X$, $Z(G)\leq Z(X)$, $\theta$ is $X$-invariant, and $B$ is exactly the set of automorphisms of $G$ induced by the conjugation action of $N_X(Q)$ on $G$. Moreover, $C_X(G)$ is trivial and since $X/G$ is cyclic, so is the Schur multiplier $H^2(X/G, \overline{\F}_\ell^\times)$.  Hence the normally embedded conditions \cite[Conditions 3.3.a-d]{TiepNavarroAWCreduction} are trivially satisfied, completing the proof.
\end{proof}

\begin{lemma}\label{lem:outcyclic}
Let $\ell$ be a prime, $S$ be a simple group with universal $\ell'$ covering group $G$, and $Q$ be an $\ell$-radical subgroup satisfying Conditions (i) and (ii) of \cite[Definition 7.2]{spathAMreduction} with $M_Q=N_G(Q)$.  Let $\chi\in\irr_0(G|Q)$ be such that $\aut(S)_{\chi}/S$ is cyclic and let $\eta\in\aut(G)_{\chi}$ with $\aut(S)_{\chi}=\langle S,\eta\rangle$.  Then there are $\tilde{\chi}\in\irr(A(\chi))$ and $\widetilde{\Omega_Q(\chi)}\in\irr(N_{A(\chi)}(Q))$, where $A(\chi):=G\rtimes\langle\eta\rangle$, such that:
\begin{enumerate}
\item $\tilde{\chi}|_G=\chi$
\item $\widetilde{\Omega_Q(\chi)}|_{N_G(Q)}=\Omega_Q(\chi)$
\item $\tilde{b}^{A_{\ell'}}=\tilde{B}$, where $\tilde{b}$ is the block of $N_{A_{\ell'}}(Q)$ containing $\widetilde{\Omega_Q(\chi)}|_{N_{A_{\ell'}}(Q)}$, $\tilde{B}$ is the block of $A_{\ell'}$ containing $\widetilde{\chi}|_{A_{\ell'}}$, and $G\leq A_{\ell'}\leq A(\chi)$ so that $A_{\ell'}/G$ is the Hall $\ell'$-subgroup of $A(\chi)/G$.
\end{enumerate}
\end{lemma}
\begin{proof}
First, note that $\chi$ extends to $A(\chi)$ since $A(\chi)/G$ is cyclic and $\chi$ is invariant under $A(\chi)$.  Let $\varphi:=\Omega_Q(\chi)$.  Since the map $\Omega_Q$ is $\aut(G)_Q$-equivariant, we have $\varphi=\Omega_Q(\chi^\alpha)=\varphi^\alpha$ for any $\alpha\in N_{A(\chi)}(Q)$, so $\varphi$ is invariant under $N_{A(\chi)}(Q)$ and therefore extends to some $\widetilde{\varphi}\in\irr\left(N_{A(\chi)}(Q)\right)$ since $N_{A(\chi)}(Q)/N_G(Q)$ is cyclic.  Let $\widetilde{b}$ be the block of $N_{A_{\ell'}}(Q)$ containing the restriction $\widetilde{\varphi}|_{N_{A_{\ell'}(Q)}}$ and let $B$ be the block of $G$ containing $\chi$.  Then $\tilde{b}^{A_{\ell'}}$ is defined, by \cite[Lemma 15.44]{isaacs}, and we claim that $\tilde{b}^{A_{\ell'}}$ covers $B$, so that by \cite[Theorem 9.4]{NavarroBlocks}, we can choose an extension $\tilde{\chi}$ of $\chi$ to $A(\chi)$ so that $\tilde{\chi}|_{A_{\ell'}}$ is contained in $\tilde{b}^{A_{\ell'}}$.

To prove the claim, first note that by \cite[Theorem 9.5]{NavarroBlocks}, $\tilde{b}^{A_{\ell'}}$ covers $B$ if and only if the central functions satisfy $\lambda_{\tilde{b}^{A_{\ell'}}}(\mathcal{K}^+)=\lambda_B(\mathcal{K}^+)$ for all classes $\mathcal{K}$ of $A_{\ell'}$ contained in $G$.  Let $b$ be the block of $N_G(Q)$ containing $\varphi$, so that $b^G=B$ by Condition (ii) of \cite[Definition 7.2]{spathAMreduction} and $\lambda_{\tilde{b}}$ covers $\lambda_b$ by \cite[Theorem 9.2]{NavarroBlocks}.  Let $\mathcal{K}_1,\ldots,\mathcal{K}_k$ be the classes of $G$ so that $\mathcal{K}=\bigcup_i \mathcal{K}_i$. Notice that $\mathcal{K}_i\cap N_{A_{\ell'}}(Q)= \mathcal{K}_i\cap N_{G}(Q)$ can be viewed as a union of classes of $N_{A_{\ell'}}(Q)$ contained in $N_G(Q)$ and $\bigcup_i (\mathcal{K}_i\cap N_{A_{\ell'}}(Q))= \mathcal{K}\cap N_{A_{\ell'}}(Q)$, so
\[\lambda_B(\mathcal{K}^+)=\sum_i \lambda_B(\mathcal{K}_i^+)=\sum_i\lambda_{b^G}(\mathcal{K}_i^+)=\sum_i\lambda_b\left((\mathcal{K}_i\cap N_G(Q))^+\right)\]\[=\sum_i\lambda_{\tilde{b}}\left((\mathcal{K}_i\cap N_{A_{\ell'}}(Q))^+\right)=\lambda_{\tilde{b}}\left((\mathcal{K}\cap N_{A_{\ell'}}(Q))^+\right)=\lambda_{\tilde{b}^{A_{\ell'}}}(\mathcal{K}^+),\] which proves the claim.
\end{proof}

\begin{theorem}\label{thm:goodAM}
Let $G=Sp_6(2^a)$ or $Sp_4(2^a)$ with $a\geq 2$.  Then $G$ is ``good" for the McKay and Alperin-McKay conjectures for all primes $\ell\neq 2$.
\end{theorem}
\begin{proof}
1) Again, notice that $G$ is its own Schur cover, so $G=S$ in the notation of either \cite[Definition 7.2]{spathAMreduction} or \cite[Section 10]{IsaacsMalleNavarroMcKayreduction}.  Also, note that reasoning similar to part (4) of the proof of \prettyref{thm:goodBAWC} implies that $G$ satisfies conditions (5)-(8) of the definition of McKay-good in \cite[Section 10]{IsaacsMalleNavarroMcKayreduction}.  Hence, if $G$ is ``good" for the Alperin-McKay conjecture (i.e. satisfies the inductive-AM-condition described in \cite[Definition 7.2]{spathAMreduction}), then $G$ satisfies conditions (1)-(4) of the definition of McKay-good, so is also ``good" for the McKay conjecture.  
Again, when $Q=1$, the map sending defect-zero characters to themselves satisfies the conditions trivially.

2) Let $Q\neq 1$ be an $\ell$-radical subgroup of $G$ which occurs as a defect group for some $\ell$-block of $G$.  Hence by replacing with a conjugate subgroup, we may assume that $Q$ is one of the groups described in \prettyref{sec:radsubs} aside from $R$.  The group $M_Q:=N_G(Q)$ satisfies condition (i) of \cite[Definition 7.2]{spathAMreduction}).  Moreover, Propositions \ref{prop:equivariantprop}, \ref{prop:equivariantprop2}, \ref{prop:inducedblockprop}, and \ref{prop:ht0} imply that the map $\Omega_Q$ from \prettyref{sec:maps} satisfies condition (ii) of \cite[Definition 7.2]{spathAMreduction}. (Again note that $Z(G)$ is trivial.)

3) Now, let $A:=\aut(G)$ and let $\chi\in\irr_0(G|Q)$.  Write $A_\chi:=\stab_A(\chi)$ and write $A_{Q,\chi}$ for the subgroup $N_{A_\chi}(Q)$ of elements of $A$ which stabilize both $Q$ and $\chi$.   Write $\chi':=\Omega_Q(\chi)$ and let $\widetilde{\chi}$ and $\widetilde{\chi'}$ be the extensions of $\chi$ to $A_{\chi}$ and $\chi'$ to $A_{Q,\chi}$ as in \prettyref{lem:outcyclic}, since $A/G$ is cyclic.  Say $\mathcal{P}$ and $\mathcal{P}'$, respectively, are the representations affording these extensions.  Then certainly, these representations satisfy the first three subconditions of condition (iii) of \cite[Definition 7.2]{spathAMreduction} and it suffices to show that they satisfy the final subcondition.  (Note that here $\mathrm{rep}\colon S\rightarrow G$ is simply the identity map.)

4)  Let $x$ be an $\ell$-regular element of $M_QA_{Q,\chi}=A_{Q,\chi}$ with $Q\in\mathrm{Syl}_\ell(C_G(x))$. If $x\in M_Q$, we are done by \cite[Proposition 7.4]{spathAMreduction}.  So, suppose that $x\not\in M_Q$.  That is, $x\in N_{A_\chi}(Q)\setminus N_G(Q)$.  Since $x$ is an $\ell'$-element, we see that in fact $x\in N_{A_{\ell'}}(Q)$, where $G\leq A_{\ell'}\leq A_{\chi}$ is as in \prettyref{lem:outcyclic}.  Let $\mathcal{K}$ be the conjugacy class of $x$ in $A_{\ell'}$.  Since $Q\in\mathrm{Syl}_\ell(C_G(x))$, we also have $Q\in\mathrm{Syl}_\ell(C_{A_{\ell'}}(x))$, since $[A_{\ell'}:G]$ is prime to $\ell$.  Hence $\mathcal{K}\cap C_{A_{\ell'}}(Q)$ is the class of $N_{A_{\ell'}}(Q)$, containing $x$ (see, for example, \cite[Lemma 4.16]{NavarroBlocks}).  Let $\tilde{B}$ and $\tilde{b}$ be the blocks of $A_{\ell'}$ and $N_{A_{\ell'}}(Q)$ containing $\widetilde{\chi}$ and $\widetilde{\chi'}$, respectively, so that $\tilde{b}^{A_{\ell'}}=\tilde{B}$.  Then we have $\lambda_{\tilde{B}}(\mathcal{K}^+)=\lambda_{\tilde{b}}\left((\mathcal{K}\cap C_{A_{\ell'}}(Q))^+\right)$, which implies that
\[\left(\frac{|A_{\ell'}|\widetilde{\chi}(x)}{|C_{A_{\ell'}}(x)|\chi(1)}\right)^\ast=
\left(\frac{|N_{A_{\ell'}}(Q)|\widetilde{\chi'}(x)}{|C_{N_{A_{\ell'}}(Q)}(x)|\chi'(1)}\right)^\ast.\] Moreover, except possibly in the case $G=Sp_4(q)$ and $Q=Q_1$ or $Q_2$, we can choose $\eta$ as in \prettyref{lem:outcyclic} to stabilize $Q$, by the discussion preceding \prettyref{prop:equivariantprop}, and therefore $[A_{\ell'}:G]=[N_{A_{\ell'}}(Q):N_G(Q)]$ and $[C_{A_{\ell'}}(x):C_G(x)]=[C_{N_{A_{\ell'}}(Q)}(x):C_{N_G(Q)}(x)]$.  However, note that if $Q=Q_1$ or $Q_2$ when $G=Sp_4(q)$, then $\gamma\not\in A_\chi$ (see the proof of \prettyref{prop:equivariantprop2}), but $\gamma^2=\sigma$ fixes $Q$, so the same is true in this case.  This yields
\[\left(\frac{|G|\widetilde{\chi}(x)}{|C_G(x)|\chi(1)}\right)^\ast=
\left(\frac{|N_{G}(Q)|\widetilde{\chi'}(x)}{|C_{N_{G}(Q)}(x)|\chi'(1)}\right)^\ast.\]
Hence by the last half of the proof of \cite[Proposition 7.4]{spathAMreduction}, the proof is complete.
%
\end{proof}

\begin{theorem}\label{thm:goodAM2}
Let $S=Sp_6(2)$ or $Sp_4(2)'\cong A_6$.  Then $S$ is ``good" for the Alperin-McKay conjecture for all primes.
\end{theorem}
\begin{proof}
Let $G:=6.A_6$ be the universal covering group of $S:=A_6$ and $\ell$ a prime dividing $|A_6|$.  We can construct $G$ in GAP using the generators given in the online ATLAS \cite{onlineATLAS} for the faithful permutation representation of $G$ on 432 letters.  Using the PrimeBlocks function to calculate the sizes of the defect groups and calculating the Sylow subgroups of centralizers of $\ell'$-elements, we see that the only noncentral defect group of $G$ are Sylow $\ell$-subgroups. Fix $P\in \syl_\ell(G)$.  Again using the PrimeBlocks function, the knowledge of the action of the outer automorphism group of $A_6$ on the conjugacy classes of $A_6$, and the character information for $6.A_6, 6.A_6.2_1,$ and $6.A_6.2_2$ in the GAP Character Table Library \cite{GAPctlib}, we see that we can construct bijections satisfying conditions (i) and (ii) of the Inductive AM-condition \cite[Definition 7.2]{spathAMreduction}, with $M_P:=N_G(P)$.  Further, by \cite[Proposition 4.2]{spathAMreduction}, for $\chi\in\irr_0(G|P)$, there exist $\mathcal{P}, \mathcal{P}'$ satisfying the first three requirements of condition (iii), so it remains to show that they fulfill the final requirement, \cite[(7.4)]{spathAMreduction}.

Now, if $\ell=3$ or $2$, then calculating with the automorphism group in GAP yields that the centralizer $C_{\aut(S)}(PZ(G)/Z(G))$ is an $\ell$-group, so this final requirement is satisfied by \cite[Proposition 7.4]{spathAMreduction}.

If $\ell=5$, then $|C_{\aut(S)}(PZ(G)/Z(G))|=10$ and this centralizer is cyclic.  Let $g$ be the order-2 element in $C_{\aut(S)}(PZ(G)/Z(G))$.  Now, $\langle S, g\rangle$ has order 720, and comparing the character table with those of $A_6.2_1, A_6.2_2,$ and $A_6.2_3$, we see that 
$\langle S, g\rangle=A_6.2_2$.  Moreover, the height-zero characters (in the notation of the GAP Character Table Library) of $G=6.A_6$ which are fixed under $g$ are $\chi_1, \chi_4, \chi_5, \chi_6, \chi_{10},$ and $\chi_{11}$ of degrees $1, 8,8,9,8,$ and $8$, respectively, and hence all other characters satisfy the final condition again by \cite[Proposition 7.4]{spathAMreduction}.

Our constructed bijections map these characters to characters of $N_G(P)$ with degree $1, 2, 2, 1, 2,$ and $2$, respectively, and we see that for these characters, $\epsilon_\chi\equiv-1\mod 5$, except in the case $\chi= \chi_1=1_G$, in which case $\epsilon_{1_G}=1$.  (Here $\epsilon_\chi$ is as defined following \cite[(7.4)]{spathAMreduction}.)  Further, $\aut(S)_\chi=\aut(S)$ for $\chi=\chi_1$ or $\chi_6$ and $\aut(S)_\chi=\langle S, g\rangle$ for the four characters of degree $8$ under consideration.  Also, $\chi_1, \chi_4, \chi_5,$ and $\chi_6$ lie in the principal block of $G$, and can be viewed as characters of $S=G/Z(G)$.  Similarly, the characters of $N_G(P)$ that they map to lie in the principal block of $N_G(P)$ and can be viewed as characters of $N_G(P)/Z(G)$.   Considering the character tables for $\aut(S)_{\chi}$ and $\aut(S)_{P,\chi}$, we see that these characters lift to characters of $\aut(S)_{\chi}$ and $\left(N_G(P)/Z(G)\right)\aut(S)_{P,\chi}$ satisfying the final condition of \cite[Definition 7.2]{spathAMreduction}.

The remaining two characters of $G$ and $N_G(P)$ under consideration are trivial on the elements of $Z(G)$ of order 3 and are nontrivial on the element $z\in Z(G)$ of order $2$.   Moreover, the values of $\chi_4$ and $\chi_{10}$ are identical on $2'$-elements and satisfy $\chi_4(x)=-\chi_{10}(x)$ when $2$ divides $|x|$.  The same is true for $\chi_5$ compared with $\chi_{11}$, and similarly for the corresponding pairs of characters of $N_G(P)$.  Hence if $\mathrm{rep}\colon S\rightarrow G$ is the $Z(G)$-section used for condition (iii) of \cite[Definition 7.2]{spathAMreduction} for the character $\chi=\chi_4$, respectively $\chi_5$, then replacing $\mathrm{rep}$ with $\mathrm{rep'}\colon y\mapsto \left\{\begin{array}{cc}
                                                                   \mathrm{rep}(y) & \hbox{if $2\not|\,|y|$} \\
                                                                   \mathrm{rep}(y)\cdot z &  \hbox{if $2||y|$}
                                                                 \end{array}\right.$ yields that
 condition (iii) of \cite[Definition 7.2]{spathAMreduction} is satisfied when $\chi=\chi_{10}$, respectively $\chi_{11}$, using the same extensions as in the case $\chi=\chi_4,$ respectively $\chi_5$.




Now let $G:=2.Sp_6(2)$ be the universal covering group of $S:=Sp_6(2)$ and let $\ell$ be a prime dividing $|G|$.  Then $\aut(G)\cong\aut(S)\cong S$, and in this case, the inductive AM-condition \cite[Definition 7.2]{spathAMreduction} is satisfied as long as the usual Alperin-McKay conjecture is satisfied.    Considerations in GAP similar to the case $A_6$ above and the situation for the BAWC, we can construct the desired bijections.  (We note that when $\ell\neq 3$, the only noncentral defect group is a Sylow subgroup, and in the case $\ell=3$, we have two noncentral defect groups, namely the Sylow subgroup and a cyclic defect group of size $3$.)
\end{proof}

Theorems \ref{thm:goodBAWC}, \ref{thm:goodAM}, and \ref{thm:goodAM2} complete the proof of the main theorem,  \prettyref{thm:Sp6good}.

\section{Acknowledgments}
This work was supported in part by a graduate fellowship from the National Physical Science Consortium and by the NSF (grant DMS-1201374) and was completed as part of my Ph.D. study at the University of Arizona.  I would like to thank my advisor, Professor P.H. Tiep, for all of his support, guidance, and helpful discussions.  I would also like to thank Professor K. Lux for numerous helpful discussions and for introducing me to computations with GAP, and Professor F. Himstedt of Technische Universit{\"a}t M{\"u}nchen for helping me get started with CHEVIE.  
\appendix
\section{The Series $\mathcal{E}_i(J)$}\label{app:seriesinfo}
We include in \prettyref{tab:lusztigseriesinfo} a list of the characters (in the notation of CHEVIE \cite{chevie}) in each $\mathcal{E}_i(J)=\mathcal{E}(G,(g_i(J)))$ for $G=Sp_6(2^a)$ and the number of such series for each $i$.  This information can be found in the character table information for $Sp_6(2^a)$ in CHEVIE, \cite{chevie} (details for obtaining character table information can be found in chapter 3 of the CHEVIE manual).  For the convenience of the reader, we include in \prettyref{tab:chardegsG} the list of the character degrees for $G$.
\begin{table}\caption{Lusztig Series of $G=Sp_6(2^a)$}\label{tab:lusztigseriesinfo}
\begin{tabular}{|c|c|c|c|c|}
  \hline
  \multirow{2}{*}{$\mathcal{E}_i(J)$} & \multirow{2}{*}{Indexing $J$} & $|J|=$ Number of  & \multirow{2}{*}{Characters in $\mathcal{E}_i(J)$} & \multirow{2}{*}{$|\mathcal{E}_i(J)|$}  \\
  &&Series of Type $i$&&\\
  \hline\hline
  $\mathcal{E}_1$ &  & 1 & $\chi_1$ through $\chi_{12}$ & 12\\
  \hline
  $\mathcal{E}_6(j)$ & $j\in I_{q-1}$ & $\frac{q-2}{2}$ & $\chi_{13}(j)$ through $\chi_{18}(j)$ & 6 \\
  \hline
  $\mathcal{E}_7(j)$ & $j\in I_{q+1}$ & $\frac{q}{2}$ & $\chi_{19}(j)$ through $\chi_{24}(j)$ & 6\\
  \hline
  $\mathcal{E}_8(j)$ & $j\in I_{q-1}$ & $\frac{q-2}{2}$ & $\chi_{25}(j)$ through $\chi_{27}(j)$ & 3\\
  \hline
  $\mathcal{E}_9(j)$ & $j\in I_{q+1}$ & $\frac{q}{2}$ & $\chi_{28}(j)$ through $\chi_{30}(j)$ & 3\\
  \hline
  $\mathcal{E}_{11}(j)$ & $j\in I_{q-1}$ & $\frac{q-2}{2}$ & $\chi_{31}(j)$ through $\chi_{34}(j)$ & 4\\
  \hline
  $\mathcal{E}_{13}(j)$ & $j\in I_{q+1}$ & $\frac{q}{2}$ & $\chi_{35}(j)$ through $\chi_{38}(j)$ & 4\\
  \hline
  $\mathcal{E}_{16}(j_1, j_2)$ & $(j_1, j_2) \in I^{2}_{q-1}$ & $\frac{(q-2)(q-4)}{4}$ & $\chi_{39}(j_1, j_2), \chi_{40}(j_1, j_2)$ & 2\\
  \hline
  $\mathcal{E}_{17}(j_1, j_2)$ & $(j_1, j_2) \in I^{2*}_{q-1}$ & $\frac{(q-2)(q-4)}{8}$ & $\chi_{41}(j_1, j_2), \chi_{42}(j_1, j_2)$ & 2\\
  \hline
  $\mathcal{E}_{18}(j_1, j_2)$ & $j_1\in I_{q-1}, j_2\in I_{q+1}$ & $\frac{q(q-2)}{4}$& $\chi_{43}(j_1, j_2), \chi_{44}(j_1, j_2)$ & 2 \\
  \hline
  $\mathcal{E}_{19}(j)$ & $j\in I_{q^2-1}$ & $\frac{q(q-2)}{4}$& $\chi_{45}(j), \chi_{46}(j)$ & 2 \\
  \hline
  $\mathcal{E}_{20}(j_1, j_2)$ & $j_1\in I_{q-1}, j_2\in I_{q+1}$ & $\frac{q(q-2)}{4}$& $\chi_{47}(j_1, j_2), \chi_{48}(j_1, j_2)$ & 2 \\
  \hline
  $\mathcal{E}_{21}(j_1, j_2)$ & $j_1\in I_{q+1}, j_2\in I_{q-1}$ & $\frac{q(q-2)}{4}$& $\chi_{49}(j_1, j_2), \chi_{50}(j_1, j_2)$ & 2 \\
  \hline
  $\mathcal{E}_{22}(j_1, j_2)$ & $(j_1, j_2) \in I^{2}_{q+1}$ & $\frac{q(q-2)}{4}$ & $\chi_{51}(j_1, j_2), \chi_{52}(j_1, j_2)$ & 2\\
  \hline
  $\mathcal{E}_{23}(j_1, j_2)$ & $(j_1, j_2) \in I^{2*}_{q+1}$ & $\frac{q(q-2)}{8}$ & $\chi_{53}(j_1, j_2), \chi_{54}(j_1, j_2)$ & 2\\
  \hline
  $\mathcal{E}_{24}(j)$ & $j\in I_{q^2+1}$ & $\frac{q^2}{4}$& $\chi_{55}(j), \chi_{56}(j)$ & 2 \\
  \hline
  $\mathcal{E}_{25}(j_1, j_2, j_3)$ & $(j_1, j_2, j_3) \in I^{3*}_{q-1}$ & $\frac{(q-2)(q-4)(q-6)}{48}$ & $\chi_{57}(j_1, j_2, j_3)$ & 1\\
  \hline
  $\mathcal{E}_{26}(j_1, j_2, j_3)$ & $(j_1, j_2) \in I^{2*}_{q-1}, j_3\in I_{q+1}$ & $\frac{q(q-2)(q-4)}{16}$ & $\chi_{58}(j_1, j_2, j_3)$ & 1\\
  \hline
  $\mathcal{E}_{27}(j_1, j_2)$ & $j_1 \in I_{q^2-1}, j_2\in I_{q-1}$ & $\frac{q(q-2)^2}{8}$ & $\chi_{59}(j_1, j_2)$ & 1\\
  \hline
  $\mathcal{E}_{28}(j_1, j_2, j_3)$ & $j_1 \in I_{q-1}, (j_2, j_3)\in I^{2*}_{q+1}$ & $\frac{q(q-2)^2}{16}$ & $\chi_{60}(j_1, j_2, j_3)$ & 1\\
  \hline
  $\mathcal{E}_{29}(j_1, j_2)$ & $j_1 \in I_{q^2-1}, j_2\in I_{q+1}$ & $\frac{q^2(q-2)}{8}$ & $\chi_{61}(j_1, j_2)$ & 1\\
  \hline
  $\mathcal{E}_{30}(j_1, j_2)$ & $j_1 \in I_{q-1}, j_2\in I_{q^2+1}$ & $\frac{q^2(q-2)}{8}$ & $\chi_{62}(j_1, j_2)$ & 1\\
  \hline
  $\mathcal{E}_{31}(j)$ & $j\in I_{q^3-1}$ & $\frac{q(q^2-1)}{6}$ & $\chi_{63}(j)$ & 1\\
  \hline
  $\mathcal{E}_{32}(j_1, j_2, j_3)$ & $(j_1, j_2, j_3) \in I^{3*}_{q+1}$ & $\frac{q(q-2)(q-4)}{48}$ & $\chi_{64}(j_1, j_2, j_3)$ & 1\\
  \hline
  $\mathcal{E}_{33}(j_1, j_2)$ & $j_1 \in I_{q^2+1}, j_2\in I_{q+1}$ & $\frac{q^3}{8}$ & $\chi_{65}(j_1, j_2)$ & 1\\
  \hline
   $\mathcal{E}_{34}(j)$ & $j\in I_{q^3+1}$ & $\frac{q(q^2-1)}{6}$ & $\chi_{66}(j)$ & 1\\
  \hline
\end{tabular}
\end{table}

\begin{table}\caption{Character Degrees for $G=Sp_6(2^a)$}\label{tab:chardegsG}
\scriptsize
$\begin{array}{|c|c|} \hline Character &Degree\\ \hline
\chi_{1}&1\\
\hline
\chi_{2}&\frac{1}{2}(q^{5}+q^{4}+2q^{3}+q^{2}+q)\\
\hline
\chi_{3}&\frac{1}{2}(q^{5}-q^{4}+2q^{3}-q^{2}+q)\\
\hline
\chi_{4}&\frac{1}{2}(q^{5}+q^{4}+q^{2}+q)\\
\hline
\chi_{5}&\frac{1}{2}(q^{5}-q^{4}-q^{2}+q)\\
\hline
\chi_{6}&q^{6}+q^{4}+q^{2}\\
\hline
\chi_{7}&q^{7}+q^{5}+q^{3}\\
\hline
\chi_{8}&\frac{1}{2}(q^{8}+q^{7}+2q^{6}+q^{5}+q^{4})\\
\hline
\chi_{9}&\frac{1}{2}(q^{8}+q^{7}+q^{5}+q^{4})\\
\hline
\chi_{10}&\frac{1}{2}(q^{8}-q^{7}+2q^{6}-q^{5}+q^{4})\\
\hline
\chi_{11}&\frac{1}{2}(q^{8}-q^{7}-q^{5}+q^{4})\\
\hline
\chi_{12}&q^{9}\\
\hline
\chi_{13}(j)&q^{5}+q^{4}+q^{3}+q^{2}+q+1\\
\hline
\chi_{14}(j)&\frac{1}{2}(q^{8}+3q^{7}+4q^{6}+4q^{5}+4q^{4}+4q^{3}+3q^{2}+q)\\
\hline
\chi_{15}(j)&\frac{1}{2}(q^{8}+q^{7}+2q^{6}+2q^{5}+2q^{4}+2q^{3}+q^{2}+q)\\
\hline
\chi_{16}(j)&\frac{1}{2}(q^{8}+q^{7}+2q^{6}+2q^{5}+2q^{4}+2q^{3}+q^{2}+q)\\
\hline
\chi_{17}(j)&\frac{1}{2}(q^{8}-q^{7}-q^{2}+q)\\
\hline
\chi_{18}(j)&q^{9}+q^{8}+q^{7}+q^{6}+q^{5}+q^{4}\\
\hline
\chi_{19}(j)&q^{5}-q^{4}+q^{3}-q^{2}+q-1\\
\hline
\chi_{20}(j)&\frac{1}{2}(q^{8}+q^{7}-q^{2}-q)\\
\hline
\chi_{21}(j)&\frac{1}{2}(q^{8}-q^{7}+2q^{6}-2q^{5}+2q^{4}-2q^{3}+q^{2}-q)\\
\hline
\chi_{22}(j)&\frac{1}{2}(q^{8}-q^{7}+2q^{6}-2q^{5}+2q^{4}-2q^{3}+q^{2}-q)\\
\hline
\chi_{23}(j)&\frac{1}{2}(q^{8}-3q^{7}+4q^{6}-4q^{5}+4q^{4}-4q^{3}+3q^{2}-q)\\
\hline
\chi_{24}(j)&q^{9}-q^{8}+q^{7}-q^{6}+q^{5}-q^{4}\\
\hline
\chi_{25}(j)&q^{6}+q^{5}+q^{4}+2q^{3}+q^{2}+q+1\\
\hline
\chi_{26}(j)&q^{8}+2q^{7}+2q^{6}+3q^{5}+3q^{4}+2q^{3}+2q^{2}+q\\
\hline
\chi_{27}(j)&q^{9}+q^{8}+q^{7}+2q^{6}+q^{5}+q^{4}+q^{3}\\
\hline
\chi_{28}(j)&q^{6}-q^{5}+q^{4}-2q^{3}+q^{2}-q+1\\
\hline
\chi_{29}(j)&q^{8}-2q^{7}+2q^{6}-3q^{5}+3q^{4}-2q^{3}+2q^{2}-q\\
\hline
\end{array}$ \bigskip
$\begin{array}{|c|c|} \hline Character & Degree\\ \hline
\chi_{30}(j)&q^{9}-q^{8}+q^{7}-2q^{6}+q^{5}-q^{4}+q^{3}\\
\hline
\chi_{31}(j)&q^{7}+q^{6}+2q^{5}+2q^{4}+2q^{3}+2q^{2}+q+1\\
\hline
\chi_{32}(j)&q^{8}+q^{7}+2q^{6}+2q^{5}+2q^{4}+2q^{3}+q^{2}+q\\
\hline
\chi_{33}(j)&q^{8}+q^{7}+2q^{6}+2q^{5}+2q^{4}+2q^{3}+q^{2}+q\\
\hline
\chi_{34}(j)&q^{9}+q^{8}+2q^{7}+2q^{6}+2q^{5}+2q^{4}+q^{3}+q^{2}\\
\hline
\chi_{35}(j)&q^{7}-q^{6}+2q^{5}-2q^{4}+2q^{3}-2q^{2}+q-1\\
\hline
\chi_{36}(j)&q^{8}-q^{7}+2q^{6}-2q^{5}+2q^{4}-2q^{3}+q^{2}-q\\
\hline
\chi_{37}(j)&q^{8}-q^{7}+2q^{6}-2q^{5}+2q^{4}-2q^{3}+q^{2}-q\\
\hline
\chi_{38}(j)&q^{9}-q^{8}+2q^{7}-2q^{6}+2q^{5}-2q^{4}+q^{3}-q^{2}\\
\hline
\chi_{39}(j_1,j_2)&q^{8}+2q^{7}+3q^{6}+4q^{5}+4q^{4}+4q^{3}+3q^{2}+2q+1\\
\hline
\chi_{40}(j_1,j_2)&q^{9}+2q^{8}+3q^{7}+4q^{6}+4q^{5}+4q^{4}+3q^{3}+2q^{2}+q\\
\hline
\chi_{41}(j_1,j_2)&q^{8}+2q^{7}+3q^{6}+4q^{5}+4q^{4}+4q^{3}+3q^{2}+2q+1\\
\hline
\chi_{42}(j_1,j_2)&q^{9}+2q^{8}+3q^{7}+4q^{6}+4q^{5}+4q^{4}+3q^{3}+2q^{2}+q\\
\hline
\chi_{43}(j_1,j_2)&q^{8}+q^{6}-q^{2}-1\\
\hline
\chi_{44}(j_1,j_2)&q^{9}+q^{7}-q^{3}-q\\
\hline
\chi_{45}(j)&q^{8}+q^{6}-q^{2}-1\\
\hline
\chi_{46}(j)&q^{9}+q^{7}-q^{3}-q\\
\hline
\chi_{47}(j_1,j_2)&q^{8}+q^{6}-q^{2}-1\\
\hline
\chi_{48}(j_1,j_2)&q^{9}+q^{7}-q^{3}-q\\
\hline
\chi_{49}(j_1,j_2)&q^{8}+q^{6}-q^{2}-1\\
\hline
\chi_{50}(j_1,j_2)&q^{9}+q^{7}-q^{3}-q\\
\hline
\chi_{51}(j_1,j_2)&q^{8}-2q^{7}+3q^{6}-4q^{5}+4q^{4}-4q^{3}+3q^{2}-2q+1\\
\hline
\chi_{52}(j_1,j_2)&q^{9}-2q^{8}+3q^{7}-4q^{6}+4q^{5}-4q^{4}+3q^{3}-2q^{2}+q\\
\hline
\chi_{53}(j_1,j_2)&q^{8}-2q^{7}+3q^{6}-4q^{5}+4q^{4}-4q^{3}+3q^{2}-2q+1\\
\hline
\chi_{54}(j_1,j_2)&q^{9}-2q^{8}+3q^{7}-4q^{6}+4q^{5}-4q^{4}+3q^{3}-2q^{2}+q\\
\hline
\chi_{55}(j)&q^{8}-q^{6}-q^{2}+1\\
\hline
\chi_{56}(j)&q^{9}-q^{7}-q^{3}+q\\
\hline
\chi_{57}(j_1,j_2,j_3)&q^{9}+3q^{8}+5q^{7}+7q^{6}+8q^{5}+8q^{4}+7q^{3}+5q^{2}+3q+1\\
\hline
\chi_{58}(j_1,j_2,j_3)&q^{9}+q^{8}+q^{7}+q^{6}-q^{3}-q^{2}-q-1\\
\hline
\end{array}$ \bigskip
$\begin{array}{|c|c|} \hline Character&Degree\\ \hline
\chi_{59}(j_1,j_2)&q^{9}+q^{8}+q^{7}+q^{6}-q^{3}-q^{2}-q-1\\
\hline
\chi_{60}(j_1,j_2,j_3)&q^{9}-q^{8}+q^{7}-q^{6}-q^{3}+q^{2}-q+1\\
\hline
\chi_{61}(j_1,j_2)&q^{9}-q^{8}+q^{7}-q^{6}-q^{3}+q^{2}-q+1\\
\hline
\chi_{62}(j_1,j_2)&q^{9}+q^{8}-q^{7}-q^{6}-q^{3}-q^{2}+q+1\\
\hline
\chi_{63}(j)&q^{9}-q^{7}+q^{6}-q^{5}-q^{4}+q^{3}-q^{2}+1\\
\hline
\chi_{64}(j_1,j_2,j_3)&q^{9}-3q^{8}+5q^{7}-7q^{6}+8q^{5}-8q^{4}+7q^{3}-5q^{2}+3q-1\\
\hline
\chi_{65}(j_1,j_2)&q^{9}-q^{8}-q^{7}+q^{6}-q^{3}+q^{2}+q-1\\
\hline
\chi_{66}(j)&q^{9}-q^{7}-q^{6}-q^{5}+q^{4}+q^{3}+q^{2}-1\\
\hline
\end{array}$
\end{table} 
\section{Non-Unipotent Block Distributions for $Sp_6(2^a)$}\label{app:Block Dist}
We list here the cross-characteristic block distribution for irreducible ordinary characters of positive defect lying in non-unipotent blocks of $G=Sp_6(q)$ with $q$ even.  Any characters not listed have defect zero or lie in a unipotent block, whose distribution can be found in \cite{white2000}.  We use the notation of characters for $G$, in CHEVIE \cite{chevie}.  The notation $B_i(J)$, as before, means the block in $\mathcal{E}_\ell(G, (g_i(J)))$ of positive defect, when only one such block exists.  (Here $g_i(J)$ is a semisimple $\ell'$-element of $G^\ast$, as in the notation of \cite[Tabelle 19]{Luebeckthesis}.)  When multiple such blocks exist, we will write $B_i(J)^{(0)}$ for the block corresponding via Morita equivalence (see \prettyref{lem:moritaequiv}) to the principal block of $C_{G^\ast}(g_i(J))$ and $B_i(J)^{(1)}$ for the other block of positive defect.  (Indeed, in such cases there are only two blocks of positive defect in $\mathcal{E}_\ell(G, (g_i(J)))$.)

The block distributions listed in this section follow from the theory of central characters (the central characters of $G$ can be obtained from CHEVIE \cite{chevie}) together with the definition of $\mathcal{E}_\ell(G, (g_i(J)))$ and Brou{\'e}-Michel's result \cite{brouemichel} that this is a union of $\ell$-blocks..

\subsection{$\ell|(q^2+1)$}

Let $k\in I_{q-1}, t\in I_{q+1}, s\in I_{q^2+1}$ with $(q^2+1)_{\ell}|s$ and write $m:=(q^2+1)_{\ell'}$.

\[\irr\left(B_6(k)\right)=\{\chi_{13}(k), \chi_{14}(k), \chi_{17}(k), \chi_{18}(k), \chi_{62}(k,r)\colon m|r\}\]

\[\irr\left(B_7(t)\right)=\{\chi_{19}(t), \chi_{20}(t), \chi_{23}(t), \chi_{24}(t), \chi_{65}(r,t)\colon m|r\}\]

\[\irr\left(B_{24}(s)^{(0)}\right)=\{\chi_{55}(r)\colon r\equiv \pm s\hbox{ or } \pm qs\mod m\}\]

\[\irr\left(B_{24}(s)^{(1)}\right)=\{\chi_{56}(r)\colon r\equiv \pm s\hbox{ or } \pm qs\mod m\}\]

\[\irr\left(B_{30}(k,s)\right)=\{\chi_{62}(k, r)\colon r\equiv \pm s\hbox{ or } \pm qs\mod m\}\]

\[\irr\left(B_{33}(s,t)\right)=\{\chi_{65}(r, t)\colon r\equiv \pm s\hbox{ or } \pm qs\mod m\}\]

\subsection{$3\neq\ell|(q^2+q+1)$}
In the following, let $k\in I_{q-1}$ and $v\in I_{q^3-1}$ with $(q^3-1)_{\ell}|v$ and write $n:=(q^2+q+1)_{\ell'}$.

\[\irr\left(B_8(k)\right)=\{\chi_{25}(k), \chi_{26}(k), \chi_{27}(k), \chi_{63}(r)\colon r\equiv \pm k (q^2+q+1)\mod (q-1)n\}\]

\[\irr\left(B_{31}(v)\right)=\{\chi_{63}(r)\colon r\equiv\pm v, \pm qv, \hbox{ or }\pm q^2v\mod(q-1)n\}\]

\subsection{$3\neq\ell|(q^2-q+1)$}
In the following, let $t\in I_{q+1}$ and $w\in I_{q^3+1}$ with $(q^3+1)_{\ell}|w$ and write $n:=(q^2-q+1)_{\ell'}$.

\[\irr\left(B_9(t)\right)=\{\chi_{28}(t), \chi_{29}(t), \chi_{30}(t), \chi_{66}(r)\colon r\equiv \pm t (q^2-q+1)\mod (q+1)n\}\]

\[\irr\left(B_{34}(w)\right)=\{\chi_{66}(r)\colon r\equiv\pm w, \pm qw, \hbox{ or }\pm q^2w\mod(q+1)n\}\]

\subsection{$\ell|(q-1)$}
In the following, let $k_1, k_2, k_3\in I_{q-1}$ with $\ell^d|k_i$ and none of $k_1, k_2, k_3$ the same.  Let $t_1, t_2, t_3 \in I_{q+1}$ with none of $t_1, t_2, t_3$ the same, $u\in I_{q^2+1}$, and $s\in I_{q^2-1}$ with $\ell^d|s$, where $\ell^d:=(q-1)_{\ell}$.  Let $v\in I_{q^3-1}$ and $w\in I_{q^3+1}$ with $(q^3-1)_{\ell}|v$.  As usual, let $m:=(q-1)_{\ell'}$.  When $\ell=3$, let $n:=(q^2+q+1)_{3'}$.

\begin{align*}
\irr\left(B_6(k_1)^{(0)}\right)=\{\chi_{13}(r), \chi_{14}(r), \chi_{15}(r), \chi_{16}(r), \chi_{18}(r), \chi_{39}(j, r), \chi_{40}(j, r), &\chi_{41}(j, r), \chi_{42}(j, r), \chi_{57}(r, j, i) \colon\\& r\equiv\pm k_1\mod m, m|j, m|i\}
\end{align*}

\[\irr\left(B_6(k_1)^{(1)}\right)=\{\chi_{17}(r)\colon r\equiv\pm k_1\mod m\}\]

\begin{align*}
\irr\left(B_7(t_1)\right)=\{\chi_{19}(t_1), \chi_{20}(t_1), \chi_{21}(t_1), \chi_{22}(t_1), \chi_{24}(t_1), \chi_{43}(r, t_1), \chi_{44}(r, t_1), & \chi_{47}(r, t_1), \chi_{48}(r, t_1), \chi_{58}(r, j, t_1)\colon\\ & m|r, m|j\}
\end{align*}

\noindent(Note: $\E_\ell(G, g_{7}(t_1))$ also contains the defect-zero block $\{\chi_{23}(t_1)\}$.)

\[\irr\left(B_8(k_1)\right)=\left\{\begin{array}{ll}
\{\chi_{25}(r_1), \chi_{26}(r_1), \chi_{27}(r_1), \chi_{39}(r_1, r_2), \chi_{40}(r_1, r_2), \chi_{57}(r_1, r_2, r_3)\colon & \multirow{2}{*}{\hbox{if $\ell\neq 3$}}\\
 r_1,r_2,r_3\equiv \pm k_1\mod m\} & \\
\{\chi_{25}(r_1), \chi_{26}(r_1), \chi_{27}(r_1), \chi_{39}(r_1, r_2), \chi_{40}(r_1, r_2), \chi_{57}(r_1, r_2, r_3), \chi_{63}(r_4)\colon & \multirow{2}{*}{\hbox{if $\ell= 3$}}\\
r_1,r_2,r_3\equiv \pm k_1\mod m, r_4\equiv \pm k_1 (q^2+q+1)\mod mn\} & \\
\end{array}\right.\]

\[\irr\left(B_9(t_1)\right)=\{\chi_{28}(t_1), \chi_{30}(t_1), \chi_{61}(r, t_1)\colon r\equiv\pm(q-1)t_1\mod m(q+1)\}\]

 \noindent(Note: $\E_\ell(G, g_{9}(t_1))$ also contains the defect-zero block $\{\chi_{29}(t_1)\}$.)

\begin{align*}
\irr\left(B_{11}(k_1)\right)=\{&\chi_{31}(r_1), \chi_{32}(r_1), \chi_{33}(r_1), \chi_{34}(r_1), \chi_{41}(r_1,r_2), \chi_{42}(r_1, r_2), \\ & \chi_{39}(r_1, j), \chi_{40}(r_1, j), \chi_{57}(r_1, r_2, j) \colon r_1\equiv\pm k_1\mod m, r_2\equiv\pm k_1\mod m, m|j\}
\end{align*}

\begin{align*}
\irr\left(B_{13}(t_1)\right)=\{&\chi_{35}(t_1), \chi_{36}(t_1), \chi_{37}(t_1), \chi_{38}(t_1),  \chi_{49}(t_1, r), \chi_{50}(t_1, r), \\ & \chi_{45}(j), \chi_{46}(j), \chi_{59}(j, r) \colon j\equiv\pm(q-1)t_1\mod m(q+1), m|r\}
\end{align*}

\[\irr\left(B_{16}(k_1, k_2)\right)=\{\chi_{39}(r_1, r_2), \chi_{40}(r_1, r_2), \chi_{57}(r_1, j, r_2)\colon r_i\equiv\pm k_i\mod m, j\equiv\pm k_1\mod m\}\]

\[\irr\left(B_{17}(k_1, k_2)\right)=\{\chi_{41}(r_1, r_2), \chi_{42}(r_1, r_2), \chi_{57}(r_1, r_2, j)\colon r_i\equiv \pm k_i\mod m, m|j\}\]

\[\irr\left(B_{18}(k_1, t_1)\right)=\{\chi_{43}(r, t_1), \chi_{44}(r, t_1), \chi_{58}(r, j, t_1)\colon r\equiv\pm k_1\mod m, j\equiv\pm k_1\mod m\}\]

\[\irr\left(B_{19}(s)\right)=\{\chi_{45}(r), \chi_{46}(r), \chi_{59}(r, j)\colon r\equiv\pm s\hbox{ or }\pm qs \mod m(q+1), m|j\}\]

\[\irr\left(B_{20}(k_1, t_1)\right)=\{\chi_{47}(r, t_1), \chi_{48}(r, t_1), \chi_{58}(r, j, t_1)\colon r\equiv\pm k_1\mod m, m|j\}\]

\begin{align*}
\irr\left(B_{21}(t_1, k_1)\right)=\{\chi_{49}(t_1, r), \chi_{50}(t_1, r), \chi_{59}(j, r) \colon &
j\equiv\pm(q-1)t_1\mod m(q+1),\\& r\equiv\pm k_1\mod m\}
\end{align*}

\[\irr\left(B_{22}(t_1, t_2)\right)=\{\chi_{51}(t_1, t_2), \chi_{52}(t_1, t_2), \chi_{61}(r, t_2)\colon r\equiv\pm(q-1)t_1\mod m(q+1)\}\]

\[\irr\left(B_{23}(t_1, t_2)\right)=\{\chi_{53}(t_1, t_2), \chi_{54}(t_1, t_2), \chi_{60}(r, t_1, t_2)\colon m|r\}\]

\[\irr\left(B_{24}(u)\right)=\{\chi_{55}(u), \chi_{56}(u), \chi_{62}(r, u) \colon r\equiv 0\mod m\}\]

\[\irr\left(B_{25}(k_1, k_2, k_3)\right)=\{\chi_{57}(r_1, r_2, r_3)\colon r_i\equiv \pm k_i\mod m\}\]

\[\irr\left(B_{26}(k_1, k_2, t_1)\right)=\{\chi_{58}(r_1, r_2, t_1)\colon r_i\equiv\pm k_i\mod m\}\]

\[\irr\left(B_{27}(s, k_1)\right)=\{\chi_{59}(r, j)\colon r\equiv\pm s\hbox{ or }\pm qs \mod m(q+1), j\equiv \pm k_1\mod m\}\]

\[\irr\left(B_{28}(k_1, t_1, t_2)\right)=\{\chi_{60}(r, t_1, t_2)\colon r\equiv\pm k_1\mod m\}\]

\[\irr\left(B_{29}(s, t_1)\right)=\{\chi_{61}(r, t_1)\colon r\equiv\pm s\hbox{ or }\pm qs \mod m(q+1)\}\]

\[\irr\left(B_{30}(k_1, u)\right)=\{\chi_{62}(r, u)\colon r\equiv\pm k_1\mod m\}\]

\[\irr\left(B_{31}(v)\right)=\left\{\begin{array}{ll}
\{\chi_{63}(r)\colon r\equiv \pm v, \pm qv,\hbox{ or }\pm q^2v\mod m(q^2+q+1)\} & \hbox{if $\ell\neq 3$}\\
\{\chi_{63}(r)\colon r\equiv \pm v, \pm qv,\hbox{ or }\pm q^2v\mod mn\} & \hbox{if $\ell= 3$}\\
\end{array}\right.\]

\[\irr\left(B_{32}(t_1, t_2, t_3)\right)=\{\chi_{64}(t_1, t_2, t_3)\}\quad\hbox{(defect zero)}\]

\[\irr\left(B_{33}(u, t_1)\right)=\{\chi_{65}(u, t_1)\}\quad\hbox{(defect zero)}\]

\[\irr\left(B_{34}(w)\right)=\{\chi_{66}(w)\}\quad\hbox{(defect zero)}\]

\subsection{$\ell|(q+1)$}
In the following, let $k_1, k_2, k_3\in I_{q-1}$ with none of $k_1, k_2, k_3$ the same.  Let $t_1, t_2, t_3 \in I_{q+1}$ with $\ell^d|t_i$ and none of $t_1, t_2, t_3$ the same, $u\in I_{q^2+1}$, and $s\in I_{q^2-1}$ with $\ell^d|s$, where $\ell^d:=(q+1)_\ell$.  Let $v\in I_{q^3-1}$ and $w\in I_{q^3+1}$ with $(q^3+1)_\ell|w$.  As usual, let $m:=(q+1)_{\ell'}$.  When $\ell=3$, write $n:=(q^2-q+1)_{3'}$.

\begin{align*}
\irr\left(B_6(k_1)\right)=\{\chi_{13}(k_1), \chi_{15}(k_1), \chi_{16}(k_1), \chi_{17}(k_1), \chi_{18}(k_1),& \chi_{47}(k_1, r), \chi_{48}(k_1, r), \\ & \chi_{49}(r, k_1), \chi_{50}(r, k_1), \chi_{60}(k_1, r, j)\colon m|r, m|j\}
\end{align*}

\noindent(Note: $\E_\ell(G, g_{6}(k_1))$ also contains the defect-zero block $\{\chi_{14}(k_1)\}$.)

\begin{align*}
\irr\left(B_7(t_1)^{(0)}\right)=\{\chi_{19}(r), \chi_{21}(r), \chi_{22}(r), \chi_{23}(r), \chi_{24}(r), \chi_{51}(j, r), \chi_{52}(j, r), &\chi_{53}(j, r), \chi_{54}(j, r), \chi_{64}(r, j, i) \colon\\& r\equiv\pm t_1\mod m, m|j, m|i\}
\end{align*}

\[\irr\left(B_7(t_1)^{(1)}\right)=\{\chi_{20}(r)\colon r\equiv\pm t_1\mod m\}\]

\[\irr\left(B_8(k_1)\right)=\{\chi_{25}(k_1), \chi_{27}(k_1), \chi_{59}(r, k_1)\colon r\equiv\pm(q+1)k_1\mod m(q-1)\}\]

 \noindent(Note: $\E_\ell(G, g_{8}(k_1))$ also contains the defect-zero block $\{\chi_{26}(k_1)\}$.)

\[\irr\left(B_9(t_1)\right)=\left\{\begin{array}{ll}
\{\chi_{28}(r_1), \chi_{29}(r_1), \chi_{30}(r_1), \chi_{51}(r_1, r_2), \chi_{52}(r_1, r_2), \chi_{64}(r_1, r_2, r_3)\colon & \multirow{2}{*}{\hbox{ if $\ell\neq 3$,}}\\
r_1,r_2,r_3\equiv \pm t_1\mod m\}& \\
\{\chi_{28}(r_1), \chi_{29}(r_1), \chi_{30}(r_1), \chi_{51}(r_1, r_2), \chi_{52}(r_1, r_2), \chi_{64}(r_1, r_2, r_3), \chi_{66}(r_4)\colon & \multirow{2}{*}{\hbox{ if $\ell= 3$,}}\\
r_1,r_2,r_3\equiv \pm t_1\mod m, r_4\equiv \pm t_1 (q^2-q+1)\mod mn\}& \\
\end{array}\right.\]

\begin{align*}
\irr\left(B_{11}(k_1)\right)= \{&\chi_{31}(k_1), \chi_{32}(k_1), \chi_{33}(k_1), \chi_{34}(k_1),  \chi_{43}(k_1, r), \chi_{44}(k_1, r), \\ & \chi_{45}(j), \chi_{46}(j), \chi_{61}(j, r) \colon j\equiv\pm(q+1)k_1\mod m(q-1), m|r\}
\end{align*}

\begin{align*}
\irr\left(B_{13}(t_1)\right)=\{&\chi_{35}(r_1), \chi_{36}(r_1), \chi_{37}(r_1), \chi_{38}(r_1), \chi_{53}(r_1,r_2), \chi_{54}(r_1, r_2), \\ & \chi_{51}(r_1, j), \chi_{52}(r_1, j), \chi_{64}(r_1, r_2, j) \colon r_1\equiv\pm t_1\mod m, r_2\equiv\pm t_1\mod m, m|j\}
\end{align*}

\[\irr\left(B_{16}(k_1, k_2)\right)=\{\chi_{39}(k_1, k_2), \chi_{40}(k_1, k_2), \chi_{59}(r, k_2)\colon r\equiv\pm(q+1)k_1\mod m(q-1)\}\]

\[\irr\left(B_{17}(k_1, k_2)\right)=\{\chi_{41}(k_1, k_2), \chi_{42}(k_1, k_2), \chi_{58}(k_1, k_2, r)\colon m|r\}\]

\begin{align*}
\irr\left(B_{18}(k_1, t_1)\right)=\{\chi_{43}(k_1, r), \chi_{44}(k_1, r), \chi_{61}(j, r) \colon &
j\equiv\pm(q+1)k_1\mod m(q-1),\\& r\equiv\pm t_1\mod m\}
\end{align*}

\[\irr\left(B_{19}(s)\right)=\{\chi_{45}(r), \chi_{46}(r), \chi_{61}(r, j)\colon r\equiv\pm s\hbox{ or }\pm qs \mod m(q-1), m|j\}\]

\[\irr\left(B_{20}(k_1, t_1)\right)=\{\chi_{47}(k_1, r), \chi_{48}(k_1, r), \chi_{60}(k_1, r, j)\colon r\equiv\pm t_1\mod m, m|j\}\]

\[\irr\left(B_{21}(t_1, k_1)\right)=\{\chi_{49}(r, k_1), \chi_{50}(r, k_1), \chi_{60}(k_1, r, j)\colon r\equiv\pm t_1\mod m, j\equiv\pm t_1\mod m\}\]

\[\irr\left(B_{22}(t_1, t_2)\right)=\{\chi_{51}(r_1, r_2), \chi_{52}(r_1, r_2), \chi_{64}(r_1, j, r_2)\colon r_i\equiv\pm t_i\mod m, j\equiv\pm t_1\mod m\}\]

\[\irr\left(B_{23}(t_1, t_2)\right)=\{\chi_{53}(r_1, r_2), \chi_{54}(r_1, r_2), \chi_{64}(r_1, r_2, j)\colon r_i\equiv \pm t_i\mod m, m|j\}\]

\[\irr\left(B_{24}(u)\right)=\{\chi_{55}(u), \chi_{56}(u), \chi_{65}(u, r) \colon r\equiv 0\mod m\}\]

\[\irr\left(B_{25}(k_1, k_2, k_3)\right)=\{\chi_{57}(k_1, k_2, k_3)\}\quad\hbox{(defect zero)}\]

\[\irr\left(B_{26}(k_1, k_2, t_1)\right)=\{\chi_{58}(k_1, k_2, r)\colon r\equiv\pm t_1\mod m\}\]

\[\irr\left(B_{27}(s, k_1)\right)=\{\chi_{59}(r, k_1)\colon r\equiv\pm s\hbox{ or }\pm qs \mod m(q-1)\}\]

\[\irr\left(B_{28}(k_1, t_1, t_2)\right)=\{\chi_{60}(k_1, r_1, r_2)\colon r_i\equiv\pm t_i\mod m\}\]

\[\irr\left(B_{29}(s, t_1)\right)=\{\chi_{61}(r, j)\colon r\equiv\pm s\hbox{ or }\pm qs \mod m(q-1), j\equiv \pm t_1\mod m\}\]

\[\irr\left(B_{30}(k_1, u)\right)=\{\chi_{62}(k_1,u)\}\quad\hbox{(defect zero)}\]

\[\irr\left(B_{31}(v)\right)=\{\chi_{63}(v)\}\quad\hbox{(defect zero)}\]

\[\irr\left(B_{32}(t_1, t_2, t_3)\right)=\{\chi_{64}(r_1, r_2, r_3)\colon r_i\equiv \pm t_i\mod m\}\]

\[\irr\left(B_{33}(u, t_1)\right)=\{\chi_{65}(u, r)\colon r\equiv\pm t_1\mod m\}\]
\[\irr\left(B_{34}(w)\right)=\left\{\begin{array}{ll}
             \{\chi_{66}(r)\colon r\equiv \pm w, \pm qw,\hbox{ or }\pm q^2w\mod m(q^2-q+1)\} & \hbox{if $\ell\neq 3$,} \\
             \{\chi_{66}(r)\colon r\equiv \pm w, \pm qw,\hbox{ or }\pm q^2w\mod mn\} & \hbox{if $\ell=3$}
           \end{array}\right.\]

\section{Non-Unipotent Brauer Characters for $Sp_6(2^a)$}\label{app:BrauerChars}
\prettyref{tab:NonUnipBrauerSp6} gives the irreducible Brauer characters of $G=Sp_6(q)$, $q$ even, listed by the families $\mathcal{E}_\ell(G, (t))$ for $\ell'$-semisimple elements $t\in G^\ast$.  The indexing sets for $t=g_k$ are as given in \prettyref{app:Block Dist} for $B_k$. Characters listed in the same set for the same choice of $t$ make up the Brauer characters of a single block.  (For example, the two sets $\{\wh{\chi}_{13}(k_1), \wh{\chi}_{14}(k_1), \wh{\chi}_{15}(k_1), \wh{\chi}_{16}(k_1), \wh{\chi}_{18}(k_1)\}$ and $ \{\wh{\chi}_{17}(k_1)\}$ are listed as $\ibr_\ell(G)\cap \mathcal{E}_\ell(G,(g_6(k_1)))$ when $\ell|(q-1)$, corresponding to the two blocks $B_6(k_1)^{(0)}$ and $B_6(k_1)^{(1)}$.)
Notation for the characters of $G$ is taken from CHEVIE \cite{chevie}, and the notation for the class representatives $t\in G^\ast$ is from \cite[Tabelle 19]{Luebeckthesis}.  As usual, $\wh{\chi}$ denotes the restriction of $\chi\in\irr(G)$ to $\ell$-regular elements of $G$.

The results in the table follow from \prettyref{lem:moritaequiv}, \prettyref{lem:centralizersemisimpleSp}, and the decomposition numbers for the unipotent blocks for the low-rank groups.  The decomposition matrices for the unipotent blocks of $SL_2(q)$ (and therefore $Sp_2(q)=SL_2(q), GL_2(q)=C_{q-1}\times SL_2(q),$ and $GU_2(q)=C_{q+1}\times SL_2(q)$) and $GL_3(q)$ can be obtained from \cite{james1990}, and those for $Sp_4(q)$ are found in \cite{white95}.  Note that the number $\alpha$ found in the description of the Brauer characters of $\mathcal{E}_\ell(G,(t))$ for $t$ in the family $g_6$ or $g_7$ when $\ell|(q+1)$ is as in \cite{white95}, and by \cite{okuyamawakiSp4}, we have $\alpha=1$ when $\ell^d=3$ and $\alpha=2$ otherwise.  The decomposition matrices for the unipotent blocks of $GU_3(q)$ were found in \cite{geckdecompSU}, up to an unknown in the case $\ell|(q+1)$, which is found in \cite{okuyamawakiSU3}.

\begin{table}\caption{$\ell-$Brauer Characters in Non-Unipotent Blocks of $G=Sp_6(2^a)$, $\ell\neq 2$}\label{tab:NonUnipBrauerSp6}
\tiny
\begin{tabular}{|c|c|c|}
  \hline
  $t$ &  Condition & \multirow{2}{*}{$\ibr_\ell(G)\cap \mathcal{E}_\ell(G,(t))$}  \\
  $C_{G^\ast}(t)$ & on $\ell$ &\\
  \hline
  \multirow{3}{*}{$g_6(k_1)$} & $\ell|(q-1)$  & $\{\wh{\chi}_{13}(k_1), \wh{\chi}_{14}(k_1), \wh{\chi}_{15}(k_1), \wh{\chi}_{16}(k_1), \wh{\chi}_{18}(k_1)\},\quad \{\wh{\chi}_{17}(k_1)\}$\\
  & \multirow{4}{*}{$\ell|(q+1)$}   & \\
  &&$\{\wh{\chi}_{13}(k_1), \wh{\chi}_{15}(k_1)-\wh{\chi}_{13}(k_1), \wh{\chi}_{16}(k_1)-\wh{\chi}_{13}(k_1), \wh{\chi}_{17}(k_1)$, \\  &&$\wh{\chi}_{18}(k_1)-\alpha\wh{\chi}_{17}(k_1)-\wh{\chi}_{16}(k_1)-\wh{\chi}_{15}(k_1)+\wh{\chi}_{13}(k_1)\},\quad \{\wh{\chi}_{14}(k_1)\}$\\
  \multirow{3}{*}{$C_{q-1}\times Sp_4(q)$}&&\\
   & $\ell|(q^2+1)$  & $\{\wh{\chi}_{13}(k_1), \wh{\chi}_{14}(k_1)-\wh{\chi}_{13}(k_1), \wh{\chi}_{18}(k_1)-\wh{\chi}_{14}(k_1)+\wh{\chi}_{13}(k_1), \wh{\chi}_{17}(k_1)\}, \quad\{\wh{\chi}_{15}(k_1)\}, \{\wh{\chi}_{16}(k_1)\}$\\
   & $\ell\not|(q^4-1)$  & $\{\wh{\chi}_{13}(k_1)\}, \{\wh{\chi}_{14}(k_1)\}, \{ \wh{\chi}_{15}(k_1)\}, \{ \wh{\chi}_{16}(k_1)\}, \{\wh{\chi}_{17}(k_1)\},\{\wh{\chi}_{18}(k_1)\}$\\
  \hline
  \multirow{3}{*}{$g_7(t_1)$} & $\ell|(q-1)$ & $\{\wh{\chi}_{19}(t_1), \wh{\chi}_{20}(t_1), \wh{\chi}_{21}(t_1), \wh{\chi}_{22}(t_1), \wh{\chi}_{24}(t_1)\},\quad \{\wh{\chi}_{23}(t_1)\}$\\
   & \multirow{4}{*}{$\ell|(q+1)$}  &\\
    & & $\{\wh{\chi}_{19}(t_1), \wh{\chi}_{21}(t_1)-\wh{\chi}_{19}(t_1), \wh{\chi}_{22}(t_1)-\wh{\chi}_{19}(t_1), \wh{\chi}_{23}(t_1),$ \\
    & & $\wh{\chi}_{24}(t_1)-\alpha\wh{\chi}_{23}(t_1)-\wh{\chi}_{22}(t_1)-\wh{\chi}_{21}(t_1)+\wh{\chi}_{19}(t_1)\}, \quad \{\wh{\chi}_{20}(t_1)\}$\\
   \multirow{3}{*}{$C_{q+1}\times Sp_4(q)$} &&\\
   & $\ell|(q^2+1)$ &  $\{\wh{\chi}_{19}(t_1), \wh{\chi}_{20}(t_1)-\wh{\chi}_{19}(t_1), \wh{\chi}_{24}(t_1)-\wh{\chi}_{20}(t_1)+\wh{\chi}_{19}(t_1), \wh{\chi}_{23}(t_1)\}, \quad\{\wh{\chi}_{21}(t_1)\}, \{\wh{\chi}_{22}(t_1)\}$\\
   & $\ell\not|(q^4-1)$ &  $\{\wh{\chi}_{19}(t_1)\}, \{\wh{\chi}_{20}(t_1)\}, \{ \wh{\chi}_{21}(t_1)\}, \{ \wh{\chi}_{22}(t_1)\}, \{\wh{\chi}_{23}(t_1)\},\{\wh{\chi}_{24}(t_1)\}$\\
  \hline
  \multirow{2}{*}{$g_8(k_1)$} & $3\neq\ell|(q-1)$ &  $\{\wh{\chi}_{25}(k_1), \wh{\chi}_{26}(k_1), \wh{\chi}_{27}(k_1)\}$ \\
   & $\ell|(q+1)$ &   $\{\wh{\chi}_{25}(k_1), \wh{\chi}_{27}(k_1)-\wh{\chi}_{25}(k_1)\},\quad\{\wh{\chi}_{26}(k_1)\}$ \\
  \multirow{2}{*}{$GL_3(q)$}  & $\ell|(q^2+q+1)$  & $\{\wh{\chi}_{25}(k_1), \wh{\chi}_{26}(k_1)-\wh{\chi}_{25}(k_1), \wh{\chi}_{27}(k_1)-\wh{\chi}_{26}(k_1)+\wh{\chi}_{25}(k_1)\}$ \\
    & $\ell\not|(q^3-1)(q+1)$ &  $\{\wh{\chi}_{25}(k_1)\}, \{\wh{\chi}_{26}(k_1)\}, \{\wh{\chi}_{27}(k_1)\}$ \\
    \hline
  \multirow{2}{*}{$g_9(t_1)$} & $\ell|(q-1)$ & $\{\wh{\chi}_{28}(t_1), \wh{\chi}_{30}(t_1)\}, \{ \wh{\chi}_{29}(t_1)\}$\\
  & $\ell|(q+1)$ &  $\{\wh{\chi}_{28}(t_1), \wh{\chi}_{29}(t_1), \wh{\chi}_{30}(t_1)-2\wh{\chi}_{29}(t_1)-\wh{\chi}_{28}(t_1)\}$\\
  \multirow{2}{*}{$GU_3(q)$} & $3\neq\ell|(q^2-q+1)$ &  $\{\wh{\chi}_{28}(t_1), \wh{\chi}_{29}(t_1), \wh{\chi}_{30}(t_1)-\wh{\chi}_{28}(t_1)\}$ \\
   & $\ell\not|(q^3+1)(q-1)$ & $\{\wh{\chi}_{28}(t_1)\}, \{\wh{\chi}_{29}(t_1)\}, \{\wh{\chi}_{30}(t_1)\}$\\
  \hline
  \multirow{2}{*}{$g_{11}(k_1)$} & $\ell|(q-1)$ &   $\{\wh{\chi}_{31}(k_1),\wh{\chi}_{32}(k_1),\wh{\chi}_{33}(k_1),\wh{\chi}_{34}(k_1)\}$ \\
   & $\ell|(q+1)$ &    $\{\wh{\chi}_{31}(k_1), \wh{\chi}_{32}(k_1)-\wh{\chi}_{31}(k_1), \wh{\chi}_{33}(k_1)- \wh{\chi}_{31}(k_1), \wh{\chi}_{34}(k_1)-\wh{\chi}_{33}(k_1)-\wh{\chi}_{32}(k_1)+\wh{\chi}_{31}(k_1)\}$ \\
   $GL_2(q)\times Sp_2(q)$ & $\ell\not|(q^2-1)$ &  $\{\wh{\chi}_{31}(k_1)\},\{\wh{\chi}_{32}(k_1)\},\{\wh{\chi}_{33}(k_1)\},\{\wh{\chi}_{34}(k_1)\}$ \\
   \hline
  \multirow{2}{*}{$g_{13}(t_1)$} & $\ell|(q-1)$ &    $\{\wh{\chi}_{35}(t_1), \wh{\chi}_{36}(t_1), \wh{\chi}_{37}(t_1), \wh{\chi}_{38}(t_1)\}$\\
   & $\ell|(q+1)$ &    $\{\wh{\chi}_{35}(t_1), \wh{\chi}_{36}(t_1)-\wh{\chi}_{35}(t_1), \wh{\chi}_{37}(t_1)- \wh{\chi}_{35}(t_1), \wh{\chi}_{38}(t_1)-\wh{\chi}_{37}(t_1)-\wh{\chi}_{36}(t_1)+\wh{\chi}_{35}(t_1)\}$ \\
   $GU_2(q)\times Sp_2(q)$ & $\ell\not|(q^2-1)$ &    $\{\wh{\chi}_{35}(t_1)\}, \{\wh{\chi}_{36}(t_1)\}, \{\wh{\chi}_{37}(t_1)\}, \{\wh{\chi}_{38}(t_1)\}$\\
   \hline
  \multirow{2}{*}{$g_{16}(k_1, k_2)$} & $\ell|(q-1)$ &   $\{\wh{\chi}_{39}(k_1, k_2), \wh{\chi}_{40}(k_1, k_2)\}$\\
   & $\ell|(q+1)$ &    $\{\wh{\chi}_{39}(k_1, k_2), \wh{\chi}_{40}(k_1, k_2)-\wh{\chi}_{39}(k_1, k_2)\}$\\
   $C_{q-1}\times GL_2(q)$ & $\ell\not|(q^2-1)$ &    $\{\wh{\chi}_{39}(k_1, k_2)\}, \{\wh{\chi}_{40}(k_1, k_2)\}$\\
   \hline
   \multirow{2}{*}{$g_{17}(k_1, k_2)$} & $\ell|(q-1)$ &    $\{\wh{\chi}_{41}(k_1, k_2), \wh{\chi}_{42}(k_1, k_2)\}$\\
   & $\ell|(q+1)$ &    $\{\wh{\chi}_{41}(k_1, k_2), \wh{\chi}_{42}(k_1, k_2)-\wh{\chi}_{41}(k_1, k_2)\}$\\
  $\left(C_{q-1}\right)^2\times Sp_2(q)$ & $\ell\not|(q^2-1)$ &   $\{\wh{\chi}_{41}(k_1, k_2)\}, \{\wh{\chi}_{42}(k_1, k_2)\}$\\
   \hline
   \multirow{2}{*}{$g_{18}(k_1, t_1)$} & $\ell|(q-1)$ &   $\{\wh{\chi}_{43}(k_1, t_1), \wh{\chi}_{44}(k_1, t_1)\}$\\
   & $\ell|(q+1)$ &    $\{\wh{\chi}_{43}(k_1, t_1), \wh{\chi}_{44}(k_1, t_1)-\wh{\chi}_{43}(k_1, t_1)\}$\\
   $C_{q+1}\times GL_2(q)$& $\ell\not|(q^2-1)$ &    $\{\wh{\chi}_{43}(k_1, t_1)\}, \{\wh{\chi}_{44}(k_1, t_1)\}$\\
   \hline
   \multirow{2}{*}{$g_{19}(s)$} & $\ell|(q-1)$ &    $\{\wh{\chi}_{45}(s), \wh{\chi}_{46}(s)\}$\\
   & $\ell|(q+1)$ &    $\{\wh{\chi}_{45}(s), \wh{\chi}_{46}(s)-\wh{\chi}_{45}(s)\}$\\
   $C_{q^2-1}\times Sp_2(q)$& $\ell\not|(q^2-1)$ &    $\{\wh{\chi}_{45}(s)\}, \{\wh{\chi}_{46}(s)\}$\\
  \hline
   \multirow{2}{*}{$g_{20}(k_1, t_1)$} & $\ell|(q-1)$ &   $\{\wh{\chi}_{47}(k_1, t_1), \wh{\chi}_{48}(k_1, t_1)\}$\\
   & $\ell|(q+1)$ &   $\{\wh{\chi}_{47}(k_1, t_1), \wh{\chi}_{48}(k_1, t_1)-\wh{\chi}_{47}(k_1, t_1)\}$\\
  $C_{q-1}\times C_{q+1}\times Sp_2(q)$ & $\ell\not|(q^2-1)$ &   $\{\wh{\chi}_{47}(k_1, t_1)\}, \{\wh{\chi}_{48}(k_1, t_1)\}$\\
   \hline
   \multirow{2}{*}{$g_{21}(t_1, k_1)$} & $\ell|(q-1)$ &    $\{\wh{\chi}_{49}(t_1, k_1), \wh{\chi}_{50}(t_1, k_1)\}$\\
   & $\ell|(q+1)$ &    $\{\wh{\chi}_{49}(t_1, k_1), \wh{\chi}_{50}(t_1, k_1)-\wh{\chi}_{49}(t_1, k_1)\}$\\
   $C_{q-1}\times GU_2(q)$& $\ell\not|(q^2-1)$  &  $\{\wh{\chi}_{49}(t_1, k_1)\}, \{\wh{\chi}_{50}(t_1, k_1)\}$\\
  \hline
   \multirow{2}{*}{$g_{22}(t_1, t_2)$} & $\ell|(q-1)$  &  $\{\wh{\chi}_{51}(t_1, t_2), \wh{\chi}_{52}(t_1, t_2)\}$\\
   & $\ell|(q+1)$ &   $\{\wh{\chi}_{51}(t_1, t_2), \wh{\chi}_{52}(t_1, t_2)-\wh{\chi}_{51}(t_1, t_2)\}$\\
   $C_{q+1}\times GU_2(q)$& $\ell\not|(q^2-1)$ &   $\{\wh{\chi}_{51}(t_1, t_2)\}, \{ \wh{\chi}_{52}(t_1, t_2)\}$\\
     \hline
   \multirow{2}{*}{$g_{23}(t_1, t_2)$} & $\ell|(q-1)$ &    $\{\wh{\chi}_{53}(t_1, t_2), \wh{\chi}_{54}(t_1, t_2)\}$\\
   & $\ell|(q+1)$ &    $\{\wh{\chi}_{53}(t_1, t_2), \wh{\chi}_{54}(t_1, t_2)-\wh{\chi}_{53}(t_1, t_2)\}$\\
  $\left(C_{q+1}\right)^2\times Sp_2(q)$ & $\ell\not|(q^2-1)$ &    $\{\wh{\chi}_{53}(t_1, t_2)\}, \{ \wh{\chi}_{54}(t_1, t_2)\}$\\
   \hline
   \multirow{2}{*}{$g_{24}(u)$} & $\ell|(q-1)$ &   $\{\wh{\chi}_{55}(u), \wh{\chi}_{56}(u)\}$\\
   & $\ell|(q+1)$ &    $\{\wh{\chi}_{55}(u), \wh{\chi}_{56}(u)-\wh{\chi}_{55}(u)\}$\\
  $C_{q^2+1}\times Sp_2(q)$ & $\ell\not|(q^2-1)$ &    $\{\wh{\chi}_{55}(u)\}, \{ \wh{\chi}_{56}(u)\}$\\
   \hline
  $g_{25}(k_1,k_2,k_3)$& \multirow{2}{*}{all $\ell\neq 2$} &  \multirow{2}{*}{$\{\wh{\chi}_{57}(k_1,k_2,k_3)\}$}\\
   $\left(C_{q-1}\right)^3$ & & \\
   \hline
   $g_{26}(k_1,k_2,t_1)$ & \multirow{2}{*}{all $\ell\neq 2$} &  \multirow{2}{*}{$\{\wh{\chi}_{58}(k_1,k_2,t_1)\}$}\\
   $\left(C_{q-1}\right)^2\times C_{q+1}$ & & \\
   \hline
   $g_{27}(s, k_1)$ & \multirow{2}{*}{all $\ell\neq 2$} & \multirow{2}{*}{$\{\wh{\chi}_{59}(s, k_1)\}$}\\
   $C_{q-1}\times C_{q^2-1}$& &\\
   \hline
   $g_{28}(k_1, t_1, t_2)$ & \multirow{2}{*}{all $\ell\neq 2$}  & \multirow{2}{*}{$\{\wh{\chi}_{60}(k_1, t_1, t_2)\}$}\\
   $C_{q-1}\times\left(C_{q+1}\right)^2$& &\\
   \hline
   $g_{29}(s, t_1)$ & \multirow{2}{*}{all $\ell\neq 2$}  & \multirow{2}{*}{$\{\wh{\chi}_{61}(s, t_1)\}$} \\
   $C_{q+1}\times C_{q^2-1}$& &\\
   \hline
   $g_{30}(k_1, u)$ & \multirow{2}{*}{all $\ell\neq 2$} &  \multirow{2}{*}{$\{\wh{\chi}_{62}(k_1, u)\}$}\\
   $C_{q-1}\times C_{q^2+1}$& &\\
   \hline
   $g_{31}(v)$ & \multirow{2}{*}{all $\ell\neq 2$} &  \multirow{2}{*}{$\{\wh{\chi}_{63}(v)\}$}\\
   $C_{q^3-1}$& &\\
   \hline
   $g_{32}(t_1, t_2, t_3)$ & \multirow{2}{*}{all $\ell\neq 2$} &  \multirow{2}{*}{$\{\wh{\chi}_{64}(t_1, t_2, t_3)\}$}\\
   $\left(C_{q+1}\right)^3$& &\\
   \hline
    $g_{33}(u, t_1)$ & \multirow{2}{*}{all $\ell\neq 2$} &  \multirow{2}{*}{$\{\wh{\chi}_{65}(u, t_1)\}$}\\
   $C_{q+1}\times C_{q^2+1}$& &\\
   \hline
   $g_{34}(w)$ & \multirow{2}{*}{all $\ell\neq 2$} &  \multirow{2}{*}{$\{\wh{\chi}_{66}(w)\}$}\\
   $C_{q^3+1}$& &\\
   \hline
\end{tabular}
\end{table} 
\clearpage
\bibliographystyle{plain}
\bibliography{researchreferences}
\end{document}